\documentclass{amsart}
\usepackage{amsmath,amssymb,amsthm}

\usepackage{microtype}

\usepackage{mathrsfs}
\usepackage{datetime}
\usepackage{enumitem}
\usepackage{xfrac}
\usepackage{subcaption}
\usepackage{varwidth}

\usepackage[margin=1in]{geometry}

\usepackage[colorlinks,linkcolor = blue,citecolor = blue]{hyperref}

\usepackage[disable]{todonotes}

\newtheorem{theorem}{Theorem}
\newtheorem{proposition}{Proposition}[section]
\newtheorem{corollary}{Corollary}
\newtheorem{lemma}{Lemma}[section]
\theoremstyle{definition}
\newtheorem{definition}{Definition}[section]
\newtheorem{remark}{Remark}
\theoremstyle{definition}

\newtheorem{assumption}{Assumption}

\numberwithin{equation}{section}

\DeclareMathOperator{\R}{\mathbb{R}}

\let\S\relax
\DeclareMathOperator{\S}{S}

\DeclareMathOperator{\bigO}{\mathcal{O}}

\newcommand{\la}{\left\langle}
\newcommand{\ra}{\right\rangle}

\newcommand{\cont}{\text{\textup{cont}}}
\newcommand{\RH}{\text{\textup{RH}}}

\newcommand{\ba}{\overset{\raisebox{0pt}[0pt][0pt]{\text{\raisebox{-.5ex}{\scriptsize$\leftharpoonup$}}}}}
\newcommand{\fa}{\overset{\raisebox{0pt}[0pt][0pt]{\text{\raisebox{-.5ex}{\scriptsize$\rightharpoonup$}}}}}

\def\eq#1{(\ref{#1})}

\def\({\left(\begin{array}{cccccc}}
\def\){\end{array}\right)}

\def\eq#1{(\ref{#1})}

\def\({\left(\begin{array}{cccccc}}
\def\){\end{array}\right)}

\def\bes{\begin{eqnarray}}
\def\ees{\end{eqnarray}}

\def\bel{\begin{equation}\label}

\newcommand{\beq}{\begin{equation}}
\newcommand{\eeq}{\end{equation}}
\newcommand{\bea}{\begin{eqnarray}}
\newcommand{\eea}{\end{eqnarray}}
\newcommand{\beann}{\begin{eqnarray*}}
\newcommand{\eeann}{\end{eqnarray*}}
\newcommand{\si}{\ensuremath{\sigma}}

\newcommand{\eps}{\varepsilon}
\newcommand{\Sweak}{\mathcal{S}_{\text{weak}}}

\newcommand{\Nu}{\mathcal{V}}
\newcommand{\ds}{\displaystyle}

\newcommand{\bp}{\begin{proof}}
\newcommand{\ep}{\end{proof}}

\DeclareMathOperator{\sgn}{sgn}

\usepackage[american]{babel}

\usepackage{import}
\usepackage{float}
\usepackage{xifthen}
\usepackage{pdfpages}
\usepackage{transparent}

\newcommand{\norm}[1]{\left\lVert#1\right\rVert}

\usepackage{xcolor}
\usepackage{comment}
\usepackage[normalem]{ulem}
\usepackage{bbm}
\usepackage{commath}
\usepackage{mathtools}

\usepackage{graphicx,marvosym}
\usepackage{cleveref}
\usepackage{mdframed}

\title[Hölder Stability for Conservation Laws]{Solutions to conservation laws are Hölder-stable in $L^2$ in the weak-BV setting}

\author[Geng Chen]{Geng Chen}
 \address[Geng Chen]{\newline Department of Mathematics, \newline University of Kansas, Lawrence, KS 66045, USA}
 \email{gengchen@ku.edu}

 \author[Cooper Faile]{Cooper Faile}
 \address[Cooper Faile]{\newline Department of Mathematics, \newline The University of Texas at Austin, Austin, TX 78712, USA}
 \email{jcfaile@utexas.edu}

 \author[Sam G. Krupa]{Sam G.  Krupa}
\address[Sam G. Krupa]{\newline Département de mathématiques et applications \newline École normale supérieure, Université PSL, CNRS\newline 45 rue d'Ulm - F 75230 PARIS cedex 05 \newline France}
\email{sam.krupa@ens.fr}

 \date{\today}
\subjclass[2020]{Primary 35L65; Secondary 76N15, 35L45, 35B35.}
 \keywords{compressible Euler system, isothermal gas dynamics, uniqueness, stability, Hölder, relative entropy, conservation law.}

 \thanks{\textbf{Acknowledgment.} 
 The first author is partially supported by National Science Foundation at grant DMS-2306258, and a SQuaRE at the American Institute of Mathematics. The first author thanks the helpful discussion with Alberto Bressan and colleagues in SQuaRE.
The second author is partially supported by National Science Foundation grant DMS-1840314.
The work of the third author is funded by the European Union through the project ``Quantitative Stability and Regularity of Large Data for Conservation Laws.'' Views and opinions expressed are however those of the author(s) only and do not necessarily reflect those of the European Union or European Research Executive Agency (REA). Neither the European Union nor the granting authority can be held responsible for them.}

\mdfdefinestyle{MyFrame}{%
    linecolor=black,
    outerlinewidth=2pt,
    innertopmargin=4pt,
    innerbottommargin=4pt,
    innerrightmargin=4pt,
    innerleftmargin=4pt,
        leftmargin = 4pt,
        rightmargin = 4pt
        }
\begin{document}

\begin{abstract}
We consider hyperbolic systems of conservation laws in one spatial dimension. For any limit of front tracking solutions $v$, and for a general weak solution $u\in L^\infty$ with no BV assumption, we prove the following Hölder-type stability estimate in $L^2$:
\begin{align*}
\norm{u(\cdot,\tau)-v(\cdot,\tau)}_{L^2} \leq K 
\sqrt{\norm{u( \cdot,0)-v( \cdot,0)}_{L^2}},
\end{align*} 
for all $\tau$ without smallness and for a universal constant $K$. Our result holds for all limits of front tracking solutions $v$ with BV bound, either for general systems with small-BV data, or for special systems (isothermal Euler, Temple-class systems) with large-BV data. Our results apply to physical systems such as isentropic Euler. The stability estimate is completely independent of the BV norm of the potentially very wild solution $u$. We use the $L^2$ theory of shock stability modulo an artificial shift of position (Vasseur [{\em Handbook of Differential Equations: Evolutionary Equations}, 4:323 -- 376, 2008])  but our stability results \emph{do not depend on an artificial shift}. Moreover, we give the first result within this framework which can show \emph{uniqueness} of some solutions with large $L^\infty$ and \emph{infinite BV} initial data. We apply these techniques to isothermal Euler.
 \end{abstract}

\maketitle

 \tableofcontents

\section{Introduction}

In this paper, we will consider $n\times n$ hyperbolic systems of conservation laws in one space dimension:
\begin{equation} 
\begin{cases}
\partial_t u + \partial_x [f(u)] = 0,\quad t > 0,\quad x \in \R \\
u(x,0)=u^0(x), \quad  x \in \R
\end{cases}
\label{system} \end{equation}
where $u = (u_1,\dots,u_n) \in \mathcal{V}\subset \R^n$ are the unknowns and $u^{0}\colon\mathbb{R}\to\mathcal{V}$ is the initial data. 
The state space $\mathcal{V}$ is assumed to be {a bounded and convex open set.}
We take $f = (f_1,\dots,f_n)$ to be $C^4$ on $\mathcal{V}$ {up to the boundary}.  
We also assume there exists a strictly convex entropy $\eta \in C^3(\bar{\mathcal{V}})$ and entropy-flux $q\in  C^3(\bar{\mathcal{V}})$ satisfying
\begin{equation}
    q' = \eta'f'
\end{equation}
within $\mathcal{V}$. 
We restrict our study to entropic solutions of equation~\eqref{system}; that is, solutions also distributionally satisfying 
\begin{equation}
    \partial_t \eta(u) + \partial_x q(u) \leq 0,\quad t > 0,\quad x \in \R \label{eq:entropy-eqn}
\end{equation}
for the fixed entropy/entropy-flux pair $(\eta,q)$. 

More precisely, we require that for all
$\phi \in C_0^{\infty}\!\bigl(\mathbb{R}\times[0,\infty)\bigr)$ with $\phi\ge 0$,

\begin{align}
\int_{0}^{\infty}\!\!\int_{-\infty}^{\infty}
\bigl[\,\phi_t(x,t)\,\eta\bigl(u(x,t)\bigr)
      +\phi_x(x,t)\,q\bigl(u(x,t)\bigr)\bigr]\,
\mathrm{d}x\,\mathrm{d}t
\;+\;
\int_{-\infty}^{\infty} 
\phi(x,0)\,\eta\!\bigl(u^{0}(x)\bigr)\,
\mathrm{d}x
\;\ge\;0,
\end{align}

where $u^{0}\colon\mathbb{R}\to\mathbb{R}^n$ is the initial data of the solution $u$.

Some of our results we state for general $n\times n$ systems, but we focus also on the $2\times 2$ case. In particular, we consider $2\times 2$ Temple-class systems as well as the isothermal gas. The isothermal gas (in Lagrangian coordinates) can be written as
\begin{equation}
\begin{cases}\label{isothermal}
   \partial_t \tau - \partial_x w = 0,  \\[6pt]
  \partial_t w + \partial_x p(\tau) = 0,
\end{cases}
\end{equation}

where $w$ is the velocity, $\rho$ is the density, $\tau = \sfrac{1}{\rho}$ is the specific volume, and $p(\tau)=\sfrac{1}{\tau}$. We remark that \eqref{isothermal} is equivalent to the Eulerian formulation of isothermal Euler for weak solutions bounded away from vacuum \cite{wagnergasdynamics}. The isothermal system is
endowed with the physical entropy, entropy-flux pair
\begin{align}\label{isothermal_eta_q}
  \eta(\tau,w) = \dfrac{w^2}{2} - \log(\tau), \qquad
  q(\tau,w) = p(\tau)w = \dfrac{w}{\tau}.
\end{align}


We will consider the following classes of solutions.

 Fix an open, convex subset $\mathcal{W}$ of $\mathcal{V}$ such that $\bar{\mathcal{W}}\subset\mathcal{V}$, fix any $d \in \mathcal{W}$, and fix a small $\epsilon>0$. Then define

\begin{align}
    \mathcal{S}_{\textrm{BV},\epsilon}^0\coloneqq\Bigg\{\text{functions } u^0\colon \mathbb{R} \to \mathcal{W} \Bigg|
   \|u^0 - d\|_{L^\infty(\mathbb{R})} \leq \epsilon
  \quad \text{and} \quad \|u^0\|_{BV(\mathbb{R})} \leq \epsilon\Bigg\}.\label{small_BV_data}
  \end{align}

  In the seminal work \cite{MR0194770}, Glimm showed that for $u^0\in\mathcal{S}_{\textrm{BV},\epsilon}^0$, there exists a solution $u$ of \eqref{system} with $u(\cdot,0) = u^0$. Along with Glimm’s method, often referred to as the Glimm scheme or random choice method, there are two additional methods which prove the small BV existence for general hyperbolic conservation laws: the front tracking scheme (see \cite{MR3443431,dafermos_big_book,MR1816648}) and the vanishing viscosity method \cite{MR2150387}.

We now define the large-BV solutions we will consider.

 Fix an open, convex subset $\mathcal{W}$ of $\mathcal{V}$ such that $\bar{\mathcal{W}}\subset\mathcal{V}$. Fix two states $u_{-\infty}$, $u_{+\infty}\in\mathcal{W}$ and $M>0$.

Then, we define 
\begin{align}
    \mathcal{S}_{\textrm{BV},M}^0\coloneqq\mathrm{cl}\Bigg(\Bigg\{\text{functions } u^0\colon \mathbb{R} \to \mathcal{W} \Bigg|
   \left\{
  \begin{aligned}
    &u^0 - u_{-\infty} \in L^1((-\infty, 0]) \\
    &u^0 - u_{+\infty} \in L^1([0,+\infty))
  \end{aligned}
  \right.
  \quad \text{and} \quad \|u^0\|_{BV(\mathbb{R})} \leq M\Bigg\}\Bigg),\label{large_BV_data}
  \end{align}
  where the closure is taken in the $L^1$ topology. 

Among limited large BV global existence results for systems of hyperbolic conservation laws, for initial data in the set $\mathcal{S}_{\textrm{BV},M}^0$, existence of large BV solutions for the Temple systems as well as the isothermal Euler system \eqref{isothermal} has been solved by Temple \cite{MR716850} and Nishida \cite{MR236526}, respectively. More related to our paper, Colombo-Risebro \cite{ColomboRisebro} proved the existence of solutions for these systems using a large-BV wave front-tracking algorithm. 

For a general class of hyperbolic systems \eqref{system}, the solutions with small-BV initial data \eqref{small_BV_data} form a semigroup $\mathcal{S}$, and $\mathcal{S}$ is $L^1$-Lipschitz continuous with constant $L$, i.e.,
  \begin{align}
  \| \mathcal{S}_{t''}u'' - \mathcal{S}_{t'}u' \|_{L^1} \leq L \cdot \left( \| u'' - u' \|_{L^1} + |t'' - t'| \right). \label{L1_estimate}
  \end{align}
Similarly, for Temple systems and the isothermal Euler \eqref{isothermal}, the class of solutions with large-BV initial data \eqref{large_BV_data}, constructed as limits of front tracking solutions by Colombo-Risebro \cite{ColomboRisebro}, also form a semigroup $\mathcal{S}$ and verify the same estimate \eqref{L1_estimate}.

We restrict our study to the solutions verifying what is often referred to in the literature as the ``Strong Trace Property.''

\begin{definition}(Strong Trace Property \cite{Leger2011})\label{strong_trace_prop}
Let $u \in L^\infty(\mathbb{R} \times [0,\infty))$. We say that $u$ verifies the \emph{strong trace property} if for any Lipschitz continuous curve $h\colon [0,\infty)\to\mathbb{R}$, there exist two bounded functions $u_-, u_+ \in L^\infty([0,\infty))$ such that, for any $T > 0$,
\[
\lim_{n \to \infty} \int_0^T \sup_{y \in (0,1/n)} |u(h(t) + y,t) - u_+(t)| \, dt
= \lim_{n \to \infty} \int_0^T \sup_{y \in (-1/n, 0)} |u(h(t) + y,t) - u_-(t)| \, dt = 0.
\]
\end{definition}

For convenience, we will use the notation $u_+(t) \coloneqq u(h(t)+,t)$, and $u_-(t) \coloneqq u(h(t)-,t)$.

We can then define the following large space of ``rough'' solutions that we consider in the paper:
\begin{multline}
\mathcal{S}_{\text{weak}} \coloneqq 
\\
\Big\{ u \in L^\infty(\mathbb{R}\times[0,T) : \mathcal{V}) \text{ weak solution to } \eqref{system}\text{ or } \eqref{isothermal}, \text{ as well as }\eqref{eq:entropy-eqn}, \text{ verifying \Cref{strong_trace_prop}} \Big\},\label{large_sol}
\end{multline}
for some $T>0$. \textbf{N.B.:} this space has no smallness condition.

In this paper, we will consider two major stability questions:

\newpage

(Problem $\epsilon$-BV)
\begin{mdframed}
\begin{align*}
    \parbox[t]{14cm}{For any system \eqref{system} with two conserved quantities verifying Assumption \ref{assum} (see \Cref{sec:prelim}), we consider the small-BV solutions with initial data \eqref{small_BV_data}, constructed as e.g. limits of front tracking solutions. Our major question is, are these solutions \emph{quantitatively} stable in the class of rough solutions \eqref{large_sol}?}
\end{align*}
\vspace{.034in}
\end{mdframed}

and

(Problem Isothermal)
\begin{mdframed}
\begin{align*}
    \parbox[t]{14cm}{For the isothermal Euler system \eqref{isothermal}, consider the large-BV solutions with initial data \eqref{large_BV_data}, constructed following Colombo-Risebro \cite{ColomboRisebro}. Our major question is, are these solutions \emph{quantitatively} stable in the class of rough solutions \eqref{large_sol}?}
\end{align*}
\vspace{.03in}
\end{mdframed}

\vspace{.1in}

We will often refer to (Problem $\epsilon$-BV) and (Problem Isothermal) as shorthand for referencing the small-BV and large-BV cases in our paper, including the small and large-BV front tracking schemes, respectively.

Assumption \ref{assum} are  verified by all systems of physical interest, including isothermal Euler \eqref{isothermal}, isentropic Euler, and the $3\times3$ ``full Euler'' system (see \cite{MR4487515,Leger2011}). We refer to \Cref{sec:prelim} below for details.

We can now state our Main Theorem.

\begin{theorem}[Main theorem -- $L^2$-Hölder stability for conservation laws]\label{main_theorem}
\hfil

 Consider (Problem $\epsilon$-BV) or (Problem Isothermal). Fix $R,T>0$. Then, consider a solution $u\in \mathcal{S}_{\text{weak}}$ to either \eqref{system} or \eqref{isothermal}. For (Problem $\epsilon$-BV), consider initial data $v^0\in\mathcal{S}_{\textrm{BV},\epsilon}^0$ for $\epsilon>0$ sufficiently small. For (Problem Isothermal), fix $M>0$ and consider initial data $v^0\in\mathcal{S}_{\textrm{BV},M}^0$. Then, for (Problem $\epsilon$-BV) consider the classical small-BV solution $v$ with initial data $v^0$ (from the Glimm scheme, front tracking method, or vanishing viscosity). For (Problem Isothermal), consider the large-BV solution $v$ with initial data $v^0$ constructed according to Colombo-Risebro \cite{ColomboRisebro}. Then, we have the following stability estimate:
\begin{align}\label{main_estimate}
     \norm{u(\cdot,\tau)-v(\cdot,\tau)}_{L^2((-R,R))} \leq K \sqrt{\norm{u( \cdot,0)-v( \cdot,0)}_{L^2((-R-s\tau,R+s\tau))}},
\end{align} 
for all $\tau\in[0,T]$, for a universal constant $K$ depending on $R$, and where $s>0$ is the speed of information (see \eqref{speed_of_info}, below).
\end{theorem}

Our estimates have absolutely no dependence on the BV norm of the potentially very ``wild'' solution $u\in \mathcal{S}_{\text{weak}}$. The solution $u$ may have infinite BV or large $L^\infty$ norm, even when we consider small-BV solutions $v$ and (Problem $\epsilon$-BV).

We make some remarks on our Main Theorem:
\begin{itemize}
\item From our result, one immediately recovers the \emph{qualitative stability} result \cite[Theorem 1.3]{MR4487515}. Our \Cref{main_theorem} links the Lipschitz-type stability estimates of the $L^1$ BV-based theory (e.g. \eqref{L1_estimate}) with the infinite-BV $L^2$-based theory which until now has only given qualitative stability results. 
\item For many years, $L^2$ was thought to not be the correct space to measure distances between solutions to conservation laws, due to the presence of shocks. The difficulty is illustrated by our example below, which shows tremendous growth in $L^2$ due to a tiny perturbation of a shock (see \Cref{sec:sharp}). Even recent advances in the $L^2$ stability theory cannot give \emph{quantitative} stability estimates. In this paper, we show that in fact the $L^2$ framework can give quantitative estimates, and for a very large class of solutions.
\item Notice that $\mathcal{S}_{\text{weak}}$ may include solutions that they don't necessarily arise from front tracking.
\item As a consequence of our results, we again recover the fact that the Glimm scheme solutions, front tracking solutions, and vanishing viscosity solutions all coincide (because we can choose any of these solutions as our ``rough'' solution within our estimate \eqref{main_estimate}).
\item Our techniques also apply to large-BV solutions for $2\times 2$ Temple-class systems verifying Assumption \ref{assum} and a wave separation condition which is verified e.g. by physical systems (cf. \Cref{sec:wave_non_interact}). For simplicity, in the large BV regime, we focus on isothermal Euler and make some remarks throughout about applications to Temple systems.
\end{itemize}

\subsection{Uniqueness with infinite-BV initial data}

As an application of our techniques, we have the following Corollary of our Main Theorem.
To fix notation, the function $K\colon[0,\infty) \to [0,\infty)$ is the coefficent from the stability estimate~\eqref{main_estimate} as a function of the BV norm of $v$.
Fixing $R > 0$ and a compact convex set $\mathcal{K}\subset\mathcal{V}$ we define the following class of initial data 
\begin{equation} \label{eq:def-class-X}\begin{aligned}
        X := \Big\{ v \in L^\infty((-R,R); \mathcal{K}) \hspace{.04in}\Big| & \text{ there exists} \left\{v_n\right\}_{n=1}^\infty\subset BV((-R,R);\mathcal{K}) \text{ such that } \\
        &\lim_{n\to\infty} K\left(\norm{v_n}_{BV((-R,R))}\right) 
        \norm{v-v_n}^{1/2}_{L^2((-R,R))} = 0 \Big\}.
     \end{aligned}
\end{equation}
\begin{corollary}[Uniqueness of (some) solutions with large $L^\infty$ and infinite BV]\label{cor:uniqueness} 
\hfill

Consider (Problem Isothermal). If two solutions $u_1,\ u_2 \in \Sweak$ satisfy $u_1(\cdot,0)|_{(-R,R)} = u_2(\cdot,0)|_{(-R,R)} \in X$ then they coincide a.e. in the interval $[-R+s\tau, R-s\tau]$ for all $\tau < R/s$,
where $s>0$ is the speed of information.
The set $X$ satisfies $ BV((-R,R);\mathcal{K}) \subsetneq X \subset L^\infty((-R,R);\mathcal{K})$.
In fact, there exists data in $X$ which has everywhere locally infinite BV, i.e. the BV-norm is infinite on every non-empty open set.
\end{corollary}

\vspace{.1in}
An example of initial data in $X \setminus BV((-R,R); \mathcal{K})$ is given by quickly oscillating between states $b_1,\ b_2 \in \mathcal{K},$
\begin{equation} \label{eq:ex-unique}
    W(x) = b_1 + \frac{b_2 - b_1}{2} \left(\sin\left(\Gamma(|x|^{-1})\right) + 1 \right),
\end{equation}
where $ \Gamma\colon[0,\infty ) \to [0,\infty ) $ is a suitabily selected function increasing to infinity. In the proof of \Cref{cor:uniqueness} we then take a convex combination of data, similar to~\eqref{eq:ex-unique}, to exhibit a $v\in X$ for which $\norm{v}_{BV((a,b))} = \infty$ for all $(a,b) \subset (-R,R)$.

For a $2\times2$ strictly hyperbolic, genuinely nonlinear system, e.g. any system verifying our Assumption \ref{assum}, we have existence of entropy-weak solutions for initial data with sufficiently small $L^\infty$ norm, with no restriction on the BV norm which may be infinite. See the work of Glimm-Lax \cite{MR0265767}, as well as the more recent paper by Bianchini-Colombo-Monti \cite{MR2737438} and in particular Theorem 1.1 in \cite{MR2737438}.

For $2\times 2$ systems, we also have existence for general $L^\infty$ initial data by compensated compactness (see e.g. \cite{MR506997,MR584398,MR725524,MR454375,MR922139,MR1284790,MR1383202} as well as \cite[Section 17.9]{dafermos_big_book}). However, for these solutions it remains an open question as to whether or not the Strong Trace Property (\Cref{strong_trace_prop}) shall hold. For recent work on trying to find the Strong Trace Property for these solutions, see \cite{golding}. On the other hand, at least if the condition \eqref{eq:entropy-eqn} is not required, the Strong Trace Property may be lost even for the $p$-system \cite{2024arXiv240307784K}.

Remark that the Glimm-Lax solutions instantaneously enter $\text{BV}_{\text{loc}}$ and thus verify the Strong Trace Property. This implies that the Glimm-Lax solutions will \emph{always live in our set} $\mathcal{S}_{\text{weak}}$.

For additional work on solutions with large but \emph{finite} total variation, see \cite{MR1375345,MR1828320,MR1883740,MR2091511}.

\vspace{.08in}

Our uniqueness result (\Cref{cor:uniqueness}) compares to the recent and very interesting work of Bressan-Marconi-Vaidya \cite{2025arXiv250500420B}, where they show uniqueness and stability for a subset of the infinite-BV, small $L^\infty$ solutions arising from the Glimm-Lax theory and with certain enhanced decay estimates assumed to hold a priori.

In our result for isothermal Euler (\Cref{cor:uniqueness}), we can in fact give the uniqueness of certain large $L^\infty$, infinite-BV solutions. 

The initial data we can consider in \Cref{cor:uniqueness} is closely related to the sets of initial data $(\tilde P_\alpha)$ considered in \cite{2025arXiv250500420B}. In particular, see \Cref{lem:X-G-inc} below. The sets $(\tilde P_\alpha)$ contain initial data which may possess infinite total variation, but almost all of this variation is localized on finitely many open intervals whose combined length is small. The sets $(\tilde P_\alpha)$ do not contain data that is locally infinite-BV everywhere.  In  \cite[Theorem 1.2]{2025arXiv250500420B} it is shown that for initial data in this set $(\tilde P_\alpha)$, solutions exist with a fast decay rate satisfying their uniqueness theorem  \cite[Theorem 1.1]{2025arXiv250500420B}. On the other hand, our techniques do not require an a priori assumption for a fast decay rate. Some subset of the Glimm-Lax solutions have initial data in our set $X$ including solutions with locally infinite BV everywhere data (see \eqref{eq:def-class-X}) and in fact we show that these solutions are unique in $\mathcal{S}_{\text{weak}}$.

The class of infinite-BV solutions we can show uniqueness for is a function of how the growth of the universal constant $K$ from \eqref{main_estimate} depends on our control on the bounded variation (see the constant $M$ in  \eqref{large_BV_data}). The dependence of $K$ on $M$ is a result of the constants which appear in the front tracking scheme \cite{ColomboRisebro} (see e.g. \eqref{iso_constant}, \eqref{bound1_a_iso}, and \eqref{bound2_a_iso}, below). Thus, an interesting new line of research questions will be to develop front tracking schemes to optimize this constant $K$. 

\subsection{Sharpness of our estimate}\label{sec:sharp}

\begin{figure}
    \centering
    \begin{subfigure}[b]{0.45\textwidth}
        \includegraphics[width=\textwidth]{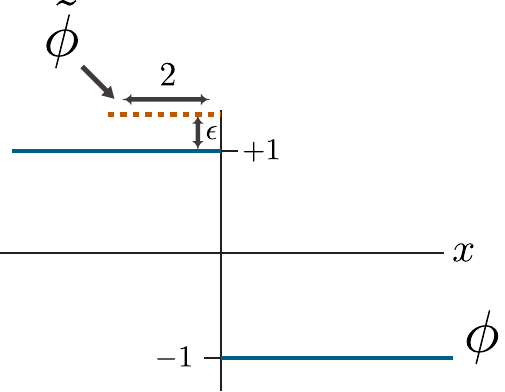}
        \caption{$t=0$}
        \label{fig:init}
    \end{subfigure}
    ~ 
    \begin{subfigure}[b]{0.45\textwidth}
        \includegraphics[width=\textwidth]{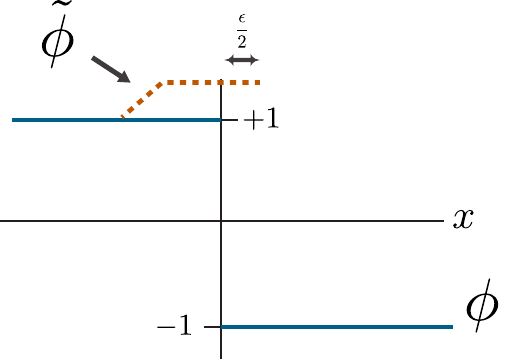}
        \caption{$t=1$}
        \label{fig:later_time}
    \end{subfigure}
    ~ 
    \caption{Two solutions $\phi$ and $\tilde\phi$ to Burgers' equation, at times $t=0$ and $t=1$. The function $\tilde\phi$ agrees with $\phi$ everywhere except where noted by the dashed line. Diagram not to scale.}\label{fig:why_holder}
\end{figure}

We now briefly show that the Hölder exponent $\frac{1}{2}$ in \Cref{main_theorem} is sharp. 

Consider Burgers', the scalar conservation law given by $u_t+(\frac{u}{2})_x=0$. Consider pure-shock initial data of the form
\begin{align}
        \phi(x,0) =
  \begin{cases}
    +1      & \quad \text{if } x<0\\
    -1  & \quad \text{if } x>0.
  \end{cases}
\end{align}
Remark that the classical solution $\phi$ with this initial data will be a steady state. On the other hand, for $\epsilon>0$ sufficiently small consider an  $\epsilon$-perturbation of this initial data given by a function $\tilde\phi(\cdot,0)$ (see \Cref{fig:init}). Due to the Rankine-Hugoniot jump condition, the perturbation induces a velocity of $\frac{\epsilon}{2}$ in the shock discontinuity. At the initial time, we have
\begin{align}
\norm{\phi(\cdot,0)-\tilde\phi(\cdot,0)}_{L^2(\mathbb{R})}=\sqrt{2}\epsilon.
\end{align}

However, at time $t=1$ (see \Cref{fig:later_time}) we have
\begin{align}
\norm{\phi(\cdot,1)-\tilde\phi(\cdot,1)}_{L^2(\mathbb{R})}\geq\sqrt{\frac{\epsilon}{2}}.
\end{align}

Thus, we cannot expect better than Hölder-$\frac{1}{2}$ stability in general.

\subsection{BV, infinite-BV, and the Strong Trace Property (\Cref{strong_trace_prop})}
In this paper, we consider both the solutions with small bounded variation (small-BV) as well as the classical solutions with large bounded variation (large-BV). We show the Hölder stability of these solutions in the class of possibly infinite-BV, large $L^\infty$ solutions we call $\mathcal{S}_{\text{weak}}$. Due to the Strong Trace Property, infinite-BV solutions which have behavior in a fashion similar to $x\mapsto \sin(1/x)$ at each fixed time $t$ are excluded from $\mathcal{S}_{\text{weak}}$. On the other hand, solutions $u$ which behave similar to $x\mapsto \sqrt{\abs{x}}\sin(1/\abs{x})$ at each fixed time $t$ will live inside $\mathcal{S}_{\text{weak}}$. The function $\sqrt{\abs{x}}\sin(1/\abs{x})$ is also infinite-BV, but it is continuous and  the oscillations are somehow much better. Solutions like this fit into our large data theory. The function $\sqrt{\abs{x}}\sin(1/\abs{x})$  has vanishing oscillations on a small scale – this contrasts with the extreme oscillations caused by convex integration for example. It is the difference between the spaces $\text{VMO}$ and $\text{BMO}$. We have counterexamples with such extreme oscillations, including the third author's recent result on bounded variation blow up \cite{2024arXiv240307784K}. We refer the reader to \Cref{fig:function_spaces}.

\begin{figure}
    \centering
        \includegraphics[width=\textwidth]{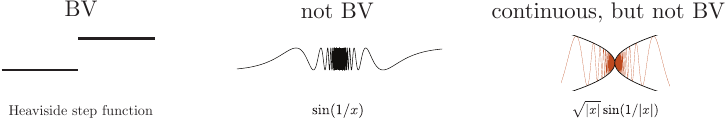}
    \caption{A solution $u$ to \eqref{system} or \eqref{isothermal} may be finite BV-norm at each fixed time, like the heaviside function. Or it may have infinite BV, and lack the Strong Trace Property, if it behaves like $\sin(1/x)$. On the other hand, it may have infinite BV while staying continuous like $\sqrt{\abs{x}}\sin(1/\abs{x})$ (and thus have the Strong Trace Property!).}\label{fig:function_spaces}
\end{figure}

\subsection{Idea of the proof and comparison with previous works}

Recently, there has been a tremendous amount of research devoted to understanding the stability of \emph{large} data solutions in $L^2$. For decades, it has been assumed that the natural space to study shock waves is $L^1$ (see e.g. \cite{MR1816648,MR1686652}). The difficulty is that in other spaces such as $L^2$, slight perturbations at a shock may cause large growth as explained in \Cref{sec:sharp} above. The way to get around this difficulty has been to actually relax the notion of stability and consider stability only \emph{up to a shift}. This approach was devised by Vasseur \cite{VASSEUR2008323}. In this methodology, we study the stability of a shock in the class of large data $L^2$ solutions by allowing the shock to move with an artificial velocity dictated by a so-called \emph{shift function} (see the early work of Leger \cite{Leger2011_original}). By introducing an artificial velocity, the new function is no longer the solution to any partial differential equation (PDE) but on the other hand the $\sqrt{\epsilon}$ error from the ``instability'' example in \Cref{sec:sharp} is prevented. 

In the front tracking scheme, piecewise-constant approximate solutions are considered. By introducing shift functions \emph{at each shock discontinuity}, the stability of front tracking solutions, and limits of front tracking solutions, can be studied. This leads to the first work, by Vasseur, the first author, and the third author, showing that limits of front tracking solutions are stable even in the class of large $L^2$ perturbations and do not experience blow up in this large class \cite{MR4487515}. This work has also lead to the recent resolution of the problem of the \emph{physical} vanishing viscosity limit going from compressible Navier-Stokes to Euler in 1-D \cite{2024arXiv240109305C}.

However, none of these $L^2$ stability results are able to explicitly \emph{quantify} the stability between the shifted solutions and the large data $L^2$ solution. Let us explain the fundamental issue. Consider an approximate front tracking solution $v$, and also an approximate front tracking solution $\psi$ \emph{with shifts}, with coinciding initial data $v(\cdot,0)=\psi(\cdot,0)$. Consider also a large data $L^2$ solution $u\in\mathcal{S}_{\text{weak}}$. By the $L^2$ stability theory described above, we know that the $L^2$ quantity $\norm{u(\cdot,t)-\psi(\cdot,t)}_{L^2}$ is stable in time in a nicely controlled way due to the use of shift functions in $\psi$. But $\psi$ is not a solution to any PDE. A quantity of interest would be $\norm{u(\cdot,t)-v(\cdot,t)}_{L^2}$, which we could study by understanding the stability of $\norm{v(\cdot,t)-\psi(\cdot,t)}_{X}$ for an appropriate normed space $X$. In this setting, $\norm{v(\cdot,t)-\psi(\cdot,t)}_{X}$ is a measure of ``how much we have to shift $\psi$'' to maintain the $\norm{u(\cdot,t)-\psi(\cdot,t)}_{L^2}$ stability in $L^2$. However, due to the shifts, the function $v$ will have wave interactions at potentially different times, and in potentially different orders (!!), than the function $\psi$. This will cause the function $v$ and $\psi$ to very quickly lose the property of being related to each other by a simple change of variables. Once this occurs, and $v$ and $\psi$ are no longer related by a change of variables, how do we measure $\norm{v(\cdot,t)-\psi(\cdot,t)}_{X}$? This is the \emph{change of variables problem} (see Section 3.4 of \cite{GiesselmannKrupa2025} as well as \cite[p.~11]{MR4487515}). 

In \cite{MR4487515}, the answer is not to solve the change of variables problem directly, but instead ``give up'' \cite[p.~11]{MR4487515} and use instead the framework for proving uniqueness (without stability) developed by the third author and Vasseur \cite{2017arXiv170905610K}. The idea is to use the good control on $\norm{u(\cdot,t)-\psi(\cdot,t)}_{L^2}$ to show that in fact $u$ must verify some condition/criteria which means it is actually unique. In the systems case, uniqueness criteria include the Tame Oscillation Condition \cite{MR1701818,MR1489317}, Bounded Variation Condition \cite{MR1757395}, and recent work shows the Liu entropy condition suffices as well \cite{MR4661213}. The work \cite{MR4487515} shows that in fact, these uniqueness results are not necessary for $2\times 2$ hyperbolic systems with a strictly convex entropy because it is shown that solutions will always verify these criteria.

In this present work, we solve the change of variables problem. We estimate precisely $\norm{v(\cdot,t)-\psi(\cdot,t)}_{X}$ and thus quantify $\norm{u(\cdot,t)-v(\cdot,t)}_{L^2}$, uniform in the limit of the front tracking approximations $v$.

Thus, with this work, we obviate the need for the earlier $L^2$ stability theories based on detecting uniqueness criteria (see \cite{MR4487515,2017arXiv170905610K,CKV2,Cheng_isothermal}). In particular, we study the stability of the large-BV front tracking solutions which are known to exist for isothermal Euler and  Temple-class systems \cite{ColomboRisebro}. In fact, in certain cases we are now able to consider infinite BV data. This is notable, because the uniqueness criteria (Tame Oscillation Condition, Bounded Variation Condition, and Liu entropy, discussed above) are classically only known to hold for BV data. The third author of the present paper recently showed a counterexample to uniqueness in the class of infinite-BV solutions which vacuously verify the Liu entropy condition \cite{2024arXiv240307784K}.

Our work has three major ingredients:
\begin{itemize}
\item \textbf{Quantified dissipation, with applications to the inviscid limit}\hspace{.1in} To control the shift, we need to understand precisely how much dissipation is caused by shifting a discontinuity. All previous constructions of shift functions (see e.g. \cite{serre_vasseur,MR4667839,MR4184662,move_entire_solution_system}) provably produce suboptimal quantities of dissipation (see the discussion after \Cref{prop:cont} and also \cite{scalar_move_entire_solution}) which would prevent our quantifying scheme from working. In this work, we produce  dissipation estimates (see \Cref{prop:small_shock_diss}), which are provably optimal as shown in the case of Burgers' equation (cf. \cite[Proposition 2.1]{MR4305935}). This new construction is quite delicate and we postpone it to \Cref{sec:proof_prop:small_shock_diss}.  We expect our \Cref{prop:small_shock_diss} to have future applications to the study of the inviscid limit in various contexts. This relates to previous work on the inviscid limit in the $L^2$ theory, e.g. the recent milestone \cite{2024arXiv240109305C}.
\item \textbf{Front tracking error estimate}\hspace{.1in} We construct new error estimates on approximate front tracking solutions when an artificial shift is introduced (see \Cref{lemma:front_tracking_error_estimate}, below). This is based on the constructions of Bressan-Colombo \cite{BressanC} and Colombo-Risebro \cite{ColomboRisebro} and holds for both small-BV front tracking constructions as well as large-BV front tracking constructions.
\item \textbf{Construction of the space-time weight}\hspace{.1in} In previous works \cite{MR4487515,CKV2,Cheng_isothermal,2024arXiv241103578C} stability of the quantity $\norm{u(\cdot,t)-\psi(\cdot,t)}_{L^2}$ is contingent on finding a suitable inhomogeneous and time-dependent weighting of the spatial domain given by a function $a(x,t)$ (see also \cite{MR3519973}). This weight is usually constructed by running a secondary front tracking scheme to build the weight itself. We give a new technique for showing the existence of such a weight $a$, which is very general and applies to both large and small-BV approximate front tracking solutions (see \Cref{prop:iso_a}, below).
\end{itemize}

\subsection{Connections to numerical analysis}\label{sec:numerical}
Our work is motivated by recent work of Jan Giesselmann and the third author which gave the first quantitative estimates in the $L^2$-type theory based on shifts \cite{GiesselmannKrupa2025}. This preliminary work \cite{GiesselmannKrupa2025} is focused on numerical a posteriori error estimates for conservation laws, using the Burgers' equation as a test case. We expect the present result (\Cref{main_theorem}) to open new directions for a large-data numerical analysis theory for 1-D systems (cf. recent work such as \cite{MR4271961} which does not apply to large-BV numerical solutions). This is especially relevant in light of work indicating that for finite difference schemes, a priori estimates on the BV norm are \emph{not} possible \cite{MR2254446}. For a fixed ``reference solution'' with data from either the sets \eqref{small_BV_data} or \eqref{large_BV_data}, playing the role of $v$ in \Cref{main_theorem}, our estimate \eqref{main_estimate} does not have any dependence on the BV-norm of the ``rough'' solution $u\in\mathcal{S}_{\text{weak}}$ which in future work we may choose to instead be a numerical solution arising from some numerical scheme. 

We remark it is an open question if the BV norm of even an \emph{exact} solution to \eqref{system} may or may not grow without bound (this is Open Problem \# 1 of Bressan \cite{MR4855159}). 

We also remark that numerical solutions can be designed to \emph{always verify the Strong Trace Property} (see \cite{GiesselmannKrupa2025,MR2404052}).



\subsection{Plan for the paper}

In \Cref{sec:prelim} we state precisely our assumptions on the system \eqref{system} for (Problem $\epsilon$-BV) and state some basic concepts. In \Cref{sec:estimates} we state the various dissipation estimates we need on shock and rarefaction waves, including the fine control we have on small shocks (\Cref{prop:small_shock_diss}). In \Cref{section:front} we introduce the front tracking algorithm technology developed by \cite{BressanC} and \cite{ColomboRisebro}, which we use extensively in our work. In \Cref{sec:weight} we give the construction of the inhomogeneous weight $a(x,t)$. Before we can do the actual construction, we first need to state various results on wave interactions. In \Cref{sec:proof_main_theorem} we prove the Main Theorem (\Cref{main_theorem}). In \Cref{sec:cor} we prove \Cref{cor:uniqueness}. Last but not least, we give the proof of our dissipation estimate on small shocks (\Cref{prop:small_shock_diss}) in \Cref{sec:proof_prop:small_shock_diss}.

\subsection{Notation}\label{sec:not}
In this paper we use the following notational conventions:
\begin{itemize}
    \item $\epsilon$ is a small constant which is used to localize our system (\eqref{system} or \eqref{isothermal}) around the point $d\in\mathcal{V}$ (see e.g. \eqref{small_BV_data}).
    \item $\nu\ll 1$ is the parameter in the front tracking approximation, which in the limit $\nu\to0$ yields exact solutions to (Problem $\epsilon$-BV) or (Problem Isothermal).
    \item A constant $K$, which may change from line to line in calculations, depends on the time interval of interest $[0,T]$, $d$, $\epsilon$, $M$ (see \eqref{large_BV_data}), the set $\mathcal{W}$ (see \eqref{small_BV_data},  \eqref{large_BV_data}), the system~\eqref{system}, and the entropy $\eta$. 

    \item The shorthand ``$a \lesssim b$'' is to be understood as ``There exists a universal constant $K$ such that $a \leq K b$.''
    \item The shorthand ``$a\sim b$'' is to be understood as ``There exists a universal constant $K$ such that $b/K \leq a \leq K b$.''
    \item The notation $a = b + \bigO(g)$ denotes $|a-b| \lesssim g$. If the function $g$ explicitly depends on $M$ or $\epsilon$, then the implicit constant $K$ will not depend on $M$ or $\epsilon$.
    \item For a function $f$, we always use $(f)_\pm$ to denote $(|f|\pm f)/2$, i.e. the positive/negative part of $f$. 
\end{itemize}

\section{Assumptions and Preliminaries}\label{sec:prelim}
\begin{assumption}\label{assum}
    Recall from the introduction that the flux function  $f=(f_1,\cdot\cdot\cdot, f_n)\in  [C^4(\bar{\mathcal{V}})]^n$.
    In addition to this, we assume the following conditions on the system:
    \begin{enumerate}[label=(\alph*)]
    \item \label{assum:hyperbolic} For any $u\in \Nu$, the matrix $f'(u)$ diagonalizable with eigenvalues verifying $\lambda_1(u)<\lambda_2(u)$ and $\lambda_{n-1}(u)<\lambda_n(u)$. 
    We denote $r_i(u)$ a unit eigenvector associated to the eigenvalue $\lambda_i(u)$ of $i = 1,\, n$.
    \vskip0.1cm
    \item  \label{assum:gnl} For any $u\in\Nu$ and $i=1,\ n$, we assume $\lambda'_i(u)\cdot r_i(u)\neq0$.
    \vskip0.1cm
    \item \label{assum:entropy-entropyflux} There exists a strictly convex function $\eta\in  C^3(\bar{\mathcal{V}})$ and a function $q\in C^3(\bar{\mathcal{V}})$ satisfying
    \begin{equation*}\label{eq:entropy}
    q'=\eta'f' \quad \mathrm{on} \quad \Nu.
    \end{equation*}
    \vskip0.1cm
    \item\label{assum:convex_left_eign} For any state $b \in \mathcal{V}$, and any left eigenvector $\ell_i$ of $f'(b)$ corresponding in $\lambda_i(b)$, $i = 1,\ n$ the function $u \mapsto \ell_i \cdot f(u)$ is either convex or concave on $\mathcal{V}$. 
    \vskip0.1cm
    \item \label{assum:bounded-char-speed} There exists $L>0$ such that $|\lambda_i(u)|\leq L$ for any $u\in \Nu$ and $i=1,n$.
    \item \label{assum:shock-curve-param}For $u_L\in \Nu$, we denote $s\mapsto  S^1_{u_L}(s)$ the $1$-shock  curve through $u_L$ defined for $s\in[0,s_{u_L})$. See \Cref{lem_hugoniot} for the definition of $s_{u_L}$ (remark that possibly $s_{u_L}=+\infty$). We choose the parametrization such that $s=|u_L-S^1_{u_L}(s)|$. Therefore, $(u_L, S^1_{u_L}(s), \sigma^1_{u_L}(s))$ is the $1$-shock with left hand  state $u_L$ and strength $s$. Similarly, we define $s\mapsto S^n_{u_R}$ to be the $n$-shock curve such that $(S^n_{u_R}(s), u_R, \sigma^n_{u_R}(s))$ is the $n$-shock with right hand state $u_R$ and strength $s$. We assume that these curves are defined globally in $\Nu$ for every $u_L\in \Nu$ and $u_R\in \Nu$.
    
    \item \label{assum:1-shock-lax-cond-greater} (for 1-shocks) If $(u_L,u_R)$ is an entropic Rankine-Hugoniot discontinuity with shock speed $\sigma$, then $\sigma>\lambda_1(u_R).$
    \item \label{assum:1-shock-lax-cond-le} (for 1-shocks) If $(u_L,u_R)$ with $u_L \in \mathcal{V}$ is an entropic Rankine-Hugoniot discontinuity with shock speed $\sigma$ verifying,
    \begin{align*}
    \sigma\leq \lambda_1(u_L),
    \end{align*}
    then $u_R$ is in the image of $S^1_{u_L}$. That is, there exists $s_{u_R}\in[0,s_{u_L})$ such that $S^1_{u_L}(s_{u_R})=u_R$ (and hence $\sigma=\sigma^1_{u_L}(s_{u_R})$).
    \item \label{assum:n-shock-lax-cond-less} (for n-shocks) If $(u_L,u_R)$ is an entropic Rankine-Hugoniot discontinuity with shock speed $\sigma$, then $\sigma<\lambda_n(u_L).$
    \item \label{assum:n-shock-lax-cond-ge} (for n-shocks) If $(u_L,u_R)$ with $u_R \in \Nu$ is an entropic Rankine-Hugoniot discontinuity with shock speed $\sigma$ verifying,
    \begin{align*}
    \sigma\geq \lambda_n(u_R),
    \end{align*}
    then $u_L$ is in the image of $S^n_{u_R}$. That is, there exists $s_{u_L}\in[0,s_{u_R})$ such that $S^n_{u_R}(s_{u_L})=u_L$ (and hence $\sigma=\sigma^n_{u_R}(s_{u_L})$).
    \item \label{assum:rel-ent-strengthens} For $u_L\in \Nu$, and  for all $s>0$,  $\ds{\frac{d}{ds}\eta(u_L | S^1_{u_L}(s))}>0$ (the shock ``strengthens" with $s$).
    Similarly, for $u_R\in \Nu$, and for all $s>0$, $\ds{\frac{d}{ds}\eta(u_R | S^n_{u_R}(s))}>0$. Moreover, for each $u_L,u_R\in \Nu$ and $s > 0$, $\frac{d}{ds}\sigma^1_{u_L}(s) < 0$ and $\frac{d}{ds}\sigma^n_{u_R}(s) > 0$.
    \end{enumerate}
\end{assumption} 

These assumptions are fairly general and apply to all systems of physical interest. They are at this point common in the $L^2$ theory (see e.g. \cite{MR4487515,MR4667839}). The first assumption ensures the strict hyperbolicity of the extremal characteristic families. The second assumption says that the extremal families are genuinely nonlinear. The third assumption is motivated by the second law of thermodynamics. The fourth assumption guarantees a contraction property for rarefaction waves measured in $L^2$ (cf. \Cref{pizza_slice_diss}), which in particular holds for the isentropic Euler system. The assumption \ref{assum:bounded-char-speed} provides a global bound on the wave speed. The assumptions \ref{assum:shock-curve-param} through \ref{assum:rel-ent-strengthens} are at this point classical assumptions in the $L^2$ theory. They are satisfied by a wide family of systems, including the $3\times 3$ ``full'' Euler system as well as the system of isentropic gas dynamics (see \cite{Leger2011}).

We note that the regularity of $f$ combined with Assumptions \ref{assum} \ref{assum:hyperbolic} and \ref{assum:gnl} implies precises asymptotics and local estimates on the extremal shock curve $S^1_u$ and $S^n_u$. 
These local improvements are needed in \Cref{sec:proof_prop:small_shock_diss} where we localize our analysis near a fixed small shock.
        
    \begin{lemma}[From \protect{\cite[p.~263-265, Theorem 8.2.1, Theorem 8.3.1, Theorem 8.4.2]{dafermos_big_book}}]\label{lem_hugoniot}
    For any fixed state $v\in \Nu$, there is an open neighborhood $U$ of $v$
    such that for each $i = 1, n$, there exist functions $s_u:U \rightarrow \R$, $\sigma_u^i(s):(U\times [0,s_u)) \rightarrow \R$ and $\S^i_u(s):(U\times [0,s_u))\rightarrow {\Nu}$ satisfying the Rankine-Hugoniot condition, for any $u\in U$ and any $0 \le s < s_u$,
    \begin{equation}
    f(\S^i_u(s)) - f(u) = \sigma_u^i(s)\left(\S^i_u(s) - u\right)
    \end{equation}
    and the Lax admissibility criterion, for $u \in U$ and $0 < s < s_u$,
    \begin{equation}
    \lambda^i(\S^i_u(s)) < \sigma_u^i(s) < \lambda^i(u).
    \end{equation}
    Furthermore, $u \mapsto s_u$ is Lipschitz, $(u,s) \mapsto \sigma^i_u(s)$ is $C^3$, and $(u,s) \mapsto \S^i_u(s)$ is $C^3$. Finally, $\sigma^i_u(s)$ satisfies the asymptotic expansion,
    \begin{equation}
    \sigma_u^i(s) = \frac{1}{2}\left(\lambda^i(u) + \lambda^i(\S^i_u(s))\right) + \bigO(s^2)
    \end{equation}
    and similarly $\S^i_u(s)$ satisfies the asymptotic expansion
    \begin{equation}
    \S^i_u(s) = u + r_i(u)s + \frac{s^2}{2}\nabla r_i(u) r_i(u) + \bigO(s^3).
    \end{equation}
    \end{lemma}

    \subsection{Method of relative entropy}\label{sec:idea}

We will use the relative entropy method first introduced by Dafermos \cite{MR546634} and DiPerna \cite{MR523630}. From the assumption of the existence of a convex entropy $\eta$, we define an associated pseudo-distance defined for any $a,b\in \Nu\times \Nu$:
\begin{align}
\eta(a|b)=\eta(a)-\eta(b)-\nabla\eta(b)(a-b).
\end{align}
The quantity  $\eta(a|b)$ is called the relative entropy of $a$ with respect to $b$, and is equivalent to $|a-b|^2$.  
We also define the relative entropy-flux: For $a,b\in\mathbb{R}^2$,
\begin{align}
q(a;b)=q(a)-q(b)-\nabla\eta(b)(f(a)-f(b)).
\end{align}
The strength of this notion, is that if $u$ is a weak solution of \eqref{system} or \eqref{isothermal} as well as \eqref{eq:entropy-eqn}, then $u$ verifies also the full family of entropy inequalities  for any constant $b\in \Nu$:
\begin{equation}\label{ineq:relative}
(\eta(u|b))_t+(q(u;b))_x\leq 0.
\end{equation}

This fact is well-known and in particular follows immediately from the proof of Theorem 5.2.1 in \cite{dafermos_big_book}.
\vskip0.3cm
For the family of Euler systems, it is well known that the relative entropy provides a contraction property for the rarefaction fan function $t,x\to b(x,t)$, even in multi-D (\cite{MR2303477,MR3357629}, also see \Cref{pizza_slice_diss} below). This is because the Euler systems verify Assumption \ref{assum} \ref{assum:convex_left_eign}. In fact, we restrict our study to systems which verify this property.
\vskip0.3cm
We state some simple but useful properties of the relative quantities.

\begin{lemma}\label{rel_facts_lemma} For any $\mathcal{W}$ open subset of $\mathcal{V}$ with $\overline{\mathcal{W}} \subset \mathcal{V}$, there exists a constant $K_1 > 0$ such that
\begin{align*}
    |q(a; b)| &\leq K_1 \eta(a|b), \quad \forall (a,b) \in \mathcal{V} \times \overline{\mathcal{W}}, \\
    |q(a; b_1) - q(a; b_2)| &\leq K_1 |b_1 - b_2|, \quad \forall (b_1, b_2) \in \overline{\mathcal{W}}^2, \quad a \in \mathcal{V}, \\
    |\eta(a|b_1) - \eta(a|b_2)| &\leq K_1 |b_1 - b_2|, \quad \forall (b_1, b_2) \in \overline{\mathcal{W}}^2, \quad a \in \mathcal{V}.
\end{align*}
\end{lemma}

\Cref{rel_facts_lemma} is well-known. For a proof, we refer the reader to e.g. \cite[Lemma 7.2]{MR4487515}.

\subsection{Dissipation at a shock}

The following classical lemma gives a precise expression for the entropy dissipated by a shock. 
\begin{lemma} \label{lemma:entropy-quant}
    Given an entropy entropy-flux pair $(\eta,\ q)$, any $v\in \Nu$, and any $i = 1,\dots, n$ we have the identity
    \begin{equation}\label{eq:entropy-quant}
        q(\S_u^i(s);v) - \sigma_u^i(s) \eta(\S_u^i(s)|v) = q(u;v) - \sigma_u^i(s)\eta(u|v) + \int_0^s \dot \sigma_u^i(t) \eta(u|\S_u^i(t))\,dt.
    \end{equation}
\end{lemma}

This result in fact goes back to Lax \cite{MR0393870}. One reference for a simple proof is \cite{Leger2011} (see Lemma 3 therein).

\section{Quantitative dissipation estimates}\label{sec:estimates}
Now, given a fixed entropic shock $(u_L,u_R,\sigma)$ we consider the pseudo distance
\begin{equation} \label{eq:pseudo-dist}
    E_t(u;a_1,a_2,h) := a_1\int_{-\infty}^{h(t)} \eta(u(x,t)| u_L)\,dx + a_2\int^{\infty}_{h(t)} \eta(u(x,t)| u_R)\,dx
\end{equation}
where $a_1,a_2  > 0$ are constants to be determined and $h(t)$ is any Lipschitz shift. 
By the following Lemma, we see that if $E_t$ is decreasing we have $L^2$ control on $u$ and our shifted shock $(u_L,u_R,\dot h)$. 
\begin{lemma}[\protect{\cite{VASSEUR2008323,Leger2011}}]\label{lemma:leger-square}
    For any compact set $V \subset \Nu$ we have, for all $(u,v) \in \Nu \times V$, that 
    $$ |u-v|^2 \sim \eta(u|v).$$ 
\end{lemma}
The goal now is to find $a_1,a_2 > 0$ such that, for any $u \in \Sweak$, there exists a Lipschitz $h$ giving this decay. 

In this paper, we consider what we term ``small shocks'' and ``large shocks.'' ``Small'' shocks $(u_L,u_R)$ will verify \eqref{small_shock_must_satisfy}, below. All other shocks will be ``large.''

\subsection{Small shocks}

We now present a major contribution of this work, which solves this problem in the case of small $\abs{u_L-u_R}$, and furthermore, produces a quantitative upper bound on this decay. 

\begin{proposition}[\protect{Existence of the shift function for small shocks with quantitative dissipation estimate}]\label{prop:small_shock_diss}
    Consider a system~\eqref{system} verifying all the Assumption \ref{assum}. Let $d\in \Nu$. 
    Then there exists $\hat \lambda,\ C_1,\ K,\ \epsilon > 0$ such that the following holds:
        Consider any $1$-shock or $n$-shock $(u_L,u_R,\sigma_{LR})$ satisfying 
        \begin{align}\label{small_shock_must_satisfy}
        |u_L - d| + |u_R - d| \leq \epsilon.
        \end{align}
        Letting $s_0$ be the arclength between $u_L$ and $u_R$ along the curve $S_{u_L}^1$ for (Problem $\epsilon$-BV) and letting $s_0=2|\ln(\rho_R/\rho_L)|$ for (Problem Isothermal), with $\rho_L$ and $\rho_R$ the densities on the left and right of the shock, respectively, we find if $a_1,a_2$ satisfy
        \begin{align}
            1 + \frac{C_1s_0}{2} &\leq \frac{a_1}{a_2} \leq 1+2C_1s_0 \quad\quad \text{if $(u_L,u_R)$ is a 1-shock} \label{control_a_one}\\
            1 - 2C_1s_0 &\leq \frac{a_1}{a_2} \leq 1-\frac{C_1s_0}{2} \quad\quad \text{if $(u_L,u_R)$ is an n-shock}, \label{control_a_two}
        \end{align}
        then for any $u \in \mathcal{S}_{\text{weak}}$, $T > 0$, $T_{start} \in [0,T]$, and $x_0 \in \R$  there exists a Lipschitz map $h\colon[T_{start},T] \to \R$ such that $h(T_{start}) = x_0$ and we have the following bound on our dissipation functional
        \begin{align} \nonumber
            &a_2\left[q(u(h(t)+,t);u_R)-\dot{h}(t) \ \eta(u(h(t)+,t)|u_R)\right]\\ \label{diss:shock}
            &\qquad\qquad  -a_1\left[q(u(h(t)-,t);u_L)-\dot{h}(t) \ \eta(u(h(t)-,t)|u_L)\right] \leq -a_2 Ks_0(\dot h(t) - \sigma_{LR})^2,
        \end{align}
        for almost all $t \in [T_{start},T]$. 
        {The function $h$ is called the \emph{shift function}.} 
        Furthermore, if $(u_L,u_R)$ is a 1-shock, then for almost all $t \in [T_{start},T]$
        \begin{equation}
            -\frac{\hat\lambda}{2} \leq \dot h(t) \leq \sup_{v \in B_{2\epsilon}(d)} \lambda_1(v). \label{control_one_shock1020}
        \end{equation}
        Likewise, if $(u_L,u_R)$ is an n-shock, then for almost all $t \in [T_{start},T]$
        \begin{equation}
            \inf_{v \in B_{2\epsilon}(d)} \lambda_n(v) \leq \dot h(t) \leq \frac{\hat\lambda}{2}. \label{control_two_shock1020}
        \end{equation}
        Furthermore, in the case of (Problem Isothermal) the above holds uniformly for $d \in \mathcal{K}$, where $\mathcal{K}\subset\Nu$ is compact.  
\end{proposition}

The proof of \Cref{prop:small_shock_diss} is saved for \Cref{sec:proof_prop:small_shock_diss}.

Integrating \eqref{ineq:relative} once with $b= u_L$ for  $x\in (-\infty, h(t))$, then integrating \eqref{ineq:relative} again with $b= u_R$ for $x\in (h(t),\infty)$, adding together the results, and using \eqref{diss:shock} together with the Strong Trace Property (\Cref{strong_trace_prop}) provides the desired contraction property for \eqref{eq:pseudo-dist} in the case of a single shock provided that,
$$
a_1/a_2 \text{ is between } 1+\frac{C_1}{2}(-1)^{i+1}s_0 \text{ and } 1+2C_1(-1)^{i+1}s_0,
$$
where $(u_L,u_R)$ is an $i$-shock with strength $s_0$. In a nutshell, we can define a function $a$ which decreases at a 1-shock and increases at a 2-shock, and the variation of $a$  can be chosen with size proportional to the size of the shock. From the estimates on $\dot{h}$ (see \eqref{control_one_shock1020}, \eqref{control_two_shock1020}), we have a finite speed of propagation. Furthermore, our control on $h$ means that a shift of a 1-shock cannot touch the shift of a 2-shock if the 1-shock started to the left of the 2-shock. This will be vitally important when we introduce shifts into the solution to Riemann problems whose solutions involve two shocks.  Both shock speeds will be dictated by  artificial velocities given by shifts. We need to guarantee that the \emph{shocks do not interact at some time after the initial time} in order to maintain the feature of  classical solutions to the Riemann problem, where two shocks emanating from a solution to a Riemann problem will never touch in the future. For more details, we refer to \Cref{sec:wave_non_interact} below.

\subsection{Large shocks}

The dissipation result we need for large shocks of the isothermal system is the following:

\begin{proposition}[\protect{Existence of the shift function for large shocks with quantitative dissipation estimate \cite[Proposition 4.1]{move_entire_solution_system}}]\label{prop:large_shock_diss}
We consider the isothermal Euler system \eqref{isothermal}. 

Fix $T,T_{\text{start}},B, \Lambda > 0$ (with $T_{\text{start}}<T$) and $x_0\in\mathbb{R}$. Consider $u\in\Sweak$, a bounded weak solution to \eqref{isothermal} with $\|u\|_{L^\infty}<B$. Then let $(u_R, u_L,\sigma_{LR})$ be a $1$-shock with $\abs{u_R}, \; \abs{u_L} <B$.

Assume that
\begin{equation}
    r > \Lambda,
\end{equation}
where $r$ satisfies $S^{1}_{u_L}(r) = u_R$.

Then, there exists a constant $a_* > 0$ such that for all $\alpha\in(0,a_*)$, and any $a\in(\alpha,a_*)$, there exists a Lipschitz continuous map $h\colon [T_{\text{start}},T] \to \mathbb{R}$ with $h(T_{\text{start}}) = x_0$ and such that for almost every $t$,
\begin{equation}\label{eq:large_shock_diss}
    a \left( q(u_+; u_R) - \dot{h}(t) \eta(u_+ | u_R) \right) - q(u_-; u_L) + \dot{h}(t) \eta(u_- | u_L)
    \leq -c \left| \sigma_{LR} - \dot{h}(t) \right|^2, 
\end{equation}
where  $u_\pm \coloneqq u(h(t) \pm, t).$ The function $h$ is called the \emph{shift function}. 

The constant $a_*$ depends on $B$ and $\Lambda$. The constant $c$ depends on $\alpha$, $B$ and $\Lambda$. 

Moreover, for each $t \in [T_{\text{start}}, T]$ either $\dot{h}(t) < \inf \lambda_1$ or  $(u_+, u_-, \dot{h})$ is a 1-shock with 
\begin{align*}
u_- \in \left\{ u \mid \eta(u \mid u_{L}) \leq a \eta(u \mid u_{R}) \right\}
\end{align*}
(possibly  $u_+ = u_-$  and $\dot{h}= \lambda_1(u_\pm)$).

We have an analogous result for 2-shocks as well. 

Moreover, let $h_1\colon[T_1,T]\to\mathbb{R}$ be a shift function corresponding to a 1-shock (constructed either according to this \Cref{prop:large_shock_diss} or \Cref{prop:small_shock_diss}, above) and let $h_2\colon[T_2,T]\to\mathbb{R}$ be a shift function corresponding to a 2-shock (again constructed either according to this \Cref{prop:large_shock_diss} or \Cref{prop:small_shock_diss}, above). Then, if at some time $T^*\in[T_1,T]\cap[T_2,T]$, we have that $h_1(T^*)\leq h_2(T^*)$, then in fact
\begin{align}\label{ordering_shifts}
h_1(t)\leq h_2(t)
\end{align}
for all $t\geq T^*$.
\end{proposition}
\begin{proof}
\Cref{prop:large_shock_diss} is a slight variation on \cite[Proposition 4.1]{move_entire_solution_system} and \cite[Proposition 5.1]{MR4184662}, which are results that apply to a general class of systems. The proof of \Cref{prop:large_shock_diss} is almost identical to the proof of \cite[Proposition 4.1]{move_entire_solution_system} in the paper \cite{move_entire_solution_system} and thus we do not reproduce the proof here. The key point is that the isothermal Euler system \eqref{isothermal} with its natural entropy and entropy-flux pair (see \eqref{isothermal_eta_q}) clearly verifies the hypotheses $\mathcal{(H)}$ on 1-shocks and $\mathcal{(H)}^*$ on $n$-shocks (with $n=2$) used in the work \cite{move_entire_solution_system}. The result \eqref{ordering_shifts} follows immediately from the proof of \cite[Proposition 5.1]{MR4184662} and the construction of the shifts for small shocks given by \eqref{eq:shift-ode}.

\end{proof}



\subsection{Approximate rarefaction waves}

We need a similar control for approximations of rarefactions discretized via the  front tracking method. 

\begin{proposition}[\cite{MR4487515}]\label{pizza_slice_diss}
There exists a constant $K>0$ such that the following is true.
For any  $\bar{u}(y)$ $v_L\leq y\leq v_R$  rarefaction wave for \eqref{system} (Problem $\epsilon$-BV) or \eqref{isothermal} (Problem Isothermal), denote 
$$
\bar\sigma=|v_L-v_R|+\sup_{v\in [v_L,v_R]} |u_L-\bar{u}(y)|, \qquad \bar{u}(v_L)=u_L, \ \bar{u}(v_R)=u_R.
$$
 Then for any $u\in \Sweak$, any $v_L\leq v\leq v_R$,  and any $t>0$ we have:
 $$
 \int_{0}^{t} \left\{q(u(tv+,t);u_R)-q(u(tv-,t);u_L)-v\left( \eta(u(tv+,t)|u_R) -\eta(u(tv-,t)|u_L) \right)
 \right\}\,dt\leq K\bar\sigma ^2 t.
 $$
\end{proposition}

Note we will always divide the rarefaction fan into many small pieces in the front tracking scheme. So we only need to consider the case when $\bar \sigma$ is small.

\section{The front tracking algorithm}\label{section:front}
Now we consider two cases: small-BV solution for general systems (Problem $\epsilon$-BV), and large-BV solutions for isothermal Euler (Problem Isothermal). We remark here that the large-BV theory also applies to the $2\times 2$ Temple-class systems.

Here we use the front tracking scheme in Bressan-Colombo \cite{BressanC} and Colombo-Risebro \cite{ColomboRisebro}. By allowing some errors in wave speed, i.e. considering the $\nu$-approximate solution, introduced below, we do not have to consider the non-physical shocks as in \cite{MR4487515}.

\subsection{Important note on notation} 

In this section, we use the tools developed by Bressan-Colombo \cite{BressanC} and Colombo-Risebro \cite{ColomboRisebro}. In the interest of readability, in this section we use the notation used by Bressan-Colombo \cite{BressanC} and Colombo-Risebro \cite{ColomboRisebro} including for shock and rarefaction curves (as introduced below in \Cref{sec:front_tracking_no_shift}) instead of the symbols $S^i_u$ and $R^i_u$ as used in other sections of this paper.

\subsection{Introduction to the Riemann problem}

Now we review the  $\nu$-approximate solution and the front tracking scheme.

To start, given a Riemann problem of \eqref{system} with two constant states $u_-$ and $u_+$ sufficiently close, a solution with at most three constant states, connected by either shocks or rarefaction fans, can always be found. More precisely, there exist $C^2$ curves $\sigma\mapsto T_i(\sigma)(u_-)$, $i=1,2$, parametrized  by arclength, such that
\begin{align}\label{RP_curves}
u_+=T_2(\sigma_2)\circ T_1(\sigma_1)(u_-),
\end{align}
for some $\sigma_1$ and $\sigma_2$. We define $u_0\coloneqq u_-$ and 
\begin{align}
u_1\coloneqq T_1(\sigma_1)(u_0),\\
u_2\coloneqq T_2(\sigma_2)\circ T_1(\sigma_1)(u_0).
\end{align}
For isentropic gas dynamics or Temple systems, one can solve the Riemann problem when the distance of $u_-$ and $u_+$ is large.

We use the convention that, when $\sigma_i$ is negative (positive) the states $u_{i-1}$ and $u_i$ are separated by an i-shock (i-rarefaction) wave. Further,  the strength of the i-wave is defined as $\abs{\sigma_i}$.


Next we apply some modification on the solution of the Riemann problem by asking for the $\nu$-approximate solution. This will make the front tracking scheme slightly different from its standard version. One advantage is that we can control the number of wave fronts without using the non-physical shock.

\subsection{The front tracking $\nu$-approximate solution}\label{sec:front_tracking_no_shift}

We follow \cite{BressanC} and \cite{ColomboRisebro}.

For hyperbolic conservation laws \eqref{system}, let $A(u)=Df(u)$ denote the Jacobian matrix of $f$ at $x$. Define the averaged matrix
\[
A(u,u')=\int_{0}^1 Df(\theta u+(1-\theta)u')\, d\theta.
\]
Let $\lambda_i(u,u')$, $r_i(u,u')$ and $l_i(u,u')$
be the $i$-th eigenvalue and the corresponding right and left $i$-th eigenvectors of $A(u,u')$, respectively, where we normalize these so that 
\[
l_i\cdot r_j=\left\{\begin{array}{ll}
1 & \hbox{if}\ i=j,\\
0 & \hbox{if}\ i\neq j.
\end{array}\right.
\]
In particular, if $u=u'$, then $\lambda_i(u)$, $r_i(u)$ and $l_i(u)$
are the $i$-th eigenvalue and the right and left $i$-th eigenvectors of $A(u)=A(u,u)$.


We always use the Riemann coordinates, where the rarefaction curves take the form
\beq\label{sigma_def}
\phi_1^+(v,\sigma)=(v_1+\sigma,v_2),\qquad
\phi_2^+(v,\sigma)=(v_1,v_2+\sigma),
\eeq
where $v=(v_1,v_2)$ and $v_1$ and $v_2$ are Riemann invariants in the first and second family, respectively. 

Furthermore, we have 
\begin{lemma}\label{lem_3.1}\cite{BressanC,ColomboRisebro}
For (Problem Isothermal), the $i$-shock curve through the point $v=(v_1,v_2)$ can be parameterized with respect to the Riemann coordinates as
\beq\label{si_shock}
\phi_1^-(v,\sigma)=(v_1+\sigma+\phi(\sigma),v_2+\phi(\sigma)),\qquad
\phi_2^-(v,\sigma)=(v_1+\phi(\sigma),v_2+\sigma+\phi(\sigma)),
\eeq
for a suitable smooth function $\phi$ satisfying
\[
\phi(\sigma)\leq 0,\qquad \phi'(\sigma)\geq 0,
\]
for any $\sigma<0$. 
In particular, for the Temple systems, $\phi=0$. 

For (Problem $\epsilon$-BV),
\[
\phi_1^-(v,\sigma)=(v_1+\sigma,v_2+\hat\phi_2(v,\sigma)\sigma^3),\qquad
\phi_2^-(v,\sigma)=(v_1+\hat\phi_1(v,\sigma)\sigma^3,v_2+\sigma),
\]  
for some smooth functions $\hat\phi_2$ and $\hat\phi_1$.
\end{lemma}


In particular, in this paper, for the isothermal gas dynamics \eqref{isothermal}, we choose a special pair of Riemann invariants
\beq\label{v_12}
v_1=w-\log \rho,\qquad v_2=w+\log \rho.
\eeq
Then by \eqref{sigma_def} and \eqref{si_shock},
\begin{equation}\label{si_iso}
\sigma=\pm 2(\log\rho_r-\log\rho_l)=\pm 2\log\frac{\rho_r}{\rho_l}, 
\end{equation}
with minus sign for $1$ wave and plus sign for $2$ wave,
where the subscripts $l$ and $r$ denote the left and right states of the shock, respectively.

We use the same front tracking algorithm as in \cite{BressanC}
and \cite{ColomboRisebro}. 
To make this paper self-contained, we briefly introduce the algorithm in \cite{ColomboRisebro} for (Problem Isothermal). For (Problem $\epsilon$-BV), the same front tracking scheme given in \cite{BressanC} will be used, where we omit the detail. {We remark that \cite{ColomboRisebro} work with isothermal Euler in the Eulerian coordinates. However, the Riemann invariants are the same in the Lagrangian viewpoint.}

First, we define the approximate solution to Riemann problems. Choose a $C^\infty$ function $\varphi\colon\mathbb{R}\to \mathbb{R}$ such that 
\begin{align}
\begin{cases}
\varphi(s)=1, &\mbox{ for all } s \leq  -2,\\
\varphi'(s)\in [-2,0] ,&\mbox{ for all } s\in [-2,-1],\\
\varphi(s)=0,&\mbox{ for all } s\geq -1,
\end{cases}
\end{align}
and, for a fixed $\nu>0$, define the interpolation between the $i$-shock and the $i$-rarefaction curve.

\beq\label{Riemann1_0}
\Phi_i^\nu(v,\sigma)=\varphi(\frac{\sigma}{\sqrt{\nu}} ) \phi_i^-(v,\sigma)+(1-\varphi(\frac{\sigma}{\sqrt{\nu}} )) \phi_i^+(v,\sigma)
\eeq
with $i=1,2$. In particular for (Problem Isothermal),
\beq\label{Riemann1}
\Phi_1^\nu(v,\sigma)=(v_1+\sigma+\varphi(\frac{\sigma}{\sqrt{\nu}} )\phi(\sigma), v_2+\varphi(\frac{\sigma}{\sqrt{\nu}} )\phi(\sigma))
\eeq
and 
\beq\label{Riemann2}
\Phi_2^\nu(v,\sigma)=(v_1+\varphi(\frac{\sigma}{\sqrt{\nu}} )\phi(\sigma), v_2+\sigma+\varphi(\frac{\sigma}{\sqrt{\nu}} )\phi(\sigma)).
\eeq
In this setting, $\sigma\mapsto\Phi_i^\nu(v,\sigma)$ coincides with the rarefaction curve $\phi_i^+$ when $\sigma\geq -\sqrt{\nu}$ and with the shock curve 
 $\phi_i^-$ when $\sigma\leq-2\sqrt{\nu}$.
 
To avoid confusion, in this section, we call jumps with $\sigma>0$  rarefaction waves, otherwise, if $\sigma<0$ then we call the jump a shock wave. This matches the notation in \cite{BressanC,ColomboRisebro}. However, this \emph{contrasts} with the notation for shock curves $s\mapsto S^i_u(s)$ used in other sections of the present paper to match \cite{MR4667839}, where the parameter for shocks is $s>0$.

Given $\nu$, $u^r$ and $u^l$ with corresponding Riemann coordinates $v^r=(v^r_1,v^r_2)$ and $v^l=(v^l_1,v^l_2)$, we find an approximate solution for the Riemann problem of \eqref{system} or \eqref{isothermal} with initial data $u(x,0)=u^l$ when $x<0$ and $u(x,0)=u^r$ when $x>0$, by finding a unique {intermediate state $v^m$ and a }pair of values $\sigma_1$ and $\sigma_2$ which verify
\beq\label{Riemann1_4}
v^m=\Phi_1^\nu(v^l,\sigma_1),\qquad
v^r=\Phi_2^\nu(v^m,\sigma_2).
\eeq


We approximate the rarefaction by a series of discontinuities, named as the $\nu$-approximate rarefaction fan.

Here we just write down the definition for the $\nu$-approximate rarefaction fan and shock wave in the first family.

Assume $\sigma_1>0$. Then let the integers $h$, $k$ be such that $h\nu\leq v_1^l<(h+1)\nu$ and 
$k\nu\leq v_1^m<(k+1)\nu$. Introducing the states $\omega_1^j=(j\nu,v_2^l)$ and $\hat\omega_1^j=((j+1/2)\nu,v_2^l)$ for $j=h,\dots, k$, we construct a $\nu$-approximate rarefaction fan
\begin{align}\label{R_RP_sol}
v^\nu(x,t)\coloneqq
\begin{cases}
v^l, &\mbox{ if } x < \lambda_1(\hat \omega_1^h)t,\\
\omega_1^j,&\mbox{ if } \lambda_1(\hat \omega_1^{j-1})t<x<\lambda_1(\hat \omega_1^j)t,\qquad j=h+1,\dots, k,\\
v^m,&\mbox{ if } \lambda_1(\hat \omega_1^k)t<x.
\end{cases}
\end{align}
On the other hand, if $\sigma_1<0$, the states $v^l$ and $v^m$ are connected by a single shock
\begin{align}\label{accurate_RP_sol}
v^\nu(x,t)\coloneqq
\begin{cases}
v^l, &\mbox{ if } x < \lambda_1^\varphi(v^l,\sigma_1)t,\\
v^m,&\mbox{ if } \lambda_1^\varphi(v^l,\sigma_1)t<x.
\end{cases}
\end{align}
The shock speed $ \lambda_1^\varphi$ is defined as 
\begin{align}\label{averaged_speed}
 \lambda_1^\varphi(v^l,\sigma_1)=\varphi(\sigma_1/\sqrt \nu)\cdot \lambda_1^s(v^l,\sigma_1)+(1-\varphi(\sigma_1/\sqrt \nu))\cdot \lambda_1^r(v^l,\sigma_1)
\end{align}
with 
\[
 \lambda_1^s(v^l,\sigma_1)= \lambda_1(v^l,\phi^{-}_1(v^l,\sigma_1)),
\]
\[
 \lambda_1^r(v^l,\sigma_1)= \sum_j \frac{\hbox{meas}([j\nu, (j+1)\nu]\cap [v_1^m,v_1^l])}{|\sigma_1|}\lambda_1(\hat\omega_1^j).
\]

Remark that here $ \lambda_1^s(v^l,\sigma_1)$ is the true shock speed, which coincides with $\lambda_1^\varphi(v^l,\sigma_1)$ when $\sigma\leq-2\sqrt{\nu}$.

The construction of the $\nu$-approximate solution for waves of the second family is similar to the first family. We refer the readers to \cite{ColomboRisebro}.

Let $\hat u(x,0)$ be a piecewise constant initial data.
Now, we construct the $\nu$-approximate solution $u^\nu$ as follows. At time $t=0$, solve the Riemann problem defined by the jumps in $\hat u(x,0)$ using the above algorithm. Let the new piecewise constant solution evolve, until the first interaction takes place. Then we still use our algorithm to solve the Riemann problem at the place where the first interaction takes place. Repeating these steps,  $u^\nu$ can be defined until any finite time $T$, since there are at most finitely many interactions in $[0,T]$ \cite{BressanC,ColomboRisebro}.

\subsection{The shifted $\nu$-approximate front tracking solution}\label{sec:shifted_front_tracking}
    In this paper we consider a shifted $\nu$-approximate front tracking solution $\psi$. 
    For this shifted approximation we still use the $\nu$-approximate curves \eqref{Riemann1} and \eqref{Riemann2}.
    When a solution to a Riemann problem gives a $\nu$-approximate rarefaction fan we produce an un-shifted rarefaction, identical to what is described by~\eqref{R_RP_sol}.
    When the Riemann problem from an interaction at $(\overline x, \overline t)$ produces a shock (without loss of generality, suppose it is a 1-shock) with states $v^l$ and $v^m = \Phi_1^\nu(v^l, \sigma)$, with $\sigma < 0$ our shifted $\nu$-approximate shock is
    \begin{equation}\label{RP_shock_shift}
        \psi(x,t)\coloneqq
            \begin{cases}
            v^l, &\mbox{ if } h_-(t) < x < h(t),\\
            v^m,&\mbox{ if } h(t) < x < h_+(t).
            \end{cases}
    \end{equation}
    The Lipschitz curves $h_-$ and $h_+$ are the closest fronts to the left/right of $h(t)$ between $\overline t $ and the time of the next interaction, while $h(t)$ is our shift function. 
    Our choice of shift $h(t)$ depends on the size $\sigma$ of the $\nu$-approximate shock $(v^l, v^m)$. 
    If $\sigma \leq -2\sqrt{\nu}$ then our $\nu$-approximate shock is a real shock, and we use either \Cref{prop:small_shock_diss} if $\sigma < \epsilon$ (where this is the $\epsilon$ defined in \Cref{prop:small_shock_diss}) or otherwise use \Cref{prop:large_shock_diss} to find a suitable shift function $h(t)$, with $(x_0, T_{start}) = (\overline x, \overline t)$. 
    If $\sigma > -2\sqrt{\nu}$ the $\nu$-approximate shock $(v^l,v^m)$ is not necessarily a real shock, so we instead apply \Cref{prop:small_shock_diss} to the shock $(v^l, \phi_1^-(v^l,\sigma))$ and use the resulting shift as our $h$ in~\eqref{RP_shock_shift}.
    In doing this we remark that for these small shocks we have the estimates
    \beq \label{eq:small-shock-error-psi}
        \begin{aligned}
            |\Phi_1^\nu(v^l,\sigma) - \phi_1^-(v^l,\sigma) | &\leq K|\sigma|^3,\\
            | \lambda_1^\varphi(v^l,\sigma)- \lambda_1^s(v^l,\sigma)| &\leq K \nu
        \end{aligned}
    \eeq
    where $\lambda_1^\varphi$ and $\lambda_1^s$ are the $\nu$-approximate interpolated shock speed~\eqref{averaged_speed} and the true Rankine-Hugoniot shock speed, respectively. 
    The first of these estimates follows easily from~\eqref{Riemann1_0} and \Cref{lem_hugoniot} while the second estimate is known from \cite[Section 7]{BressanC}. We remark that these two estimates will be needed later for estimates to go between the $\nu$-approximate shocks and true shocks in Step 1 of \Cref{sec:proof_main_theorem}.

Here, using a similar argument as in Section 4 in \cite{BressanC} for un-shifted  $\nu$-approximate solution, we can show that there are only finitely many wave fronts and wave interactions for any shifted $\nu$-approximate solution $\psi(x,t)$ in $\mathcal D$, where $\mathcal D$ is the space of functions with small or finite BV norm for (Problem $\epsilon$-BV) or (Problem Isothermal) to be defined later, respectively, for any $\nu\ll 1$. 


\vspace{.1in}

\subsubsection{Non-interaction of waves}\label{sec:wave_non_interact} Our construction of the shift for shock waves (\Cref{prop:large_shock_diss} and \Cref{prop:small_shock_diss})  guarantees that waves generated in the solution of a Riemann problem will not meet (after they are initially born) even after adding the shift.

For (Problem $\epsilon$-BV) this property follows from \eqref{control_one_shock1020} and \eqref{control_two_shock1020} and making $\epsilon$ sufficiently small.

For (Problem Isothermal), wave interactions are prevented due to \Cref{prop:large_shock_diss} and the construction given by \eqref{eq:shift-ode} (see also \cite[p.~158]{MR4184662}). In particular, for the isothermal system in Lagrangian coordinates, both shifted and un-shifted 1-waves will always have a strictly negative speed and both shifted and un-shifted 2-waves will always have a strictly positive speed.

\subsection{Control on the BV norm}
For both the un-shifted and shifted $\nu$-approximate solutions, we use the following notations on total variation and Glimm potentials. At any time $t$, we use $|\sigma_{i,\alpha}|'s$ to denote the strengths of all jumps, where the wave corresponding to $|\sigma_{i,\alpha}|$ is located at $x=x_\alpha$ and is in the $i$-th family with $i=1$ or $2$. We abuse the notion a bit here, since at the time when a rarefaction wave is formed at $x_\alpha$, there might be more than 1 wave sharing the same name $\sigma_{i,\alpha}$. But these waves split immediately.
\begin{itemize}
\item
For small-BV solution (Problem $\epsilon$-BV), we denote
\beq\label{def_v2}
V(u)=\sum_\alpha\sum_{i}|\sigma_{i,\alpha}|,\qquad
Q(u)=\sum_{\sigma_{i,\alpha},\sigma_{j,\beta}\in \mathcal{A}}|\sigma_{i,\alpha}\sigma_{j,\beta}|,
\qquad 
U=V+\kappa Q,
\eeq
where  $\mathcal{A}$ is the set of any pair of approaching waves $\alpha$ and $\beta$. Here we say a pair of waves corresponding to $\sigma_{i,\alpha}$ and $\sigma_{j,\beta}$ are approaching if either $x_{\alpha}<x_{\beta}$ and $j<i$, or if $j=i$, $\sigma_{i,\alpha}<0$ or $\sigma_{j,\beta}<0$. Here $\sigma_{i,\alpha}$ and $\sigma_{j,\beta}$ are the $\sigma$ values for waves defined in \eqref{sigma_def}, \eqref{si_shock} and \eqref{si_iso}.
\item
For large-BV solution of isothermal gas (Problem Isothermal) as well as the Temple systems, correspondingly, one can define 
\beq\label{def_v1}
V_2(u)=\sum_{\alpha}\sum_i(1-\eta\, \hbox{sign}\,\sigma_{i, \alpha})|\sigma_{i,\alpha}|,\ 
Q(u)=\sum_{\sigma_{i,\alpha},\sigma_{j,\beta}\in \mathcal{A}}|\sigma_{i,\alpha}\sigma_{j,\beta}|,
\ 
U=\kappa_2\cdot V_2+Q.
\eeq

Here $\mathcal{A}$ is the set of any pair of approaching waves $\alpha$ and $\beta$. Furthermore, $\eta$ in $(0,1)$, and $\kappa_2$ is a positive constant depending on $M$, $u_{\pm\infty}$ and the flux $f$ (cf. \eqref{large_BV_data}).
\end{itemize}
Here, without confusion, $U$ has different meanings for (Problem $\epsilon$-BV) and (Problem Isothermal).


In this paper, we only consider the un-shifted (see \Cref{sec:front_tracking_no_shift}) or shifted (see \Cref{sec:shifted_front_tracking}) $\nu$-approximate solutions in the domain of small BV solutions for (Problem $\epsilon$-BV), i.e. 
$$\mathcal D=\{u\hbox{ solves \eqref{u_piece} }|\ 0<U(u)<\epsilon\},$$ for sufficiently small $\epsilon$, {for the piecewise-constant functions defined below in \eqref{u_piece}}, or  in the domain of bounded BV solutions for (Problem Isothermal) with the same far field states $u_{-\infty}$ and $u_{\infty}$, i.e.
\beq\label{Dcal_iso}\mathcal D=\{u\hbox{ solves \eqref{u_piece} }, \lim_{x\rightarrow \pm\infty}u=u_{\pm\infty}\,|\ 0<U(u)<M\},\eeq
for any $M>0$. Compare with \eqref{small_BV_data} and \eqref{large_BV_data}, respectively. 
Without confusion, the domain $\mathcal D$ has different meaning for two different problems, respectively.

Applying the method in \cite{ColomboRisebro} and \cite{BressanC} to shifted or un-shifted {$\nu$-approximate} solution $u$,
we know that, when $u\in \mathcal D$, $\kappa$ and $\kappa_2$ are sufficiently large and $\eta$ is sufficiently small,
 $U(u)$ decays in $t$, and for any time $t$, i.e.
 there exists a constant $K$, such that,
\beq\label{UTV}
\frac{1}{K}\cdot \|u\|_{BV(\mathbb{R})} \leq U(u)\leq K\cdot \|u\|_{BV(\mathbb{R})} .
\eeq


\subsubsection{Control on BV norm of the shifted $\nu$-approximate solution $\psi$}

We have the following control on any shifted or un-shifted $\nu$-approximate solution $\psi$:
\begin{align}\label{BV_psi}
\norm{\psi(\cdot,t)}_{\text{BV}(\mathbb{R})}\leq \tilde{K} \norm{\psi(\cdot,0)}_{\text{BV}(\mathbb{R})},
\end{align} 
for $t>0$ and for a constant $\tilde{K}\geq 1$ which does not depend on $\epsilon$ or $M$.

In fact, for (Problem $\epsilon$-BV), \eqref{BV_psi} is a direct result from the time decay of $U$ since the total variation $V$ is much larger than $Q$ when $\epsilon$ is small enough. More precisely, since $\epsilon$ is always less than $1$, so $V(0)<1$ and $Q(0)\leq V^2(0)<V(0)$. Hence 
\[V(t)<U(t)\leq U(0)<V(0)+\kappa Q(0)\leq (1+\kappa)V(0).\] 

It is a little trickier to show \eqref{BV_psi} for (Problem Isothermal) than (Problem $\epsilon$-BV). Later, after estimating the error of weighted $L^1$ norm caused by adding shifts (\Cref{lemma:front_tracking_error_estimate}), we will allow some tiny shift in the wave speeds to avoid the interaction of more than two waves. See Step 4 in the Proof of \Cref{prop:iso_a}. Then for any pairwise interaction, the total variation of the solution, measured by $V=\sum{|\sigma_{i,\alpha}|}$, decreases in time, see Proposition
\ref{prop4.2_2} summarizing the results in \cite{ColomboRisebro}. Note this result still holds for the shifted $\nu$-approximate solution $\psi$, since real wave speeds are not used in the proof. If one measures the total variation of $\psi$ using another wave strength equivalent to $\sum{|\sigma_{i,\alpha}|}$, \eqref{BV_psi} holds with some $\hat K\geq 1$, using the equivalence of different BV norms.

\subsection{The contraction of weighted $L^1$ distance between two $\nu$-approximate solutions}
In this subsection, we first revisit the contracting weighted $L^1$ distance between two front tracking $\nu$-approximate solutions  $u(x,t)$ and $\bar u(x,t)$ both with Rankine-Hugoniot speed (with up to $\mathcal{O}(\nu)$ error), established in \cite{BressanC}
and \cite{ColomboRisebro}. 

Consider any piecewise constant function $u=u(x)$ of the form
\beq\label{u_piece}
u=u_{-\infty}\cdot \chi_{(-\infty,x_1)}+\sum_{\alpha=1}^{n-1}u^\alpha\cdot \chi_{(x_\alpha,x_{\alpha+1}]}+u_{+\infty}\cdot\chi_{(x_n,+\infty)}.
\eeq
Then the Riemann problem at $x_\alpha$ can be solved uniquely by the algorithm defined before in terms of a 1-wave with strength $\sigma_{1,\alpha}$ and a 2-wave with strength $\sigma_{2,\alpha}$, using \eqref{Riemann1_0} and \eqref{Riemann1_4}.

We can now define an elementary path for a homotopy between any two $\nu$-approximate solutions in $\mathcal D$.
	   \begin{figure}[tb] \centering
   		   \includegraphics[scale=.6]{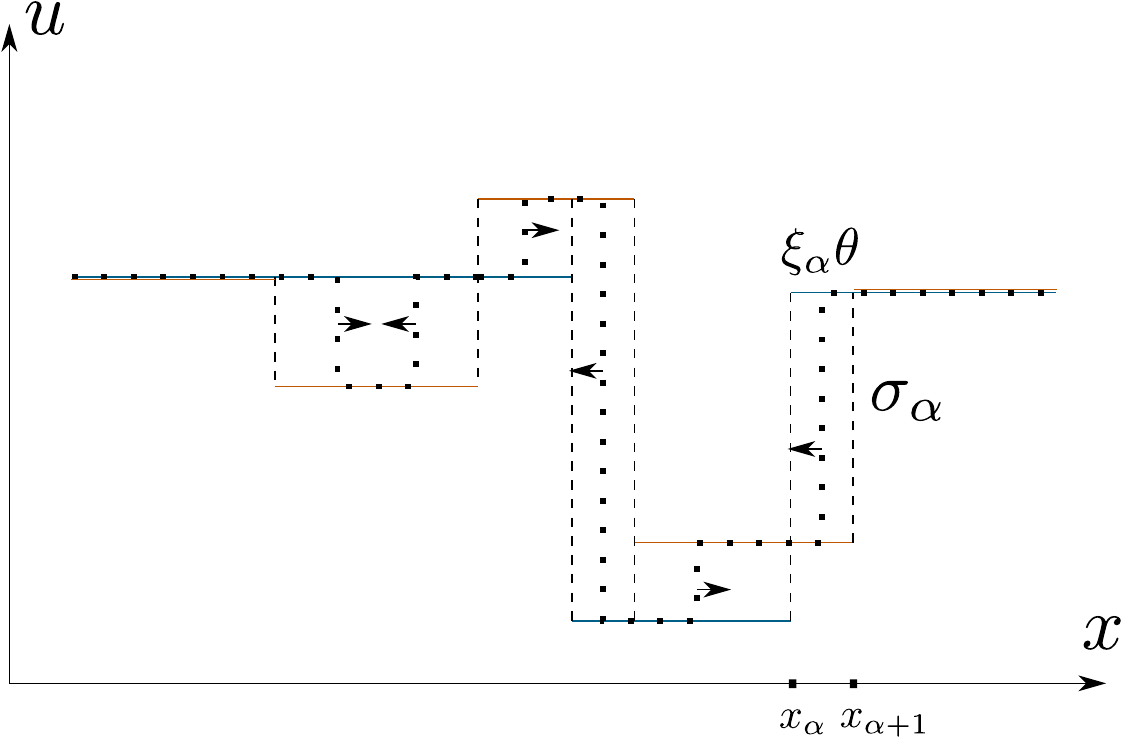}
	\caption{The homotopy between two un-shifted piecewise constant $\nu$-approximate front tracking  solutions in Definition \ref{def_homo} can always be defined. Following the arrow, the path $\gamma(\theta)$ (black dot) moves from the orange solution when $\theta=a$ to the blue solution when $\theta=b$. The $L^1$ distance between these two piecewise constant solutions equals to the total area of all rectangles. The width and hight of each rectangle are $\xi_{\alpha}\theta$ and wave strength $|\sigma_\alpha|$, respectively. In each rectangle, the homotopy are between two discontinuities with the same speed, so $\xi_{\alpha}\theta$ and $|\sigma|$ are constant functions in time, except at interactions.   \label{pic0}
	}
 \end{figure}
\begin{definition} \label{def_homo}
Let $(a,b)$ be an open interval. An \emph{elementary path} in $\mathcal D$  is a map $\gamma\colon(a,b)\to \mathcal  D$ of the form 
\beq\label{0}
\gamma(\theta)=u_{-\infty}\cdot \chi_{(-\infty,x_1^\theta]}+\sum_{\alpha=1}^{n-1}u^\alpha\cdot \chi_{(x_\alpha^\theta,x_{\alpha+1}^\theta]}+u_{+\infty}\cdot\chi_{(x_n^\theta,+\infty)},\qquad
x_\alpha^\theta=x_\alpha+\xi_{\alpha}\theta
\eeq
with $x_{\alpha}^\theta<x_{\alpha+1}^\theta$ for all $\theta\in(a,b)$
and $\alpha=1,\dots, n-1$. Here $\xi_\alpha$ and $u^\alpha$ are independent of $\theta$.  
\end{definition}

At a positive time $t>0$, due to the evolution of the states, the elementary path may become a pseudopolygonal, defined as following, including at most countably many pieces of elementary paths.

\begin{definition}\label{def3}
A map $\gamma\colon[a,b]\rightarrow L^1_{\text{loc}}$ is a \emph{pseudopolygonal} if there exist countably many disjoint open intervals $J_h=(a_h,b_h) \subset [a,b]$ such that
\begin{itemize}
\item[(i)] The restriction of $\gamma$ to each $J_h$ is an elementary path $\gamma_h$, where we still use \eqref{0} but with a subscript $h$ on $a,\ b$ and superscript $h$ on $\xi_\alpha$, $x_\alpha$ and $x^{\theta}_\alpha$, to denote $J_h$.
\item[(ii)] The difference between $[a,b]$ and $\cup_{h\geq 1} J_h$ is at most countable.
\end{itemize}

\end{definition}


Using a semi-group notation, we denote any $\nu$-approximate solution by
\[
u(\cdot,t)={\mathcal S}_t \circ u(\cdot,0)
\] 
for any time $t$, where ${\mathcal S}_t$ denotes the semi-group generated by the front tracking scheme discussed in the previous section. For convenience, we omit the superscripts showing the dependence of $u$ on $\nu$ in the rest of the paper.

Let us summarize some important results established by \cite{BressanC} and \cite{ColomboRisebro}. First, there are only finitely many discontinuities in $u$ when $t\in[0,T]$.  Secondly, starting from an initial pseudopolygonal $\gamma_0$ satisfying the Definition \ref{def_homo} connecting $u$ and $\bar u$,
the path $\gamma(\theta)={\mathcal S}_t \circ \gamma_0$ with $\theta\in[a,b]$, for any time $t>0$, is still a pseudopolygonal. This means there exist countably many open intervals $J_h$ such that $[a,b] \backslash \cup J_h$ is countable and the wave-front configuration of the solution $u^\theta$ on $[0,T]$ remains the same as $\theta$ ranges on each $J_h$. And the homotopy has $L^1$ continuity at the endpoints $\theta=a_h$ and $b_h$ of $J_h$.
We can always set our initial path $\gamma_0$ as an elementary path as shown in Figure \ref{pic0}.



Then we define the weighted distance, which is equivalent to the $L^1$ distance, as following. We still use the definition of \cite{ColomboRisebro} and \cite{BressanC}, so their main results on the time decay of weighted distance still hold.

Note, for any $t$, in each interval $\theta\in(a_h, b_h)$, there is an elementary path, the wave strengths $\sigma_{i,\alpha}$ and shifting speeds $\xi_\alpha$ which take constant values. Remark that the shifting speeds $\xi_\alpha$ are not to be confused with the shifting \emph{functions} defined in \Cref{prop:large_shock_diss} and \Cref{prop:small_shock_diss}. 

Recall (see \Cref{sec:not}) that for a function $f$, we always use $(f)_\pm$ to denote $(|f|\pm f)/2$, i.e. the positive/negative part of $f$. 

\begin{itemize}
\item For large data solution of isothermal gas (Problem Isothermal) as well as the Temple systems as in \cite{ColomboRisebro}:
\beq\label{1}
\|\gamma\|=\sum_{h\geq 1}(b_h-a_h)\cdot \Upsilon^h_{\xi}
\eeq
where, after dropping the superscripts $h$, we define
\beq\label{Upsilon_def}
\Upsilon_{\xi}=\sum_{\alpha=1}^n\sum_{i=1}^2 |\sigma_{i,\alpha}\xi_{\alpha}|\, e^{H_1 S_{i,\alpha}+H_2 R_\alpha+H_3 V_1}
\eeq
where $H_1$, $H_2$ and $H_3$ are some positive constants, and
\beq\label{3}
S_{i,\alpha}=2[\sum_{\beta=1}^n\sum_{j=1}^2(\sigma_{j,\beta})_-]-(\sigma_{i, \alpha})_-,\quad
R_{\alpha}=\sum_{\beta=1}^{\alpha-1}|\sigma_{2,\beta}|
+\sum^{n}_{\beta=\alpha+1}|\sigma_{1,\beta}|.
\eeq
\item For small BV solutions (Problem $\epsilon$-BV) as in \cite{BressanC},
\beq\label{4}
\|\gamma\|=\sum_{h\geq 1}(b_h-a_h)\cdot \Upsilon^h_{\xi}
\eeq
where, after dropping the superscripts $h$, we define
\beq\label{5}
\Upsilon_{\xi}=\sum_{\alpha=1}^n\sum_{i=1}^2 |\sigma_{i,\alpha}\xi_{\alpha}|\, R_i^\alpha,
\eeq
with
\beq\label{6}
R_i^\alpha=(2+\sgn \sigma_{i,\alpha})
\Big(1+K\sum_{(\sigma_{i, \alpha},\sigma_{j, \beta})\in \mathcal A}
 |\sigma_{j,\beta}|\Big)\, \exp\Big\{   K\sum_{(\sigma_{j, \beta},\sigma_{j', \beta'})\in \mathcal A}|\sigma_{j,\beta}\sigma_{j',\beta'}|\Big\}
\eeq
and $\mathcal A$ denotes the set of all pairs of approaching waves,
$Q$ is the wave interaction potential defined before, and $K$ is a sufficiently large constant.
\item We then define the weighted distance as
\beq\label{d_def}
d_\nu(u,\bar u)=\inf\{\|\gamma\|\hbox{ such that $\gamma\colon[a,b]\rightarrow \mathcal D $ is a pseudopolygonal joining $u$ with $\bar u$}\}.
\eeq
\end{itemize}

%

\begin{proposition}[\cite{BressanC,ColomboRisebro}]\label{control_L1_norm} 

Consider both (Problem $\epsilon$-BV) and (Problem Isothermal).
For any ${\mathcal S}_t$, which is the $\nu$-approximate semi-group from the front tracking solution, there exists positive constants $K$, $H_1$,  $H_2$ and $H_3$, such that the map $\Upsilon_{\xi}({\mathcal S}_t\circ\gamma_0)$, the weighted length $\|\gamma\|$ on each elementary path, is non-increasing in time. 

The norm defined above for the $\nu$-approximate solution is equivalent to the $L^1$ norm, i.e.
\beq\label{equiv}
K_1\cdot\|\gamma\|_{L^1}\leq \|\gamma\|\leq K_2\cdot\|\gamma\|_{L^1},
\eeq
for some positive constants $K_1$ and $K_2.$

The weighted distance $d_\nu$ is uniformly equivalent to the $L^1$ distance. Moreover, $d_\nu$ is contracting with respect to the semi-group ${\mathcal S}_t$, i.e.
\beq\label{d_contract}
d_\nu({\mathcal S}_t u_0, {\mathcal S}_t \bar u_0)\leq d_\nu(u_0,  \bar u_0),
\eeq
for any $t\geq 0$ and $\nu$-approximate initial data $u_0$, $\bar u_0$ in $\mathcal D $.
\end{proposition}
Corresponding to \eqref{1} and \eqref{4}, and the $L^1$-continuity of approximate solutions {with} respect to $\theta$,
\beq\label{gamma_l1}
\|\gamma\|_{L^1}=\sum_{h\geq 1}(b_h-a_h)\sum_{\alpha=1}^n\sum_{i=1}^2 |\sigma^h_{i,\alpha}\xi^h_{\alpha}|.
\eeq
\begin{proof}
Using the time decay of $U(u)$, it is easy to show \eqref{equiv}. Then \eqref{d_contract} is proved in \cite{BressanC,ColomboRisebro}.
The other parts of the theorem are also proved in \cite{BressanC,ColomboRisebro}. Let's add some details to make this paper self-contained.

To show the equivalence between the weighted distance $d_\nu$ and the standard $L^1$ norm, for any time $t$,  we consider a special pseudopolygonal 
\[
\hat \gamma(\theta)=u\cdot\chi_{(-\infty,\theta]}+\bar u\cdot\chi_{(\theta,\infty)},\qquad a\leq\theta\leq b,
\] 
where we choose $a$ and $b$ such that $u=\bar u$ when $x\leq a$ or $x\geq b.$

In fact, for $\hat \gamma(\theta)$, the homotopy changes configuration once there is a jump in $u$ or $\bar u$, so totally finite many times.
When $\theta\in(a_h,b_h)$, where $a_h,b_h$ are defined in Definition \ref{def3}, 
\[(b_h-a_h)\sum_{\alpha=1}^n\sum_{i=1}^2 |\sigma^h_{i,\alpha}\xi^h_{\alpha}|=\|u-\bar u\|_{L^1(a_h,b_h)}.\] 
Then by \eqref{equiv},
\[
K_1 d_\nu(u,  \bar u)\leq K_1 \|\hat\gamma\|\leq\|\hat \gamma\|_{L^1}=\|u-\bar u\|_{L^1}.
\]

On the other  hand, for any pseudopolygonal $\gamma$ connecting $u$ and $\bar u$,
\[
\|u-\bar u\|_{L^1}\leq \| \gamma\|_{L^1}\leq K_2\cdot\|\gamma\|,
\]
where the first inequality comes from the triangle inequality on countable elementary paths and the fact that $\gamma: [a,b]\rightarrow L^1$ is a continuous path.
Taking the infimum on the right hand side, we find 
\[
\|u-\bar u\|_{L^1}\leq K_2\cdot d_\nu(u,  \bar u).
\]

\end{proof}

\begin{remark}\label{dependenceTV}
For (Problem Isothermal), $K_1\leq\min_{i,\alpha,h}\{e^{H_1 S^h_{i,\alpha}+H_2 R^h_\alpha+H_3 V^h_1}\}$, so we can choose
\[
K_1=1.
\]
On the other hand,
 $K_2\geq\max_{i,\alpha,h}\{e^{H_1 S^h_{i,\alpha}+H_2 R^h_\alpha+H_3 V^h_1}\},$
since in each $\theta\in(a_h,b_h)$, 
\[
\|\gamma\|_{L^1}=\sum_{\alpha=1}^n\sum_{i=1}^2 |\sigma_{i,\alpha}\xi_{\alpha}|,
\]
and the $L^1$ norm is continuous at endpoints $\theta=a_h$ and $b_h$, proved in \cite{BressanC, ColomboRisebro}.
 So by \eqref{3}, the time decay of $U$ and \eqref{UTV}, for (Problem Isothermal), we can find some constant $H_1$ in $\mathcal O(1)$, $H_2$ and $H_3$ in $\mathcal O(\exp(M))$, so
 \begin{align}\label{iso_constant}
 K_2=\exp(\mathcal O(M\,\exp( M))\big).
 \end{align}

Similarly, for (Problem $\epsilon$-BV), we can choose
\[
K_1=1
\]
and
\beq\label{iso_constant2}
K_2=4\exp(\mathcal O(\epsilon^2)),
\eeq
when $\epsilon$ is small enough,
and also using the decay of $Q$ proved in \cite{BressanC}.
\end{remark}



\subsection{The $L^1$ estimate for a weighted distance between non-shifted and shifted  $\nu$-approximate solutions}

Now we consider the weighted $L^1$ distance $d_\nu(u,\psi)$ defined in \eqref{d_def} between a front tracking $\nu$-approximate solution $u(x,t)$ with true Rankine-Hugoniot speed (allowing however for $\mathcal{O}(\nu)$ error due to the interpolation \eqref{averaged_speed}) and a shifted front tracking $\nu$-approximate solution $\psi(x,t)$ with shifted speed (see \Cref{sec:shifted_front_tracking}).  Both $u$ and $\psi$ take values in $\mathcal D $.

The following result is of fundamental importance to the proof of our Main Theorem (\Cref{main_theorem}).

\bigskip

\begin{lemma}[An estimate on shifted front tracking solutions]\label{lemma:front_tracking_error_estimate}

Consider (Problem $\epsilon$-BV) or (Problem Isothermal). Let $u$ be an $\nu$-approximate front tracking solution constructed as described in \Cref{sec:front_tracking_no_shift}. On the other hand, let $\psi$ be an $\nu$-approximate \emph{modified} front tracking solution where the shocks move according to shift functions, constructed as described in \Cref{sec:shifted_front_tracking}, with shocks moving with artificial speeds determined by functions $h_\alpha(t)$.
Then, in the sense of distributions
\begin{equation}\label{L1_deriv_estimate}
   \frac{d}{dt}d_\nu(u(\cdot,t),\psi(\cdot,t))\leq  K \sum_{\alpha=1}^{n} |\sigma_{\alpha}||{\dot h}_\alpha(t)- \dot{h}_{\text{true},\alpha}(t)|,
\end{equation}
where the  summation runs over all shocks in $\psi$ at time $t$, {and $K=K_2>0$ is the universal constant in \eqref{equiv} satisfying estimates \eqref{iso_constant} and \eqref{iso_constant2}.} 
The quantity ${\dot h}_\alpha$ is the wave speed for $\psi(x,t)$ (shifted speed). Moreover, $\dot{h}_{\text{true},\alpha}$ is the speed that this discontinuity in $\psi$ would be traveling if it was a true front tracking solution as constructed in \Cref{sec:front_tracking_no_shift}, i.e. the speed given by \eqref{averaged_speed}.
\end{lemma}

\begin{proof}
First, we fix a time $T>0$, then use the profile $\psi(x,T)$ as our initial data to find the $\nu$-approximate solution $\bar{u}$ with
\[\bar u_T(x,t) \coloneqq \mathcal S_{t-T}\circ\psi(x,T)\] 
when $t-T\in(-\delta,\delta)$, with sufficiently small $\delta$. Recall the semi-group $\mathcal S_t$ is corresponding to the $\nu$-approximate front tracking scheme without shift.

It is clear that $\bar u_T(x,t)$ is well defined, when $t>T$. And   $\bar u_T(x,t)$  and $\psi(x,t)$ have the same configuration, when $t\in(T,T+\delta)$ with sufficiently small  $\delta$. The only difference may be in the positions of the discontinuities.

To show $\bar u_T(x,t)$ is well-defined  when $t\in(T-\delta,T)$, to then obtain the desired estimate \eqref{L1_deriv_estimate}, we need to consider two cases:

\begin{center}
\begin{varwidth}{\textwidth}
\begin{enumerate}
\item[Case 1] $\psi$ has no interaction at $t=T$,
\item[Case 2] $\psi$ has an interaction at $t=T$.
\end{enumerate}
\end{varwidth}
\end{center}

\bigskip 

   \begin{figure}[tb] \centering
   		   \includegraphics[scale=.5]{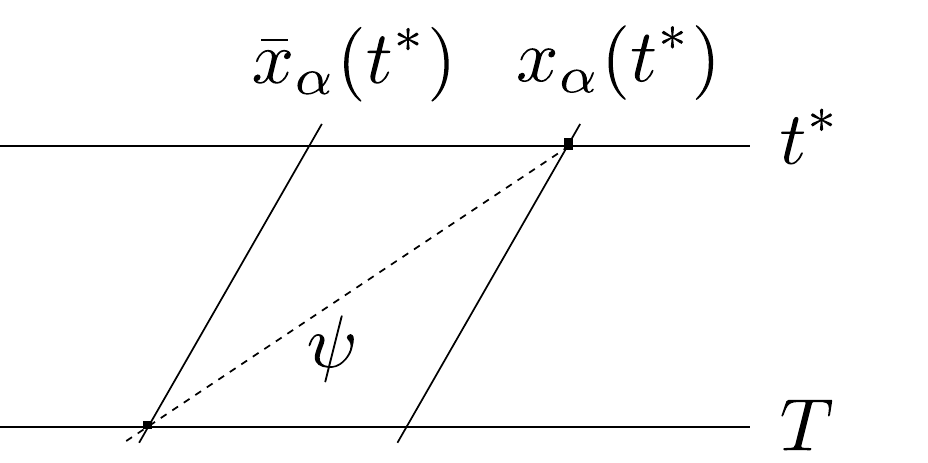}
	\caption{ \label{pic1} Two solid lines are two jumps: $\bar x_\alpha(t)$   in $\bar u_T(x,t)$ (intersects with the dotted line at $t=T$); and $ x_\alpha(t)$  in 
$\bar u_{t^*}(x,t)$ (intersects with the dotted line at $t=t^*$), with the same Rankine-Hugoniot speed. The dotted line is a shifted jump in $\psi$.    All three approximate solutions $\bar u_T(x,t)$, $\bar u_{t^*}(x,t)$ and $\psi$ share the same left and right states, along the corresponding jump $\bar x_\alpha(t)$, $x_\alpha(t)$ and the dotted line, respectively.
The distance between $x_\alpha(t^*)$ and $\bar x_\alpha(t^*)$ approaches $(t^*-T)|{\dot x}_\alpha- {\dot {\bar x}}_\alpha|$ as $t^*\rightarrow T+$.
	}
 \end{figure}
\paragraph{\bf Case 1.}
First, if there is no interaction for $\psi$ at $t=T$, i.e., all jumps are isolated, then clearly, $\bar u_T(x,t)$ is well-defined and shares the same configuration with $\psi(x,t)$, when $t\in(T-\delta,T)$ with sufficiently small $\delta$. 

Similar, for any $T<t^*<T+\delta$ with sufficiently small $\delta$, where all jumps are isolated, we can define another $\nu$-approximate solution by
\[\bar u_{t^*}(x,t)\coloneqq {\mathcal S}_{t-t^*}\circ\psi(x,t^*)\]
when $t\in(T-\delta,T+\delta)$, which shares the same configuration with $\psi(x,t)$ but with real jump speeds (up to $\nu$-error).

We use
$\gamma_{u(x,T),\psi(x,T)}$ to denote any  pseudopolygonal 
 joining $u(x,T)$ and $\psi(x,T)=\bar u_T(x,T)$.  By theories in  \cite{BressanC,ColomboRisebro} we know ${\mathcal S}_{t-T}\circ \gamma_{u(x,T),\psi(x,T)}$ is also a pseudopolygonal when $t\in[T,T+\delta)$.
 Then by Proposition \ref{control_L1_norm}, and $t^*\in(T,T+\delta)$,
\beq\label{key1}
\|\gamma_{u(x,T),\psi(x,T)}\|\geq
\|{\mathcal S}_{t^*-T}\circ\gamma_{u(x,T),\psi(x,T)}\|.
\eeq

Next we define an elementary path from $\bar u_T(x,t)$ to 
$\bar u_{t^*}(x,t)$, which both share the same configuration with $\psi(x,t)$ but with real jump speed (up to $\nu$-error), when $t\in(T-\delta,T+\delta)$, by 
\[
x^\theta_\alpha(t)={\bar x}_\alpha(t)+\xi_\alpha\theta,\qquad \theta\in[0,1],
\]
where ${\bar x}_\alpha(t)$ and $x_\alpha(t):=x^\theta_\alpha(t)|_{\theta=1}$ are any pair of corresponding {shock waves}, which are parallel, in $\bar u_T(x,t)$ and
$\bar u_{t^*}(x,t)$, respectively. {Note we do not shift rarefaction jumps.}  See Figure \ref{pic1}. We name this elementary path as $\hat \gamma(t)$. 

By \eqref{equiv}, 
\beq\label{key2}
\|\hat \gamma(t^*)\|\leq K\cdot \sum_{\alpha=1}^{n} |\sigma_{\alpha}\xi_{\alpha}| = K\cdot\sum_{\alpha=1}^{n}|\sigma_{\alpha}||x_\alpha(t^*)- {\bar x}_\alpha(t^*)|,
\eeq
{where $K=K_2$ as used in \eqref{equiv}, and where the summation runs over all shocks in $\psi$.}

At $t=t^*$, the combination of $\hat \gamma(t^*)$ and ${\mathcal S}_{t^*-T}\circ\gamma_{u(x,T),\psi(x,T)}$ gives a  pseudopolygonal joining $u(x,t^*)$ and $\bar u_{t^*}(x,t^*)=\psi(x,t^*)$. Now by \eqref{key1} and \eqref{d_def},
\beq\label{key3}
\|\gamma_{u(x,T),\psi(x,T)}\|\geq  d_\nu(u(x,t^*), \psi(x,t^*))-\|\hat \gamma(t^*)\|.
\eeq 
 Since \eqref{key3} holds for any pseudopolygonal joining $u(x,T)$ and $\psi(x,T)$,
\[
d_\nu(u(\cdot,T),\psi(\cdot,T))\geq d_\nu(u(\cdot,t^*), \psi(\cdot,t^*))-\|\hat \gamma(t^*)\|,
\]
i.e. by \eqref{key2},
\begin{eqnarray}\label{d_eqn}
&&d_\nu(u(\cdot,t^*), \psi(\cdot,t^*))-d_\nu(u(\cdot,T),\psi(\cdot,T))\nonumber\\[2mm]
&\leq&\|\hat \gamma(t^*)\|
\nonumber\\
&\leq&  K\cdot\sum_{\alpha=1}^{n} |\sigma_{\alpha}||x_\alpha-  {\bar x}_\alpha|(t^*)\nonumber\\
&=&  K\cdot \sum_{\alpha=1}^{n}\Bigg[ |\sigma_{\alpha}|\int\limits_T^{t^*}|{\dot h}_\alpha(s)- \dot{h}_{\text{true},\alpha}(s)|\,ds\Bigg].
\end{eqnarray}

From this, 
we show that
\beq\label{result}\frac{d}{dT}d_\nu(u(\cdot,T),\psi(\cdot,T))\leq  K\cdot\sum_{\alpha=1}^{n} |\sigma_{\alpha}||{\dot h}_\alpha(T)- \dot{h}_{\text{true},\alpha}(T)|,\eeq
where $\dot{     h}_\alpha$ and $\dot{h}_{\text{true},\alpha}$ are the wave speeds for $\psi(x,T)$ (shifted speed) and $\bar u_T(x,T)$ (constant Rankine-Hugoniot speed, up to error, see \eqref{averaged_speed}), respectively. See Figure \ref{pic1}. 


\bigskip

\paragraph{\bf Case 2.}

When there is an interaction between two incoming waves in $\psi(x,t)$ at $t=T$, we find an $\nu$-approximate front tracking scheme without shift
\[
\bar u_{T^+}(x,t)={\mathcal S}_{t-T}\circ\psi(x,T+)
\] 
when  $t\in(T, T+\delta)$ with some sufficiently small $\delta$ before the next interaction. Clearly, $\bar u_{T^+}(x,t)$ shares the same configuration of $\psi(x,t)$ when  $t\in(T, T+\delta)$,
with sufficiently small $\delta$.

Similarly we want to define
\[
\bar u_{T^-}(x,t)={\mathcal S}_{t-T}\circ\psi(x,T-)
\]
when  $t\in(T-\delta,T)$ with some sufficiently small $\delta$. But since $\psi(x,t)$ has shifted jumps, it is non-trivial to show that the back-in-time solution $\bar u_{T^-}(x,t)$ is well-defined when $t\in(T-\delta,T)$. Here, only two cases can happen:
\begin{itemize}
\item[Case 1] The left incoming wave is a 2-wave and the right incoming wave is a 1-wave. 
\item[Case 2] Two incoming waves are in the same family, and at least one incoming wave is a shock i.e. $\sigma<0$.
\end{itemize}
In fact, these two cases are cases from the definition of approaching waves for un-shifted $\nu$-approximate solution (see \Cref{sec:front_tracking_no_shift}), where a pair of waves in the following two cases are \emph{not approaching}: (i) the left incoming wave is a $1$-wave and the right incoming wave is a $2$-wave, (ii) two incoming waves are in the same family and both are rarefactions.

Now we consider the shifted $\nu$-approximate solution (see \Cref{sec:shifted_front_tracking}). Since we do not shift rarefaction waves and always use the rarefaction curve when $\sigma>0$ in \eqref{Riemann1_0}, case (ii) is still a non-approaching case. Case (i) is a non-approaching case, due to our construction of shift for a shock, i.e. a shifted 1-shock to the left will not interact with any 2-waves, and similarly for the symmetric case. See \eqref{ordering_shifts}.


In both Case 1 and Case 2, the Rankine-Hugoniot speed of the left incoming wave is always larger than the speed of the right incoming wave due to Lax's entropy condition. 

Hence, $\bar u_{T^-}(x,t)$ shares the same configuration of $\psi(x,t)$ when  $t\in(T-\delta,T)$, with sufficiently small $\delta$. Still using the Lax entropy condition and \eqref{ordering_shifts}, similar result holds for interactions of more than 2 waves. In fact, for any two adjacent incoming waves in the same family, at least one of them is a shock, otherwise, they won't interact. So
the Rankine-Hugoniot speed of the left wave is always larger than the speed of the right wave due to Lax's entropy condition.

Then  by Proposition \ref{control_L1_norm},
\[
d_\nu(u(x,T+\delta_1),\bar u_{T^+}(x,T+\delta_1))\leq d_\nu(u(x,T-\delta_2),\bar u_{T^-}(x,T-\delta_2))
\]
for any sufficiently small positive $\delta_1$ and $\delta_2$.
Letting  $\delta_1$ and $\delta_2$ approach zero, we know
\beq\label{d_con_2}
\limsup_{t\rightarrow T+}d_\nu(u(x,t),\psi(x,t))\leq \lim_{t\rightarrow T-}d_\nu(u(x,t), \psi(x,t)).
\eeq
where \eqref{result} guarantees the existence of one-side limit on the right hand side of the inequality. The case when more than two waves interact can be treated similarly.
\end{proof}
\begin{remark}\label{remark2}
For future use, if we change the speeds of shocks \emph{as well as}  rarefactions, the error estimate \eqref{L1_deriv_estimate} still holds with $\sigma_\alpha$ corresponding to any jump, under the condition that the change of speeds do not cause the interaction between two adjacent rarefactions in the same family, or interaction between a 1-wave on the left and a 2-wave on the right.
\end{remark}

\section{The inhomogeneous weight $a$ on the spatial domain}\label{sec:weight}
In this section, we construct the inhomogeneous weight $a(x,t)$ which is needed to ensure stability in a weighted-type $L^2$ space via \Cref{prop:small_shock_diss} and \Cref{prop:large_shock_diss}.

We define the weight function $a$ for small-BV solutions (Problem $\epsilon$-BV) and for large-BV solutions for special systems like (Problem Isothermal), for $\nu$-approximate front tracking solutions with shift.

\textbf{Important note on notation:} We continue to follow the notation used in the previous section, \Cref{section:front}.

\subsection{Wave interactions}
Now we recall some basic interaction estimates for $\nu$-approximate solutions, for both small data general systems (Problem $\epsilon$-BV) and large data special systems such as (Problem Isothermal). 

Note most interaction estimates in this subsection are provided in \cite{BressanC} and \cite{ColomboRisebro} for un-shifted $\nu$-approximate solutions,
but also hold for shifted $\nu$-approximate solutions (\Cref{sec:shifted_front_tracking}), because the results come from the comparison of incoming and outgoing wave strengths, without measuring the wave speed. So all results in this subsection hold for both shifted or un-shifted solutions.

%

A $j$-wave and an $i$-wave, with the former crossing the $t$-time line to the left of the latter, are called {\it approaching} when either $i<j$, or $i=j$ and at least one of these waves is a shock (a wave with negative $\sigma$).

From now on, we may abuse the notation by referring to the wave $\sigma_i$ by referring to its strength.



For small data problem,  we can define the weight $a$ the same as in  \cite{MR4487515}, for the  $\nu$-approximate solution. Here we will use a new way to define the weight $a$ to make it consistent with the large data problem.



\begin{proposition}\label{prop:delta}
Consider (Problem $\epsilon$-BV). Then for a solution in $\mathcal D$, there exists $\kappa>0$ such that for any $\eps$ small enough, the functionals $U(t)$ and $Q(t)$ are decreasing in time. Moreover, for any time $t$ where waves with strength $|\sigma_i|$ and $|\sigma_j|$ interact, we have
$$
\Delta U(t)\leq -\frac{\kappa}{2}|\sigma_i| |\sigma_j|.
$$

Consider now (Problem Isothermal). Then for a solution  with finite total variation  in $\mathcal D$, there exists positive  $\kappa_2$ and $\eta$ such that the functionals $U(t)$ and $Q(t)$ are decreasing in time. 
\end{proposition}
\vskip0.3cm

The following interaction estimate holds for  $\nu$-approximate solutions of small data general systems \cite{BressanC}. 


\begin{proposition}{\protect{\cite{BressanC}}}\label{prop4.2}
We consider $\nu$-approximate solutions  in the context of (Problem $\epsilon$-BV).
When there is an interaction at $(x,t)$ between incoming waves  from the second family with strength $\sigma''_i$'s and incoming waves from the first family with strength $\sigma'_i$'s, we denote the strength of the outgoing wave in the first family as $\sigma_1$ and in the second family as $\sigma_2$, and the following estimates hold:
 \begin{equation}\label{est1}
|\sigma_2-\sum\sigma''_i|+|\sigma_1-\sum\sigma'_i|\leq K_0|\Delta Q|,
\end{equation}
for some constant $K_0$ independent of $\kappa$,
where $\Delta Q=Q(t+)-Q(t-)$.

Furthermore, we have
\begin{itemize}
\item[(i)]
In any family, if there are two or more incoming rarefaction waves, they cannot be adjacent.
\item[(ii)]
If, in any family, all incoming waves are shocks, i.e. with negative $\sigma$, then the outgoing wave in that family is still a shock.

\item[(iii)]
If the incoming waves in the first family include both rarefaction and shock, then 
\begin{equation}\label{est2_0}
\Big(\sum|\sigma'_i|\Big)-|\sigma_1|>\min\Big\{ \sum(\sigma'_i)_-,\sum(\sigma'_i)_+\Big\}.
\end{equation}
\end{itemize}
where recall we use the notation
\[
(f)_\pm=(|f|\pm f)/2
\]
to denote the positive and negative parts of any function $f$, respectively.
\item[(iv)]
If the second family of incoming waves include both rarefaction and shock, then 
\begin{equation}\label{est2}
\Big(\sum|\sigma''_i|\Big)-|\sigma_2|>
\min\Big\{ \sum(\sigma''_i)_-,\sum(\sigma''_i)_+\Big\}.
\end{equation}

\end{proposition}

The estimates \eqref{est2_0}, \eqref{est2} mean that when there are both incoming rarefaction and shock in any family, there will be cancellation, which causes the total variation to decay at least the total strength of waves in the weaker type (rarefaction or shock).


\begin{proof}
The estimate \eqref{est1} has been proved for small data problem in Lemma 5 in \cite{BressanC}, which directly gives (ii). The item (i) holds since rarefactions in the same family do not interact.

Now we show \eqref{est2_0} by \eqref{est1}. The proof of \eqref{est2} is entirely similar.

Without loss of generality, assume that $\sigma_1>0$, then by \eqref{est1},
\beq\label{proof_lemma}
-K_0|\Delta Q|\leq \sum\sigma'_i-\sigma_1=\sum(\sigma'_i)_+-\sum(\sigma'_i)_--|\sigma_1|\leq K_0|\Delta Q|.
\eeq
So we have
\[
(\sum|\sigma'_i|)-|\sigma_1|=(\sum|\sigma'_i|)-(\sum(\sigma'_i)_+-\sum(\sigma'_i)_-)+\sum(\sigma'_i)_+-\sum(\sigma'_i)_--|\sigma_1|
\geq 2\sum(\sigma'_i)_--K_0|\Delta Q|.
\]
Since $\sum(\sigma'_i)_->K_0|\Delta Q|$ for small data solution,
\[
(\sum|\sigma'_i|)-|\sigma_1|\geq \sum(\sigma'_i)_-.
\]
On the other hand, again by \eqref{proof_lemma},
\[
\sum(\sigma'_i)_+\geq \sum(\sigma'_i)_-+|\sigma_1|- K_0|\Delta Q|\geq \sum(\sigma'_i)_-.
\]
So we prove \eqref{est2_0}. When $\sigma_1<0$, we can prove it similarly.

The idea of this proof can be summarized as the following. When the incoming $1$-waves include both rarefaction and shock, the total variation of waves in that direction decays almost twice the total variation of incoming waves in the weaker type (1-rarefaction or 1-shock), since $Q$ is much smaller than $V$. So \eqref{est2} holds. The $2$-wave case is entirely similar.
\end{proof}

For (Problem Isothermal), we only consider pairwise interactions. We divide pairwise interactions into two classes: \emph{overtaking interactions} when two incoming waves are from the same family, and \emph{head-on interactions} when two incoming waves are from different families.

For overtaking interactions, we call the direction or family of incoming waves as the major direction or family, and call the other direction as the reflected direction or family. The outgoing wave(s) in the reflected direction is called the reflected wave(s).

\begin{itemize}
\item 
For any head-on interaction, by Lemma 3.2 in \cite{ColomboRisebro}, in each family,
there is one outgoing wave, whose corresponding $\sigma$ defined in \eqref{sigma_def} is the same as $\sigma$ for the incoming wave in the same family.
\item
For overtaking interactions of Temple system, since $\phi=0$ in Lemma \ref{lem_3.1}, there is no reflecting wave.

\item 
For overtaking shock-shock interaction of isothermal gas dynamics, the reflected wave is a rarefaction. And the strength of outgoing shock in the major family is larger than the strength of each incoming shock. The proof is given in Appendix \ref{app_A3}.
\item For overtaking shock-rarefaction interaction of isothermal gas dynamics, there is a reflected wave, but the total variation of waves decays \cite{ColomboRisebro}.
\end{itemize}

Summarizing above results and other results which can be proved using the proof of Lemma 3.1 and 3.2 in \cite{ColomboRisebro}, we have the following Proposition.

\begin{proposition}\label{prop4.2_2}\cite{ColomboRisebro}
We consider pairwise interactions of (Problem Isothermal).

When there is a pairwise interaction at $(x,t)$ between two incoming waves  with strength $\sigma'$ and $\sigma''$, we denote the strength of outgoing wave in the first family as $\sigma_1$ and  in the second family as $\sigma_2$. Then the following estimates hold, when $\eta\in(0,1)$ is sufficiently small and then $\kappa_2$ is sufficiently large.
\begin{itemize} 
\item[1.] 
For all pairwise interactions,
\[
\Delta U<-|\sigma'\sigma''|.
\]
The total variation $V=\sum_{i,\alpha}|\sigma_{i,\alpha}|$ is not increasing from before to after interaction. 
\item[2.]
When two incoming waves are from two different families ($\sigma'$ for first family, $\sigma''$ for the second family), then they will cross without changing the type and strength, i.e. $\sigma'=\sigma_1$ and $\sigma''=\sigma_2$.

\item[3.] When two shocks from the same family interact, the rarefaction wave is reflected for the isothermal gas dynamics. 
And for any fixed $\eta>0$
\[
\Delta V_2=-2\eta |\sigma_i|,
\]
and 
\beq\label{Uk0}
\Delta U\leq -|\sigma_i|
\eeq
when $\kappa_2$ is sufficiently large,
where the subscript $i$ is for the reflected direction. 

The strength of the outgoing shock in the major family is larger than the strength of each incoming shock.

\item[4.]
When a 1-shock ($\sigma'$) interacts with a 1-rarefaction ($\sigma''$), we know $\Delta V_2<0$.
If the outgoing $2$-wave is a shock, its strength is less than $\nu$, where the strength of the incoming $1$-rarefaction is always less than $\nu$ by our construction. And for any  positive constant $K_1$,
 \begin{equation}\label{est2_2}
 |\sigma_2|\leq -K_1\Delta U
\end{equation}
when $\kappa_2$ is sufficiently large. 

If the outgoing 1-wave is a shock, its strength is less than the incoming 1-shock, and for any  positive constant $K_1$,
 \begin{equation}\label{est2_2_2}
0<|\sigma'|-|\sigma_1|\leq -K_1\Delta U,
\end{equation}
when $\kappa_2$ is sufficiently large. 

If the outgoing 1-wave is a rarefaction, for any  positive constant $K_1$,
 \begin{equation}\label{est2_2_3}
|\sigma'|\leq -K_1\Delta U,
\end{equation}
when $\kappa_2$ is sufficiently large. 

\item[5.]
When a 2-shock ($\sigma''$) interacts with a 2-rarefaction ($\sigma'$), we know $\Delta V_2<0$.
If the outgoing $1$-wave is a shock, its strength is less than $\nu$, where the strength of the incoming $2$-rarefaction is always less than $\nu$ by our construction.
And if the outgoing $2$-wave is a shock, then for any  positive constant $K_1$,
\begin{equation}\label{est2_3}
|\sigma_1|\leq -K_1\Delta U,
\end{equation}
when $\kappa_2$ is sufficiently large.

If the outgoing 2-wave is a shock, its strength is less than the incoming 2-shock, and for any  positive constant $K_1$,
 \begin{equation}\label{est2_3_2}
0<|\sigma''|-|\sigma_2|\leq -K_1\Delta U,
\end{equation}
when $\kappa_2$ is sufficiently large.

If the outgoing 2-wave is a rarefaction, then for any  positive constant $K_1$,
 \begin{equation}\label{est2_3_3}
|\sigma'|\leq -K_1\Delta U,
\end{equation}
when $\kappa_2$ is sufficiently large.
\end{itemize}
\end{proposition}
\begin{remark}
For Item 3, there is no reflected wave for Temple systems. By \eqref{def_v2} and \eqref{def_v1}, $V_2=V$ when $\eta=0$. One can get the estimates on $V$ in Proposition \ref{prop4.2_2} by applying the corresponding proof in Lemma 3.1 of \cite{ColomboRisebro} to the case when $\eta=0$. 
\end{remark}

\subsection{The weight function $a$ }

Now we start to define the weight $a(x,t)$ for a $\nu$-approximate shifted front tracking solution $\psi$ constructed as in \Cref{sec:shifted_front_tracking}.
%
%


\paragraph{\bf (Problem $\epsilon$-BV)}
We provide a new construction of the weight $a(x,t)$, which is well-defined and decaying in time after any interaction of two and more than two waves.

For the small data problem, for any shock wave in $\psi$, we can always set
\beq\label{geng1}
\frac{a_r}{a_l}=e^{-\frac{3C_1}{4} |\sigma|} \hbox{ for 1 wave, and }
\frac{a_r}{a_l}=e^{\frac{3C_1}{4} |\sigma|}\hbox{ for 2 wave},
\eeq
where the subscripts $r$ and $l$ denote the values of $a$ to the right and left  of the shock, respectively. The constant $C_1$ is from \Cref{prop:small_shock_diss}. In the context of \Cref{prop:small_shock_diss}, $a_2=a_r$ and $a_1=a_\ell$.

When the wave strength is small enough,
\[
e^{\frac{3C_1}{4}|\sigma|} \approx 1+ \frac{3C_1}{4} |\sigma|,
\]
and the conditions on $a$ in Proposition \ref{prop:small_shock_diss} are satisfied.
%



The weight function is  defined as following
\beq\label{adef}
a(x,t)\coloneqq \exp\Bigg({\frac{3C_1}{4}\Big(V(t)+  \frac{3\kappa}{2}  Q(t) - \sum_{1\text{-shock in } (-\infty,x)} |\sigma_i| +\sum_{ 2\text{-shock in } (-\infty,x)} |\sigma_i| \Big)}\Bigg),
\eeq
for a sufficiently large $\kappa$.

Note that the function $a$ is piecewise-constant. It has discontinuities only along shock curves of $\psi$. On the other hand, it is constant across rarefaction curves of $\psi$.

We now state some important properties of the function $a$.

\begin{proposition}[\protect{Properties  of the weight for (Problem $\epsilon$-BV) \cite[Proposition 6.2]{MR4487515}}]
\label{prop:small_BV_a}
\hfill

For a fixed $\nu$-approximate shifted front tracking solution $\psi$ for (Problem $\epsilon$-BV) constructed as in \Cref{sec:shifted_front_tracking}, the weight function \eqref{adef} verifies the following properties.

For every time without wave interaction, and for every $x$ such that a 1-shock $\sigma_\alpha$ is located at $x = x_\alpha(t)$ in $\psi$,
\begin{equation}\label{small_a_1shock_control}
1 - 2C_1|\sigma_\alpha| \leq \frac{a( x_\alpha(t)+,t)}{a( x_\alpha(t)-,t)} \leq 1 - \frac{C_1}{2}|\sigma_\alpha|.
\end{equation}

For every time without wave interaction, and for every $x$ such that a 2-shock $\sigma_\alpha$ is located at $x = x_\alpha(t)$ in $\psi$,
\begin{equation}\label{small_a_2shockcontrol}
1 + \frac{1}{2}C_1|\sigma_\alpha| \leq \frac{a(x_\alpha(t)+,t)}{a(x_\alpha(t)-,t)} \leq 1 + 2C_1|\sigma_\alpha|.
\end{equation}

For every time $t$ with a wave interaction, and almost every $x$, the weight function decays in time:
\begin{equation}\label{small_a_decay}
a(x,t+) \leq a(x,t-),
\end{equation}
when $\kappa$ is sufficiently large. This holds except at finitely many points $(x,t)$ where interactions happen.

Finally, we have 
\beq\label{small_a_bounds}
\norm{a}_{L^\infty(\mathbb{R}\times[0,T])}\leq K,\qquad
\norm{\frac{1}{a}}_{L^\infty(\mathbb{R}\times[0,T])} \leq K,
\eeq
for a universal constant $K>0$.
\end{proposition}
\begin{proof}

\textbf{Step 1}

Under this definition \eqref{adef}, clearly \eqref{geng1} is satisfied on any isolated shock wave except at the interaction point. More precisely, we can equivalently write the weight function as
\beq\label{adef2}
a(x,t)=e^{\frac{3C_1}{4}\Big(V(t)+ \frac{3\kappa}{2} Q(t) +\int_{-\infty}^x \mu(x,t)dx \Big)}
\eeq
where we use the measure $\mu(\cdot,t)$ as a sum of Dirac measures in $x$:
\begin{eqnarray*}
&&\mu(x,t)= -\sum_{i: 1\text{-shock}} |\sigma_i| \delta_{\{x_i(t)\}}+\sum_{i: 2\text{-shock}} |\sigma_i| \delta_{\{x_i(t)\}}.
\end{eqnarray*}

This gives \eqref{small_a_1shock_control} and \eqref{small_a_2shockcontrol}.


\textbf{Step 2}

We now show \eqref{small_a_decay}.

When there is no wave interaction, $a(x,t)$ is constant in any region on the $(x,t)$-plane between two jumps.

Now we consider the time when there is a wave interaction. 

We assume there is only one interaction point $(x_0,t)$ at time $t$. Otherwise, if there are more than one, then we can apply the following proof to all of the interaction points at time $t$.

Recall, we use $(f)_\pm$ to denote the $(|f|\pm f)/2$, i.e. the positive/negative part of $f$. And for function $f(t)$, we denote $\Delta f=f(t+)-f(t-)$.

If 1-waves $\sum\sigma'_i$  and  2-waves $\sum\sigma''_i$ interact at $(x_0,t)$, then 
\begin{eqnarray}
&&\mu(t+)-\mu(t-)=\delta_{\{x_0\}}\left((\sigma_2)_--(\sigma_1)_--\sum(\sigma''_i)_- +\sum(\sigma'_i)_-\right)\nonumber\\
&&\qquad \leq \delta_{\{x_0\}}\left( |(\sigma_2)_--\sum(\sigma''_i)_-|+|(\sigma_1)_--\sum(\sigma'_i)_-|\right)\label{est3}.
\end{eqnarray}
Then we will use Proposition \ref{prop4.2} to show that
\beq\label{est3_2} 
\eqref{est3} \leq \Big((\Delta V)_-+\frac{\kappa}{2} |\Delta Q| \Big)\delta_{\{x_0\}},
\eeq
when $\kappa$ is large enough. Here $\Delta V=V(t+)-V(t-)$, similarly for $\Delta Q$.

Here when \eqref{est3_2} holds, by \eqref{adef2}, \eqref{est3}, and $V(t)+\kappa Q(t)$ and $Q(t)$ both decaying in time, we know  
$$
 \frac{a(x,t+)}{a(x,t-)}\leq 1,
\qquad
\hbox{i.e.}\qquad
 a(x,t+)-a(x,t-)\leq 0,
$$
for any $x\in \mathbb R$ except at the point $(x_0,t)$ where an interaction happens.

\bigskip

Finally, we come to prove \eqref{est3_2}. Let's first consider $|(\sigma_2)_--\sum(\sigma''_i)_-|$.
\bigskip
\\{\bf Case 1:} 
If all incoming 2-waves are the same type (shock or rarefaction) then by Proposition \ref{prop4.2}, the outgoing 2-wave does not change type. If 2-waves are all shocks, by \eqref{est1},
\[
|(\sigma_2)_--\sum(\sigma''_i)_-|=
|\sigma_2-\sum\sigma''_i|\leq K_0 |\Delta Q|.
\]
If 2-waves are all rarefactions
\[
|(\sigma_2)_--\sum(\sigma''_i)_-|=0.\]
\bigskip
\\
{\bf Case 2:} 
Otherwise, suppose there are both incoming 2-rarefaction and 2-shock waves. 

If $\sigma_2\geq 0$, i.e. the outgoing 2-wave is a rarefaction, then by \eqref{est1},
\[
 K_0|\Delta Q|\geq |\si_2-\sum\sigma''_i|=
|\si_2-\sum(\sigma''_i)_++\sum(\sigma''_i)_-|\geq \si_2-\sum(\sigma''_i)_++\sum(\sigma''_i)_-,
\]
i.e.
\[
\sum(\sigma''_i)_-
\leq
K_0|\Delta Q|-\si_2+\sum(\sigma''_i)_+
\leq K_0|\Delta Q|+\sum(\sigma''_i)_+,
\]
so
\[
\sum(\sigma''_i)_+\geq \sum(\sigma''_i)_- -K_0|\Delta Q|.
\]
And clearly
\[
\sum(\sigma''_i)_-> \sum(\sigma''_i)_- -K_0|\Delta Q|,
\]
so
\[
\min\{ \sum(\sigma''_i)_-,\sum(\sigma''_i)_+\}\geq  \sum(\sigma''_i)_--K_0|\Delta Q|.
\]
So by \eqref{est2},
\begin{eqnarray}
&&\hspace{-.4in}|(\sigma_2)_--\sum(\sigma''_i)_-|\nonumber\\
&=&|\sum(\sigma''_i)_-|\nonumber\\
&\leq&\min\{ \sum(\sigma''_i)_-,\sum(\sigma''_i)_+\}+K_0|\Delta Q|\nonumber\\
&\leq&(\sum|\sigma''_i|)-|\sigma_2|+K_0|\Delta Q|.\label{4.12}
\end{eqnarray}
\bigskip

If $\sigma_2\leq 0$,  i.e. when the outgoing 2-wave is a shock,
\[
 K_0|\Delta Q|\geq |\si_2-\sum\sigma''_i|=
|\si_2-\sum(\sigma''_i)_++\sum(\sigma''_i)_-|\geq \sum(\sigma''_i)_+-\si_2-\sum(\sigma''_i)_-.
\]
So it is clear that
\begin{align}\label{step1}-\sigma_2-\sum(\sigma''_i)_-\leq K_0|\Delta Q|-\sum(\sigma''_i)_+\leq K_0|\Delta Q|.
\end{align}
Then
\begin{eqnarray}
&&\hspace{-.4in}|(\sigma_2)_--\sum(\sigma''_i)_-|\nonumber\\
&=&|-\sigma_2-\sum(\sigma''_i)_-| \nonumber\\
&\leq&\sigma_2+\sum(\sigma''_i)_-+2K_0|\Delta Q| \nonumber\\
&=&\sum(\sigma''_i)_--|\sigma_2|+2K_0|\Delta Q|  \nonumber\\
&\leq&(\sum|\sigma''_i|)-|\sigma_2|+2K_0|\Delta Q| ,\label{4.13}
\end{eqnarray}
where we used \eqref{step1} in the first inequality.
\bigskip

For 1-wave we have a similar estimate for $|(\sigma_1)_--\sum(\sigma'_i)_-|$. For completeness, we delineate the cases.
\\{\bf Case 1':} 
If all incoming 1-waves are shocks then from \eqref{est1} we have,
\[
|(\sigma_1)_--\sum(\sigma'_i)_-|\leq K_0 |\Delta Q|.
\]
If all incoming 1-waves are rarefactions,
\[
|(\sigma_1)_--\sum(\sigma'_i)_-|=0.\]
\\
{\bf Case 2':} 
Otherwise, suppose there are both incoming 1-rarefaction and 1-shock waves. 

If $\sigma_1\geq 0$, i.e. the outgoing 1-wave is rarefaction, then we have in analogy with \eqref{4.12}
\beq|(\sigma_1)_--\sum(\sigma'_i)_-|\nonumber
\leq(\sum|\sigma'_i|)-|\sigma_1|+K_0|\Delta Q|.\label{4.12'}
\eeq

If $\sigma_1\leq 0$,  i.e. when the outgoing 1-wave is a shock, then we have in analogy with \eqref{4.13}
\begin{align*}
|(\sigma_1)_--\sum(\sigma'_i)_-|\leq(\sum|\sigma'_i|)-|\sigma_1|+2K_0|\Delta Q| .
\end{align*}

Now, choosing $\kappa>8K_0$ and recalling \eqref{def_v2}, we prove \eqref{est3_2}.

\textbf{Step 3}

With the uniform bounds of $V$ and $Q$, which can be obtained by the time decay of $U$ and $Q$ in Proposition \ref{prop:delta}, we can immediately get a pair of uniform finite $L^\infty$ bounds on $a(x,t)$ and $1/a(x,t)$. This gives \eqref{small_a_bounds}.
\end{proof}


\paragraph{\bf (Problem Isothermal)} 
Now, we address the (Problem Isothermal).

\begin{proposition}[\protect{The weight for (Problem Isothermal)}]
\label{prop:iso_a}

For a fixed $\nu$-approximate shifted front tracking solution $\psi$ for (Problem Isothermal) constructed as in \Cref{sec:shifted_front_tracking}, there exists a weight function $a(x,t)$ which verifies the same properties as the construction for the (Problem $\epsilon$-BV), i.e. \eqref{small_a_1shock_control}, \eqref{small_a_2shockcontrol}, \eqref{small_a_decay}, and \eqref{small_a_bounds}.
\end{proposition}

Cf. \cite[Section 6]{Cheng_isothermal}, where a construction of the weight $a$ is given for isothermal Euler using a different front tracking scheme. 

\subsubsection{Proof of \Cref{prop:iso_a}}
\hfill
\vspace{.1in}

\paragraph{\bf Step 1.}

We want to pre-set the value of $a(-\infty,t)$ for all $t$, then adjust it later when needed. We set 
\beq\label{aninf}\lim_{x\rightarrow -\infty} a(x,t):=e^{\frac{3C_1}{4}U},
\eeq
where $U$ (see \eqref{def_v1}) decays in time.

At later times, we will decrease the value of $\lim_{x\rightarrow -\infty} a(x,t)$ after some interactions as needed. The detail will be given.

\bigskip

\paragraph{\bf Step 2.}
We define $a(x,0)$ from left to right on the $x$-axis, starting from the left-most constant state $\lim_{x\rightarrow -\infty} a(x,0)$. We do not change the value of $a(x,0)$ until meeting the next shock (a jump with $\sigma<0$) in the scheme.

Once there is a shock at $x$, then we define $a(x+,0)$ by $a(x-,0)$, using the following rule on the ratio $a(x+,0)/a(x-,0)$:
\begin{itemize}
\item[Rule 1:] If the strength $|\sigma|$ of shock at $x$ is larger than $\epsilon$, set the ratio $a(x+,0)/a(x-,0)$ as $1/a^*$ for 2-shock or $a^*$ for 1-shock. Here, $a^*$ is playing the role of $a$ in \eqref{eq:large_shock_diss}.
\item[Rule 2:] If the strength $|\sigma|$ of shock at $x$ is less than $\epsilon$, set the ratio $a(x+,0)/a(x-,0)$ by $e^{\frac{3C_1}{4}|\sigma|}$ for 2-shock or $e^{-\frac{3C_1}{4}|\sigma|}$  for 1-shock, which is consistent with \eqref{geng1},
\end{itemize}
where $\epsilon$ is a fixed sufficiently small number, such that if $|\sigma|\leq \epsilon$ for a $1$- or $2$-shock, then Proposition \ref{prop:small_shock_diss} holds (in particular, let $\epsilon$ be from \eqref{small_shock_must_satisfy}).

\bigskip

\paragraph{\bf Step 3.}
Once we define $a(x,0)$, we define $a(x,t)$ until the first wave interaction, by keeping $a$ as a constant in the region on the $(x,t)$-plane between any two jumps in the scheme.

\bigskip

\paragraph{\bf Step 4.}

Next, we consider wave interactions. We only have to consider the case when a pairwise interaction happens at $(x_0,t_0)$. If more than two waves interact at one point, we can slightly change the wave speeds satisfying the condition in Remark \ref{remark2}, 
to change it to pairwise interactions. By \eqref{result} and Remark \ref{remark2}, the errors of the weighted $L^1$ norm caused by such changes will go to zero as $\nu$ goes to zero if we change the speeds only slightly, since the total number of jumps and their interactions is finite and can be estimated by a function of $\nu$, see Lemma 9 in \cite{BressanC} (for any $u_{\pm \infty}$, when $\nu$ is small enough, the initial condition of Lemma 9 in \cite{BressanC}  is satisfied). 

As in \Cref{sec:estimates}, we name the shock with strength larger than $\epsilon$ as ``large'' shock, otherwise ``small'' shock.

The following construction heavily relies on Proposition \ref{prop4.2_2}. So we will not mention it every time we use it.

\smallskip

We still use Rule 1 and 2 to define $a(x+,t_0+)$ by $a(x-,t_0+)$, starting from $\lim_{x\rightarrow -\infty} a(x,t_0+)$. But to guarantee that $a(x,t_0+)\leq a(x,t_0-)$ for any $x\neq x_0$, we need to drop the values of all $a(x,t_0+)$ including $\lim_{x\rightarrow -\infty} a(x,t_0+)$ for some special type of interactions.

Recall, 
\beq\label{mono}
\frac{a_r}{a_l}<1\quad\hbox{for 1-shock, and}\quad \frac{a_r}{a_l}>1\quad \hbox{for 2-shock,}
\eeq
where the subscripts $r$ and $l$ denote the right and left states of the shock, respectively.

First, when all incoming and outgoing shock waves are small, we do not drop $a(\cdot,t_0+)$. Then we show the decay of $a$ at $t_0$.
\begin{itemize}
\item For head-on interactions, $\sigma$ does not change in each direction, so $a(x_0+,t_0+)<a(x_0+,t_0-)$. So $a$ decays in time at $t_0$ except at $x=x_0$.
\item For the overtaking shock-shock interaction, we denote two incoming small shocks as $\sigma'$ and $\sigma''$, and the outgoing waves in the first and second families as $\sigma_1$ and $\sigma_2$, respectively.

Let's first assume $\sigma'$ and $\sigma''$ are both $2$-shocks.

For Temple system, $\sigma_2=\sigma'+\sigma''$ and $\sigma_1=0$.

For isothermal Euler, a rarefaction wave is reflected, so $\sigma_2<0$ and $\sigma_1>0$.
And $\rho$ is monotonic on $x$ before and after an overtaking shock-shock interaction (see \cite{ColomboRisebro}). By \eqref{si_iso} 
we know
\[\sigma_2>-\sigma_1+\sigma_2=\sigma'+\sigma''\]
so
\[|\sigma_2|<|\sigma'|+|\sigma''|.\] By Rule 2, $a(x_0+,t_0+)<a(x_0+,t_0-)$. So $a$ decays in time at $t_0$ except at $x=x_0$.

When $\sigma'$ and $\sigma''$ are both $1$-shocks, for isothermal Euler, 
$\rho$ is monotonic on $x$ before and after an overtaking shock-shock interaction \cite{ColomboRisebro}. By \eqref{si_iso}, we know
\[-\sigma_1+\sigma_2=-\sigma'-\sigma'',\]
where $\sigma_1,\sigma'$ and $\sigma''$ are negative and $\sigma_2$ is positive. So
\[
|\sigma_1|-|\sigma'|-|\sigma''|=-|\sigma_2|.
\]
So by Rule 2 and \eqref{Uk0}, we know $a(x_0+,t_0+)<a(x_0+,t_0-)$, when $\kappa_2$ is sufficiently large. So $a$ decays in time at $t_0$ except at $x=x_0$.
\item For shock-rarefaction overtaking interactions, by Item 4 and 5 in Proposition \ref{prop4.2_2} and Rule 2, $a(x_0+,t_0+)<a(x_0+,t_0-)$, where especially \eqref{est2_2}-\eqref{est2_2_3} 
show the increasing part of $a$ coming from Rule 2 are all controlled by the decreasing part of $a$ from $\Delta U$, by choosing $\kappa_2$ large enough. So $a$ decays in time at $t_0$ except at $x=x_0$.
\end{itemize}


Then we consider other cases including incoming or outgoing large shock(s). 

\paragraph{(i).}

If the interaction is a head-on interaction, waves just cross each other without changing their strength. We do not drop $a(\cdot,t_0+)$.  The ratios of $a$ for the first (second) incoming and outgoing waves are the same.

After this interaction, $a$ keeps the same value as before the interaction, except in the region between two outgoing waves of the interaction.

\smallskip

\paragraph{(ii).}

For overtaking interaction of two shocks, the strength of outgoing wave in the major family 
is not larger than the sum of strengths of two incoming ones, but larger than the strength of each incoming shock. The reflected wave is rarefaction for the isothermal gas. There is no reflected wave for Temple system.

We first consider interactions of 2-shock and 2-shock, involving incoming or outgoing large shock(s). 
\begin{itemize}
\item If the outgoing 2-shock is large, while the incoming 2-shocks are both small, then we
divide $a(\cdot,t_0+)$ at every state by $1/a^*$. By \eqref{mono}, $a$ decays in time at $t_0$ except at $x=x_0$. There are at most finitely many (depending on $M/\epsilon$) of these kind of interactions, see  Lemma 5.3 in \cite{Cheng_isothermal} (where Cheng shows that the number of interactions in his cases 4B and 9B are finite). 

\item
Note if at least one incoming 2-shock is large, then the outgoing 2-shock is also large. In this case, we do not drop $a(\cdot,t_0+)$. By \eqref{mono}, $a(x_0+,t_0+)<a(x_0+,t_0-)$. So $a$ decays in time at $t_0$ except at $x=x_0$.
\end{itemize}

We consider interactions of 1-shock and 1-shock, where a 2-rarefaction is reflected, involving incoming or outgoing large shock(s).
\begin{itemize}
\item If the outgoing 1-shock is large and two incoming 1-shocks are both large, we
divide $a(\cdot,t_0+)$ at every state by $1/a^*$. By \eqref{mono}, $a$ decays in time at $t_0$. There are at most finitely  many ($\mathcal{O}(M\epsilon^{-2})$) of these kind of interactions, since $U$ decays at least by $\mathcal{O}(\epsilon^2)$ at each interaction by Proposition \ref{prop4.2_2}.

\item If the outgoing 1-shock is large, while one incoming 1-shock is large and the other one is small with strength $|\sigma|$, then we
divide $a(\cdot,t_0+)$ at every state by $e^{\frac{3C_1}{4} |\sigma|}$. Then $a$ decays in time at $t_0$ except at $x=x_0$. For each of such interaction,
$U$ decays at least by $\epsilon \sigma$ by Proposition \ref{prop4.2_2}. Although there might be infinitely many such kind of interactions, the total ratio drop caused by such kind of interactions is finite. 

\item Note if at least one incoming 1-shock is large, then the outgoing 1-shock is also large.

So the last case is, when the outgoing 1-shock is large, while two incoming 1-shocks are both small. In this case, we do not drop $a(\cdot,t_0+)$. By \eqref{mono}, $a(x_0+,t_0+)<a(x_0+,t_0-)$. So $a$ decays in time at $t_0$ except at $x=x_0$.
\end{itemize}

\paragraph{(iii).}
Now we consider the overtaking interaction of a shock and a rarefaction in the same family, involving incoming or outgoing large shocks. The reflected wave is either a small shock or rarefaction for the isentropic gas, by Proposition \ref{prop4.2_2} and $\nu<\epsilon$. There is no reflected wave for Temple system.

We first consider interactions of 2-shock and 2-rarefaction. We do not drop $a(\cdot,t_0+)$ in this case. Note if the outgoing 2-wave is still a shock, it will be weaker than the incoming one. So $a(x,t)$ always decays at $t_0$ by \eqref{mono}.

Finally, we consider interactions of 1-shock and 1-rarefaction.
By Item 4 in Proposition \ref{prop4.2_2}, the reflected wave cannot be a large shock, and, if the outgoing 1-wave is still a shock, it will be weaker than the incoming one. So there are only two sub-cases including large shock(s). 

\begin{itemize}
\item If the outgoing 1-wave is a rarefaction or small shock, while the incoming 1-shock is large, then we
divide $a(\cdot,t_0+)$ at every state by $1/a^*$. Then by \eqref{mono} and \eqref{est2_2}, $a(x_0+,t_0+)<a(x_0+,t_0-)$. So $a$ decays in time at $t_0$ except at $x=x_0$. 

Now we show the total number of such kind of interactions is finite.

First, we show that the total number of interactions, where the number of outgoing large shock(s) is more than the number of incoming large shock(s), is finite. 
By Proposition \ref{prop4.2_2},
\begin{itemize}
\item head-on interaction does not change the strength of waves,
\item the shock-rarefaction overtaking interaction will not change a small shock to a large shock in the main family, and will not reflect large shock,
\end{itemize}
so the only case adding to the total number of large shocks is when two small shocks in the same family interact to form a large shock, where the total number of large shocks is incremented by 1.

 Using the method of Lemma 5.3 in \cite{Cheng_isothermal} (number of interactions in 4B and 9B are finite), we can show the total number of such kind of overtaking interaction (two small shocks produce a large shock) is finite and bounded by $\mathcal O(M/\epsilon)$.

Now we come back to interactions of 1-shock and 1-rarefaction when the outgoing 1-wave is a rarefaction or small shock, while the incoming 1-shock is large. Such an interaction causes the total number of large shocks to decrease by one.
So there are at most finitely many of them,
 because the total possible increase of the number of large shocks is finite, and on the other hand the total number of large shocks initially is finite and bounded by $\mathcal O(M/\epsilon)$. 


\item 
Another case is when the incoming and outgoing 1-shocks are both large. In this case, we do not drop $a(\cdot,t_0+)$.
By \eqref{mono} and \eqref{est2_2}, $a(x_0+,t_0+)<a(x_0+,t_0-)$. So $a$ decays in time at $t_0$ except at $x=x_0$. 
\end{itemize}

Note, for isothermal Euler, the $\nu$-approximate front tracking scheme used in \cite{ColomboRisebro} is slightly different from the scheme used in \cite{Cheng_isothermal} on very small shocks. Since $\epsilon\gg\nu$, Lemma 5.3 (number of interactions in 4B and 9B are finite) in \cite{Cheng_isothermal}  still holds for the $\nu$-approximate front tracking scheme equations in Eulerian coordinates. Although we use different wave strength from \cite{Cheng_isothermal}, the method still holds for our choice of strength, using the fact that the total variation $\sum_{\alpha}|\sigma_\alpha|$ is not increasing from before to after any pairwise interaction in Proposition \ref{prop4.2_2}. We leave the details to the reader.
\bigskip

\paragraph{\bf Step 5.}
Then we repeat the process to define $a(x,t)$ until the next wave interaction, by keeping $a$ as a constant between any two jumps in the scheme. Then repeat Step 4 at the wave interaction. After finitely many steps, we can define $a(x,t)$ for any $t>0$ and $x\in\mathbb R$.

\bigskip
\paragraph{\bf Step 6: Bounds on $a$ from above and below}

Now using our construction of $a(x,t)$, we can conclude the upper and lower bound of $a(x,t)$.
More precisely, by \eqref{aninf} and
\[
1\leq e^{\frac{3C_1}{4}U}\leq e^{\frac{3C_1}{4}M},
\]
we know 
\[
\lim_{x\rightarrow-\infty}a(x,t)\leq e^{\frac{3C_1}{4}M},
\]
since we only drop $\lim_{x\rightarrow-\infty}a(x,t)$ for special interactions, not increase it.

For each time, under Rule 2, $a$ is at most multiplied by a ratio $e^{\frac{3C_1}{4}V(t)}\leq e^{\frac{3C_1}{4}V(0)}\leq e^{\mathcal O(M)}$. For each time, Rule 1 is only used at waves with strength larger than $\epsilon$, so it is at most used $V(t)/\epsilon\leq V(0)/\epsilon\leq \mathcal{O}(M)/\epsilon$ times.
Hence,
\begin{align}\label{bound1_a_iso}
\norm{a}_{L^\infty(\mathbb{R}\times[0,T])}\leq 
e^{\frac{3C_1}{4}M}\cdot e^{\mathcal O(M)}\cdot (a^*)^{-\frac{\mathcal{O}(M)}{\epsilon}}=e^{\mathcal O(M)}(a^*)^{-\frac{\mathcal{O}(M)}{\epsilon}}.
\end{align}

Now we give the lower bound of $a$. First, $\lim_{x\rightarrow-\infty}a(x,t)$ may be dropped by a ratio $1/a^*$ for $\mathcal{O}(M\epsilon^{-2})$ times with $\epsilon\ll 1$. Furthermore, there is also an additional ratio $e^{\mathcal{O}(V)}\leq e^{\mathcal{O}(M)}$. Further,
similar to before, Rule 2 may decrease $a$ by dropping it at a ratio $e^{\frac{3C_1}{4}V(t)}\leq e^{\frac{3C_1}{4}V(0)}\leq e^{\mathcal O(M)}$. 
On the other hand, Rule 1 happens at most $V(t)/\epsilon\leq V(0)/\epsilon\leq \mathcal{O}(M)/\epsilon$ times. So 
\begin{align}\label{bound2_a_iso}
\norm{\frac{1}{a}}_{L^\infty(\mathbb{R}\times[0,T])}\leq 
1\cdot (a^*)^{-\mathcal{O}(M\epsilon^{-2})}\cdot e^{\mathcal{O}(M)}\cdot e^{\mathcal O(M)}\cdot (a^*)^{-\frac{\mathcal{O}(M)}{\epsilon}}=e^{\mathcal O(M)}(a^*)^{-\mathcal{O}(M\epsilon^{-2})}.
\end{align}

This gives \eqref{small_a_bounds} for (Problem Isothermal).

\section{Proof of Main Theorem (\Cref{main_theorem})}\label{sec:proof_main_theorem}

Before we start the proof of the Main Theorem (\Cref{main_theorem}), we collect a few important facts.

\subsection{Stopping and restarting the clock}

To run the front tracking scheme, we need to stop and restart the clock every time two waves collide. In the class of low regularity solutions we are considering in this paper, the ability to stop and restart the clock at positive times $t>0$ is a bit delicate and requires the use of approximate limits (for a reference on approximate limits, see \cite[p.~55-57]{MR3409135}). We give the precise statement we need below. These estimates will be stacked together to create a telescoping sum (see \Cref{sec:diss_calcs}, below).

\begin{lemma}[A technical lemma on stopping and restarting the clock {\cite[Lemma 2.5]{2017arXiv170905610K}}]
    \label{stop_restart}
    Let $u \in \mathcal{S}_{\text{weak}}$ with initial data $u^0$. Then for all $v \in \mathcal{V}$, and for all $c, d \in \mathbb{R}$ with $c < d$, the approximate right- and left-hand limits
    \begin{equation}
        \operatorname{ap} \lim_{t \to t_0^\pm} \int_c^d \eta(u(x,t)|v) \,dx
    \end{equation}
    exist for all $t_0 \in (0, T)$ and satisfy that
    \begin{equation}
        \operatorname{ap} \lim_{t \to t_0^-} \int_c^d \eta(u(x,t)|v) \,dx \geq \operatorname{ap} \lim_{t \to t_0^+} \int_c^d \eta(u(x,t)|v) \,dx.
    \end{equation}
    Moreover, the approximate right-hand limit exists at $t_0 = 0$ and satisfies
    \begin{equation}
        \int_c^d \eta(u^0(x)|v) \,dx \geq \operatorname{ap} \lim_{t \to t_0^+} \int_c^d \eta(u(x,t)|v) \,dx.
    \end{equation}
\end{lemma}

The proof of Lemma \ref{stop_restart} is the same as the proof of \cite[Lemma 2.5]{2017arXiv170905610K}. We do not reproduce the proof here.

\subsection{Proof of Main Theorem (\Cref{main_theorem})}

We can now begin the proof of \Cref{main_theorem}. We follow \cite[Section 7]{MR4487515} and \cite[Section 6]{GiesselmannKrupa2025}.

\subsubsection{Step 1: the dissipation calculations}\label{sec:diss_calcs} 
 
 For (Problem $\epsilon$-BV), consider initial data $v^0\in\mathcal{S}_{\textrm{BV},\epsilon}^0$ for $\epsilon>0$ sufficiently small. For (Problem Isothermal), fix $M>0$ and consider initial data $v^0\in\mathcal{S}_{\textrm{BV},M}^0$. Moreover, consider a fixed wild solution $u \in \mathcal{S}_{\text{weak}}$. In these calculations, we assume the total variation upper bound is $M$ as for the (Problem Isothermal) but the calculations are nearly identical in the (Problem $\epsilon$-BV) case.


Consider a sequence of front tracking approximations $v_\nu$ (see \Cref{sec:front_tracking_no_shift}) with initial data verifying $v_\nu(\cdot,0)\to v^0(\cdot)$ in $L^2$ as $\nu\to0$. For a fixed $\nu$, consider the shifted $\nu$-approximate front tracking solution $\psi$ verifying $\psi(x,0)=v_\nu(x,0)$ (see \Cref{sec:shifted_front_tracking}). For this fixed $\psi$, we consider also the associated weight function $a(x,t)$ from either \Cref{prop:small_BV_a} or \Cref{prop:iso_a}.

Consider two successive interaction times $t_j < t_{j+1}$ of the shifted front tracking solution $\psi$. Let the discontinuity curves  between the two times $t_j < t_{j+1}$ be $h_1, \dots, h_N$ for some  $N \in \mathbb{N}$ such that
\begin{equation}
    h_1(t) < \dots < h_N(t)
\end{equation}
for all $t \in (t_j, t_{j+1})$. We restrict ourselves to the cone of information. Thus, we define for $t\geq0$
\begin{equation}
    h_0(t) \coloneqq -R + s(t-\tau),
\end{equation}
\begin{equation}
    h_{N+1}(t) \coloneqq R - s(t-\tau),
\end{equation}
where $s > 0$ is the speed of information and verifies 
\begin{align}\label{speed_of_info}
|q(a; b)| \leq s \eta(a|b)
\end{align}
for all $a \in \mathcal{V}$ and $b$ in the image of $\psi$. Note that the constant $s$ exists by the first statement in \Cref{rel_facts_lemma}. We fix the value $s$ to be bigger than the magnitudes of the speeds of any of the shift functions or any of the waves used in the front tracking scheme (\Cref{sec:front_tracking_no_shift}).

Remark that there are no interactions between waves in $\psi$ and the information cone (defined by $h_0$ and $h_{N+1}$, above). For  $t \in [t_j, t_{j+1}]$, note that on 
\begin{equation}
    Q = \{(x, r) \  |\  t_j < r < t, h_i(r) < x < h_{i+1}(r)\},
\end{equation}
the function $\psi(x, r) = b$ for some constant $b$. The weight function $a(x,r)$ is also constant on this set by our construction(s). We integrate \eqref{ineq:relative} on $Q$ and use the Strong Trace Property (\Cref{strong_trace_prop}). This gives  that

\[
    \operatorname{ap} \lim_{r \to t^-} \int_{h_i(t)}^{h_{i+1}(t)} a(x,t-) \eta(u(x,r)|\psi(x,t)) \,dx
    \leq \operatorname{ap} \lim_{r \to t_j^+} \int_{h_i(t_j)}^{h_{i+1}(t_j)} a(x,t_j+) \eta(u(x,r)|\psi(x,t_j)) \,dx
\]
\[
   \hspace{3.8in} + \int_{t_j}^{t} \big( F_i^+(r) - F_{i+1}^-(r) \big) \,dr,
\]
where
\begin{align*}
    F_i^+(r) &= a(h_i(r)+,r) \big[ q(u( h_i(r)+,r); \psi( h_i(r)+,r)) - \dot h_i(r) \eta(u( h_i(r)+,r)|\psi(h_i(r)+,r)) \big], \\
    F_i^-(r) &= a(h_i(r)-,r) \big[ q(u( h_i(r)-,r); \psi(h_i(r)-,r)) - \dot h_i(r) \eta(u( h_i(r)-,r)|\psi( h_i(r)-,r)) \big].
\end{align*}

We sum over $i$, and collect the terms corresponding to $i$ into one sum, and the terms corresponding to $i+1$ into a second sum, to find that
\begin{equation}
    \operatorname{ap} \lim_{r \to t^-} \int_{-R+s(t-\tau)}^{R-s(t-\tau)} a(x,t-) \eta(u(x,r)|\psi(x,t)) \,dx.
\end{equation}

\[
    \leq \operatorname{ap} \lim_{r \to t_j^+} \int_{-R+s(t_j-\tau)}^{R-s(t_j-\tau)} a(x,t_j+) \eta(u(x,r)|\psi(t_j, x)) \,dx
\]
\[
    + \sum_{i=1}^{N} \int_{t_j}^{t} \big( F_i^+(r) - F_i^-(r) \big) \,dr,
\]

where we note that $F_0^+ \leq 0$ {and} $F_{N+1}^- \geq 0$ by the first statement of \Cref{rel_facts_lemma}, the choice of $s$, and $\dot{h}_0 = -s = -\dot{h}_{N+1}$. 

We partition the sum into two sums: one sum collects the shocks and the other sum collects the rarefaction fronts. We know from \Cref{prop:small_shock_diss} and \Cref{prop:large_shock_diss}, that for any $i$ denoting a shock front {with size $|\sigma_i| \geq 2\sqrt{\nu}$},
\begin{equation}\label{combined_diss}
    F_i^+(r) - F_i^-(r) \leq -\frac{1}{K}\abs{\psi(h_i(r)+,r)-\psi(h_i(r)-,r)}(\sigma(\psi(h_i(r)+,r),\psi(h_i(r)-,r))-\dot{h}_i(r))^2
\end{equation}
where $\sigma(\psi(h_i(r)+,r),\psi(h_i(r)-,r))$ is the true Rankine-Hugoniot speed of the shock 
\begin{align*}
(\psi(h_i(r)+,r),\psi(h_i(r)-,r)).
\end{align*}
For small shocks with size $|\sigma_i| < 2 \sqrt{\nu}$ we find
\begin{equation}\label{combined_diss_small}
    \begin{aligned}
    F_i^+(r) - F_i^-(r) \leq& -\frac{1}{K}\abs{\psi(h_i(r)+,r)-\psi(h_i(r)-,r)}(\lambda_{\alpha_i}^\varphi(\psi(h_i(r)-,r),\psi(h_i(r)+,r))-\dot{h}_i(r))^2 \\
    &\hspace{2in}+ K\nu\abs{\psi(h_i(r)+,r)-\psi(h_i(r)-,r)} ,
    \end{aligned}
\end{equation}
where $\lambda_{\alpha_i}^\varphi(\psi(h_i(r)-,r),\psi(h_i(r)+,r))$ is the interpolated speed~\eqref{averaged_speed} corresponding to the interpolated shock $(\psi(h_i(r)-,r),\psi(h_i(r)+,r))$ and the error term is from the estimates~\eqref{eq:small-shock-error-psi} and \Cref{rel_facts_lemma}.
Equations~\eqref{combined_diss} and~\eqref{combined_diss_small} will hold for almost every $t_j < r < t$ except for the very small intervals of time during which the speed of $h_i$ will be slightly perturbed to avoid multiple shock interactions (see the proof of \Cref{prop:iso_a}). Remark that \Cref{prop:large_shock_diss} will give us \eqref{combined_diss} for $K>0$ sufficiently large because we only invoke \Cref{prop:large_shock_diss} when \begin{align*}
\abs{\psi(h_i(r)+,r)-\psi(h_i(r)-,r)}
\end{align*}
is sufficiently large, otherwise we invoke \Cref{prop:small_shock_diss} (see the construction of $\psi$ in \Cref{sec:shifted_front_tracking}). Moreover, $\abs{\psi(h_i(r)+,r)-\psi(h_i(r)-,r)}$ is uniformly bounded.

Denote $\mathcal{R}(r)$ the set of $i$ corresponding to approximated rarefaction fronts at time $r$. Then for any $i \in \mathcal{R}(r)$ $a(h_i(r)+, r) = a(h_i(r)-, r)$ by construction. We have then from \Cref{pizza_slice_diss},
\begin{equation}
    \sum_{i \in \mathcal{R}(r)} \int_{t_j}^{t} \big( F_i^+(r) - F_i^-(r) \big) \,dr 
    \leq K {\nu} (t - t_j) \sum_{i \in \mathcal{R}(r)} |\sigma_i|
    \leq K {\nu} (t - t_j),
\end{equation}
where we use \eqref{BV_psi} to control the sum $\sum |\sigma_i|$. Here, we are not yet taking into account that for very small intervals of time we may have to slightly adjust the speed of the rarefaction fronts to avoid multiple shock interactions (see the proof of \Cref{prop:iso_a}). 

Collecting now all the families of waves, we receive
\begin{multline}
    \operatorname{ap} \lim_{r \to t^-} \int_{-R+s(t-\tau)}^{R-s(t-\tau)} a(x,t-) \eta(u(x,r)|\psi(x,t)) \,dx\\
     \leq \operatorname{ap} \lim_{r \to t_j^+} \int_{-R+s(t_j-\tau)}^{R-s(t_j-\tau)} a(x,t_j+) \eta(u(x,r)|\psi(x,t_j)) \,dx
    +K \nu (t - t_j) \\
    -\frac{1}{K} \int_{t_j}^{t} \sum_{i \in \mathcal{S}(r)}   \abs{\psi(h_i(r)+,r)-\psi(h_i(r)-,r)}\big({\lambda_{\alpha_i}^\varphi(\psi(h_i(r)-,r),\psi(h_i(r)+,r))}-\dot{h}_i(r)\big)^2 \,dr,
\end{multline}

where $\mathcal{S}(r)$ denotes the set of shocks at time $r$.
{We remark that the sum of the error terms from small shocks in~\eqref{combined_diss_small} are bounded by $K\nu(t-t_j)$ due to our total variation bound~\eqref{BV_psi}.}

Consider $0 < \tau < T$, and let $0 < t_1 < \dots < t_J$ be the times of wave interactions before $\tau$, where we also have $t_0 = 0$, and $t_{J+1} = \tau$. By the convexity of $\eta$, \Cref{stop_restart}, and the decay of the weight function $a(x,t)$ (see \Cref{prop:small_BV_a} or \Cref{prop:iso_a}) we get
\begin{equation}
    \int_{-R}^{R} a(x,\tau) \eta(u(x,\tau)|\psi(x,\tau)) \,dx - \int_{-R-s\tau}^{R+s\tau} a(x,0) \eta(u(x,0)|\psi(x,0)) \,dx
\end{equation}
\[
    \leq \operatorname{ap} \lim_{r \to \tau^+} \int_{-R}^{R} a(x,\tau) \eta(u(x,r)|\psi(x,\tau)) \,dx - \int_{-R-s\tau}^{R+s\tau} a(x,0) \eta(u(x,0)|\psi(x,0)) \,dx
\]
\[
    \leq \sum_{j=1}^{J+1} \left( \operatorname{ap} \lim_{r \to t_j^+} \int_{-R+s(t_j-\tau)}^{R-s(t_j-\tau)} a(x,t_j-) \eta(u(x,r)|\psi(x,t_j)) \,dx \right.
\]
\[
    \left. - \operatorname{ap} \lim_{r \to t_{j-1}^+} \int_{-R+s(t_{j-1}-\tau)}^{R-s(t_{j-1}-\tau)} a(x,t_{j-1}-) \eta(u(x,r)|\psi(x,t_{j-1})) \,dx \right)
\]
\[
    \leq \sum_{j=1}^{J+1} \left( \operatorname{ap} \lim_{r \to t_j^-} \int_{-R+s(t_j-\tau)}^{R-s(t_j-\tau)} a(x,t_j-) \eta(u(x,r)|\psi(x,t_j)) \,dx \right.
\]
\[
    \left. - \operatorname{ap} \lim_{r \to t_{j-1}^+} \int_{-R+s(t_{j-1}-\tau)}^{R-s(t_{j-1}-\tau)} a(x,t_{j-1}-) \eta(u(x,r)|\psi(x,t_{j-1})) \,dx \right)
\]

\begin{multline}
    \leq   
    -\frac{1}{K} \int_0^{\tau} \sum_{i \in \mathcal{S}(r)}   \abs{\psi(h_i(r)+,r)-\psi(h_i(r)-,r)}\big({\lambda_{\alpha_i}^\varphi(\psi(h_i(r)-,r),\psi(h_i(r)+,r))}-\dot{h}_i(r)\big)^2 \,dr
   + K \nu,
\end{multline}
where now the error term $K {\nu}$ can also account for the very small intervals of time during which the speeds of waves will be slightly perturbed to avoid multiple shock interactions (see the proof of \Cref{prop:iso_a}).

From this, we get

\begin{align}\label{l2_stab_estimate}
      \int_{-R}^{R} a(x,\tau) \eta(u(x,\tau)|\psi(x,\tau)) \,dx \leq \int_{-R-s\tau}^{R+s\tau} a(x,0) \eta(u(x,0)|\psi(x,0)) \,dx +K {\nu} 
\end{align}

as well as

\begin{multline}\label{complete_diss_estimate}
    \frac{1}{K} \int_0^{\tau} \sum_{i \in \mathcal{S}(r)}   \abs{\psi(h_i(r)+,r)-\psi(h_i(r)-,r)}\big({\lambda_{\alpha_i}^\varphi(\psi(h_i(r)-,r),\psi(h_i(r)+,r))}-\dot{h}_i(r)\big)^2 \,dr
    \\
    \leq \int_{-R-s\tau}^{R+s\tau} a(x,0) \eta(u(x,0)|\psi(x,0)) \,dx +K {\nu} .
\end{multline}

\subsubsection{Step 2: calculation of the amount of shifting}\label{sec:how_much_shift}

Recall from above that $v_\nu$ is a front tracking solution with initial data verifying $v_\nu(\cdot,0)=\psi(\cdot,0)$ except unlike $\psi$, the shocks in $v_\nu$ will be moving with speeds dictated by the true front tracking algorithm in \Cref{sec:front_tracking_no_shift}, i.e. \eqref{averaged_speed}.

Then, from \Cref{control_L1_norm},
\begin{align}
     \norm{v_\nu(\cdot,\tau)-\psi(\cdot,\tau)}_{L^1(\mathbb{R})}
     \leq K d_\nu(v_\nu(\cdot,t),\psi(\cdot,t)),
\end{align}
where remark that for (Problem Isothermal) we know from \eqref{iso_constant} that in fact at this step \begin{align}
K=\exp\big(\tilde{K}M e^M\big),
\end{align}
for some constant $\tilde{K}>0$ which does not depend on $M$ (see  \Cref{dependenceTV} and recall $M$ from \eqref{large_BV_data}).
Continuing, from $v_\nu(\cdot,0)=\psi(\cdot,0)$ and \Cref{lemma:front_tracking_error_estimate}, we have
\begin{multline}
\leq K \int_0^\tau \hspace{.4in}\sum_{\mathclap{\substack{h_i:\text{shock in $\psi$ at time $r$}}}} \left| \psi({h}_i(r)-, r) - \psi({h}_i(r)+, r) \right|
    \left| \dot{{h}}_{i} -  \lambda_{\alpha_i}^\varphi(\psi(h_i(r)-,r),\psi(h_i(r)+,r)) \right|\,dr 
    \\+[\text{error due to perturbing wave speeds to avoid multiple interactions}],
    \end{multline}
    where $\lambda_{\alpha_i}^\varphi$ is the speed at which the shock $h_i$ in $\psi$ would be traveling if the speed was determined by the interpolated speed used in the front tracking algorithm (see e.g. equation \eqref{averaged_speed}) instead of using shifts. The constant $K$ again comes from \eqref{iso_constant}.
   
    Using the Cauchy-Schwartz inequality
\begin{multline}
    \leq K \Bigg[\int_0^\tau \hspace{.4in} \sum_{
    \mathclap{\substack{h_i:\text{shock in $\psi$ at time $r$}}}} \left| \psi({h}_i(r)-, r) - \psi({h}_i(r)+, r) \right|\,dr\Bigg]^{\frac{1}{2}}\times \\
    \Bigg[\int_0^\tau \hspace{.4in} \sum_{
    \mathclap{\substack{h_i:\text{shock in $\psi$ at time $r$}}}} \left| \psi({h}_i(r)-, r) - \psi({h}_i(r)+, r) \right|
    \left| \dot{{h}}_{i} - {\lambda_{\alpha_i}^\varphi \big( \psi({h}_i(r)-, r), \psi({h}_i(r)+, r) \big) }\right|^2\,dr\Bigg]^{\frac{1}{2}}
    +\nu.
\end{multline}

From \eqref{complete_diss_estimate}, we get

\begin{align}
    &\leq K \Bigg[\int_0^\tau \hspace{.4in} \sum_{
    \mathclap{\substack{h_i:\text{shock in $\psi$ at time $r$}}}} \left| \psi({h}_i(r)-, r) - \psi({h}_i(r)+, r) \right|\,dr\Bigg]^{\frac{1}{2}}\times \\
    &\hspace{2.6in}\Bigg[K\Big[\int_{-R-s\tau}^{R+s\tau} a(x,0) \eta(u(x,0)|\psi(x,0)) \,dx +K {\nu}  \Big]\Bigg]^{\frac{1}{2}}+\nu\\
    &\leq K\sqrt{M} \Bigg[\int_{-R-s\tau}^{R+s\tau} a(x,0) \eta(u(x,0)|\psi(x,0)) \,dx +K {\nu}\Bigg]^{\frac{1}{2}}+\nu,
\end{align}
where we remark that the term $\Bigg[\int_0^\tau \sum \left| \psi({h}_i(r)-, r) - \psi({h}_i(r)+, r) \right|\,dr\Bigg]^{\frac{1}{2}}\leq \tilde{K}\sqrt{M}$ for some $\tilde{K}$ which does not depend on $M$.

\subsubsection{Step 3: putting it all together}

A simple triangle inequality gives,

\begin{multline}
    \norm{u(\cdot,\tau)-v_\nu(\cdot,\tau)}_{L^1((-R,R))}\leq \norm{u(\cdot,\tau)-\psi(\cdot,\tau)}_{L^1((-R,R))}+\norm{\psi(\cdot,\tau)-v_\nu(\cdot,\tau)}_{L^1((-R,R))}.
\end{multline}

From the nesting property of the $L^p$ spaces on compact sets (in particular $L^2\subset L^1$), we get

\begin{align}
    \norm{u(\cdot,\tau)-v_\nu(\cdot,\tau)}_{L^1((-R,R))}
    \leq \sqrt{2R}\norm{u(\cdot,\tau)-\psi(\cdot,\tau)}_{L^2((-R,R))}+\norm{\psi(\cdot,\tau)-v_\nu(\cdot,\tau)}_{L^1((-R,R))}
\end{align}

Then, from \Cref{lemma:leger-square}, \eqref{l2_stab_estimate}, and the calculation in \Cref{sec:how_much_shift}, we find

\begin{multline}
     \norm{u(\cdot,\tau)-v_\nu(\cdot,\tau)}_{L^1((-R,R))} 
     \lesssim  
      (\sqrt{2R}+K\sqrt{M}) \Bigg[\norm{u( \cdot,0)-v_\nu( \cdot,0)}_{L^2((-R-st,R+st))}^2 +K{\nu}  \Bigg]^{\frac{1}{2}}
     +\nu,
\end{multline}
where we have used also that $v_\nu(\cdot,0)=\psi(\cdot,0)$.

Recall now the baby interpolation inequality for $L^p$-norms, $\norm{f}_{L^2}\leq\sqrt{\norm{f}_{L^1}\norm{f}_{L^\infty}}$ for a function $f$. Then, taking the limit $\nu\to0$, the error terms disappear, and there exists a solution $v$ to \eqref{system} or \eqref{isothermal} such that $v(\cdot,0)=v^0(\cdot)$ and $v_\nu(\cdot,t)$ converges to $v(\cdot,t)$ in $L^2$ for all $t$. We get this \emph{for all} $t$ instead of simply almost every $t$ by \cite[Theorem 2.4]{MR1816648}. This gives,

\begin{align}
     \norm{u(\cdot,\tau)-v(\cdot,\tau)}_{L^2((-R,R))} \lesssim  
     \sqrt{\norm{u( \cdot,0)-v( \cdot,0)}_{L^2((-R-s\tau,R+s\tau))}}.
\end{align}

Finally, for the isothermal case, remark that control on the constant $K$ in \eqref{main_estimate} can come from  \eqref{iso_constant}, \eqref{bound1_a_iso}, and \eqref{bound2_a_iso}.





\section{Proof of \Cref{cor:uniqueness}}\label{sec:cor}

In this section, we consider the isothermal Euler system \eqref{isothermal}, but this argument applies to any system which satisfies \Cref{main_theorem} for large BV solutions.
Fixing $R > 0$, a phase space $\Nu$, and a compact convex set $\mathcal{K}\subset\Nu$, equation~\eqref{main_estimate} can be written as 
\begin{align} \label{eq:est-arb-BV}
     \norm{u(\cdot,\tau)-v(\cdot,\tau)}_{L^2((-R+s\tau,R-s\tau))} \leq K\left(\norm{v(\cdot, 0)}_{BV((-R,R))}\right) 
     \sqrt{\norm{u( \cdot,0)-v( \cdot,0)}_{L^2((-R,R))}}
\end{align}
for all $\tau < T = R/s$, where $K$ is an increasing function of the total variation of $v(\cdot,0)$. 
The growth rate of $K$ completely determines the class $X$ of initial data we consider (see equation~\eqref{eq:def-class-X} for a definition of $X$). 
This corollary is the observation that for suitably chosen data $v(\cdot,0)=v_0(\cdot)$, there exist approximations $v_0^{(n)}$ for which the upper bound~\eqref{eq:est-arb-BV} converges to zero, establishing uniqueness in the class of wild solutions $\Sweak$ within the cone of information $ \{(x,\tau): \tau < R/s,\ x\in(-R + s\tau, R-s\tau) \}$.

To begin, we show a collection of conditions which guarantee a function $v$ is within $X$.
These are then used in the proof of \Cref{cor:uniqueness} to exhibit infinite BV data within $X$.
Given a fixed decreasing function $G\colon [0,\infty) \to [0,\infty)$ we define the set $X_G$ to be the collection of all $v_0 \in L^\infty((-R, R);\mathcal{K})$ where for every $n$ sufficiently large there exists an open set $V^n = \cup_{i = 1}^N (a_i,b_i)\subset(-R,R) $ satisfying the following properties,
\begin{enumerate}[label=(\alph*)]
    \item \label{prop:uni-measure} $\text{meas}(V^n) \leq G(2n)$,
    \item \label{prop:TV-outside} $\norm{v}_{BV( (-R,R)\setminus V^n)} \leq n $, and 
    \item \label{prop:n-of-intervals} $N \leq n/(2\ \text{diam}(\mathcal{K}))$.
\end{enumerate}
\begin{lemma} \label{lem:X-G-inc}
We have the inclusion
$$ X_G\subset X,$$
for every $G$ which decreases sufficiently quickly to satisfy
\begin{equation} \label{eq:G-K-limit} \lim_{n\to \infty} K(n) G(2n)^{1/4} = 0.\end{equation}
\end{lemma}
\begin{proof}
    Let $v \in X_G$ where $G:[0,\infty) \to [0,\infty)$ is a decreasing function satisfying equation~\eqref{eq:G-K-limit}.
    The open sets $V^n$ associated to $v$ are used to construct the approximations 
    \begin{equation}
        v_n(x) = 
        \begin{cases}
            v(x) & x \notin V^n\\
            \frac{1}{|V^n|} \int_{V^n} v\,dx & x\in V^n.
        \end{cases}
    \end{equation}
    From properties \ref{prop:TV-outside} and \ref{prop:n-of-intervals} we can bound
    \beq \label{eq:approximate-BV-bound} \norm{v_n}_{BV((-R,R))}\leq \norm{v}_{BV((-R,R) \setminus V^n)} + 2N\text{diam}(\mathcal{K}) \leq 2n \eeq
    while property \ref{prop:uni-measure} gives us the bound
    \beq \label{eq:approximate-L2-bound}\norm{v - v_n}_{L^2((-R,R))} \leq \text{diam}(\mathcal{K}) G(2n)^{1/2}.\eeq
    From the bounds~\eqref{eq:approximate-BV-bound} and~\eqref{eq:approximate-L2-bound} we compute
    \begin{equation}\label{eq:K-BV-L2-limit} \lim_{n\to\infty} K\left(\norm{v_n}_{BV((-R,R))}\right) 
        \norm{v-v_n}^{1/2}_{L^2((-R,R))} \leq 
        \lim_{n\to\infty}K(2n)G(2n)^{1/4}\text{diam}(\mathcal{K})^{1/2} = 0,
        \end{equation}
    establishing $v\in X.$
\end{proof}

We note that the construction of the $X_G$ closely mirrors that of the sets of initial data $(\tilde P_\alpha)$ considered in \cite{2025arXiv250500420B}. 
In that work, the structure of the initial data in $(\tilde P_\alpha)$ gave a solution with a quick total variation decay, satisfying the a priori assumption of the uniqueness theorem presented in that paper.
In this work, we use the above assumptions to construct a sequence in $BV((-R,R))$ which converges in $L^2$ with efficient cost in the BV-norm of our approximations. 
In particular, we make no a priori assumptions on the total variation decay of solutions with initial data $v\in X$.

\begin{proof}[Proof of \Cref{cor:uniqueness}]
\hfill

    \textit{Step 1: Uniqueness.}
    Suppose that $u_1,\, u_2 \in \mathcal{S}_{weak}$ are two solutions with $u_i(\cdot, 0)|_{(-R+sT,R-sT)} = v \in X$ for $i = 1,\ 2$.
    By the triangle inequality,
    \begin{equation}\label{eq:uni-triangle}
        \begin{aligned}
            \norm{u_1(\cdot, \tau) - u_2(\cdot,\tau)}_{L^2((-R+s\tau,R-s\tau))} \leq& \norm{u_1(\cdot, \tau) - v_n(\cdot,\tau)}_{L^2((-R+s\tau,R-s\tau))} \\&+ \norm{u_2(\cdot, \tau) - v_n(\cdot,\tau)}_{L^2((-R+s\tau,R-s\tau))}
        \end{aligned}
    \end{equation}
    for any $n$, where the $v_n$ are the approximating sequence given by our class $X$~\eqref{eq:def-class-X}.
    We can estimate either term on the right side with \Cref{main_theorem}.
    This estimate yields 
    $$ \norm{u_i(\cdot, \tau) - v_n(\cdot,\tau)}_{L^2((-R+s\tau,R-s\tau))} \leq K\left( \norm{v_n}_{BV((-R,R))} \right) \norm{v - v_n}_{L^2((-R,R))}^{1/2} $$
    where we note that by assumption the upper bound converges to zero as $n \to \infty$ by equation~\eqref{eq:G-K-limit}. 
    As a result, taking the lim inf of equation~\eqref{eq:uni-triangle} as $n\to \infty$ we arrive at 
    $$ \norm{u_1(\cdot, \tau) - u_2(\cdot,\tau)}_{L^2((-R+s\tau,R-s\tau))} = 0\quad \forall \tau \in [0,R/s],$$
    establishing uniqueness. 

    \vskip.2cm
    \textit{Step 2: $X$ is strictly larger than $BV$.}
    The inclusion $BV((-R, R);\mathcal{K}) \subset X$ is given by our construction. 
    What remains is to show is that these sets are not equal.
    To do this we consider the example~\eqref{eq:ex-unique}, 
    $$ W(x) = b_1 + \frac{b_2 - b_1}{2} \left(\left(\sin\left(\Gamma (|x|^{-1})\right)\right) + 1 \right), $$
    and show $W \in X_G$ for an unbounded increasing $\Gamma\colon[0,\infty) \to [0,\infty)$.
    To start, since $\Gamma$ is increasing the total variation will be concentrated within intervals $V^n = (-a_n,a_n)$.
    To satisfy property \ref{prop:uni-measure}, we select $a_n = G(2n)/2$ for a positive, decreasing $G$ satisfying equation~\eqref{eq:G-K-limit}. 
    Next we can bound the variation in terms of $\Gamma$,
    $$ \norm{W}_{BV( (-R,R)\setminus V^n)} \leq \frac{|b_1 - b_2|}{\pi} \Gamma(a_n^{-1}).$$
    From our bound on $a_n$ we find if $\Gamma$ satisfies 
    $$ \Gamma\left(2(G (2n))^{-1}\right) \leq \frac{\pi}{|b_1 - b_2|} n$$
    then $W$ satisfies property \ref{prop:TV-outside} establishing $W \in X$.
    Finally, we can select a $\Gamma$ that satisfies the above bound and $\Gamma(k) \to \infty$ as $k\to \infty$, establishing $\norm{W}_{BV((-R,R ))} = \infty$.

    \vskip.2cm
    \textit{Step 3: $X$ contains elements that are everywhere locally infinite BV.}
    We note that the main estimates in the proof of \Cref{lem:X-G-inc} are the bounds~\eqref{eq:approximate-BV-bound} and~\eqref{eq:approximate-L2-bound}. 
We see that these bounds also hold if we take a convex combination of data $v_i \in X_G$ and the associated approximating sequences $\{v_i^{(n)} \}_{n = 1}^\infty$ constructed in the proof of \Cref{lem:X-G-inc}. 
This can even be done in the case of a convex combination of an infinite number of $v_i$ (as long as the notion of ``$n$ sufficiently large'' in our definition of $X_G$ is \textit{uniform} among the data we are taking a convex combination of).
To see this, fix $\{v_i\}_{i=1}^\infty \in X_G$ for a positive $G$ satisfying equation~\eqref{eq:G-K-limit} and a collection of positive $a_i$ satisfying $\sum a_i = 1$.
Let $$ v = \sum_{i=1}^\infty a_i v_i,$$ which converges in $ L^\infty((-R,R); \mathcal{K})$ as $\mathcal{K}$ is convex and compact. 
We further suppose that the $v_i$ are in fact translations of the function $W$ (see \eqref{eq:ex-unique}), $v_i(x)=W(x-x_i)$ where the points $\{x_i\}$ are dense in $(-R,R)$. In particular, each $v_i$ has infinite total variation around the point $x_i$. We can  then conclude $\norm{v}_{BV((a,b))} = \infty$ for all $(a,b) \subset (-R,R)$. 
Indeed, given $x_k \in (a,b)$ we see there exists $N > 1$ such that 
\begin{equation}\label{unique:partial}
    \sum_{n>N} a_n <  \norm{W}_{L^\infty}a_k/2.
\end{equation}
We can then write 
$$ v(x) = \underbracket{\sum_{\substack{n \leq N \\ n\ne k}}a_n W(x - x_n)}_{\text{total variation } <\,\,\infty} + \underbracket{\vphantom{\sum_{\substack{n \leq N \\ n\ne k}} } \left( a_kW(x-x_k) +\sum_{n > N} a_nW(x - x_n) \right)}_{\text{total variation } =\,\,\infty},$$
where the first term has finite total variation near $x_k$ and cannot impact whether $v$ has infinite total variation close to $x_k$.
The second term has infinite total variation, as the variation between the adjacent peaks/valleys of $a_kW(x - x_k)$ are at most shrunk by the factor $a_k$, twice the bound for the sum \eqref{unique:partial} over $n>N$.
This shows the total variation of $v$ indeed blows up as we approach $x_k,$ independent of choice of $\{x_n\}$ and $\{a_n\}$.

We define the associated approximations 
$$v_n = \sum_{i=1}^\infty a_i v_i^{(n)},$$ 
where the $\{v_i^{(n)} \}_{n=1}^\infty$ satisfy~\eqref{eq:approximate-BV-bound} and~\eqref{eq:approximate-L2-bound} by construction in \Cref{lem:X-G-inc}.
We note by the triangle inequality, equation~\eqref{eq:approximate-BV-bound} and~\eqref{eq:approximate-L2-bound} we have the bounds
\begin{align*}
\norm{v_n}_{BV} &\leq \sum a_i \norm{v_i^{(n)}} = 2n, \\
\norm{v-v_n}_{L^2} &\leq \sum a_i \norm{v_i - v_i^{(n)}}_{L^2} \leq \text{diam}
(\mathcal{K}) G(2n)^{1/2}
\end{align*}
which suffice to give the limit~\eqref{eq:K-BV-L2-limit}, establishing $v \in X$.
\end{proof}


\section{Dissipation estimate for small shocks}\label{sec:proof_prop:small_shock_diss}
\subsection{Important note on notation}

In this section, we have some additional notes on notation (cf. \Cref{sec:not}), largely mirroring that of \cite{MR4667839}. Throughout this section we localize our analysis to a region $B_{2\epsilon}(d)$ so that we can guarantee separation of characteristic velocities.
More precisely, we require $\epsilon > 0$ is such that
\begin{equation}\label{eq:depsilon}
    B_{2\epsilon}(d) \subset \Nu \quad \text{ and } \quad \sup_{u \in B_{2\epsilon}(d)} \lambda_1(u) < \inf_{u \in B_{2\epsilon}(d)} \lambda_2(u).
\end{equation}
For brevity, most of the forthcoming statements will omit this fact. 
Furthermore, we will often make the statement 
\begin{quote}
    ``For all $C$ sufficiently large and $s_0$ sufficiently small [...]''
\end{quote}
as a concise form of the more precise statement
\begin{quote}
    ``For all $d \in \Nu$ and $\epsilon > 0$ satisfying~\eqref{eq:depsilon}, and $C_1$ sufficiently large there exists a $\tilde s_0 = \tilde s_0(C_1)$ such that for all $C$ satisfying $C_1/2 < C < 2C_1$ and $s_0 < \tilde s_0$ [...]''
\end{quote}

Furthermore, 
\begin{itemize}
    \item $C$ denotes the constant used to define $\tilde \eta,\tilde q, \Pi_{C,s_0}$.
    \item In this section, the constant $K$ depend will depend additionally on $C_1$ and $\tilde s_0$. These constants are understood to be independent of $C$ and $s_0$. 
    This is true for the implicit constant in the notation ``$a \lesssim b$'' as well. 
    \item For $A$ a $n$-tensor and $B$ an $m$-tensor, we denote $A\otimes B$ as the $n+m$ tensor $a_{i_1\dots i_n}b_{j_1\dots j_m}$.
    \item We use $| \cdot|$ to denote a finite-dimensional norm of vectors, matrices, and tensors. 
    \item We denote derivatives as $g'$ when $g$ is matrix or vector-valued and $\nabla f$ when $f$ is a scalar function.
\end{itemize}

\subsection{Proof outline}
The main idea of this proof is to take the relative entropy inequality,
\begin{equation}
    \partial_t \eta(u|v) + \partial_x q(u;v) \leq 0
\end{equation}
and integrate it on the left and right of our shift $h(t)$, where $v = u_L$ or $v = u_R$ depending on whether $x$ is to the left or right of $h$.
Thanks to the solution $u$ satisfying the strong trace property (\Cref{strong_trace_prop}) we can differentiate the quantity defined in equation~\eqref{eq:pseudo-dist} to arrive at the following entropy dissipation function
\begin{equation} \label{eq:diss-function}
    \frac{d}{dt} E_t(u) \leq h'(t)(a_1 \eta(u_-(t)| u_L) - a_2\eta(u_+(t)|u_R)) - (a_1 q(u_-(t); u_L) - a_2q(u_+(t);u_R)).
\end{equation}
Without loss of generality, we fix a 1-shock $(u_L,u_R,\sigma_{LR})$. 
Given $a_1,a_2 > 0$ we define the set 
\begin{equation} \label{eq:def-pi}
    \Pi := \left\{ u \in \mathcal{V}\ |\ a_1 \eta(u|u_L) - a_2 \eta(u|u_R) < 0 \right\}.
\end{equation}
This set is of great importance in our analysis, because if $u_-(t) = u_+(t) \in \partial \Pi$ we find that the $\dot h$ term in~\eqref{eq:diss-function} vanishes. 
Since this term vanishes, we may simply take $\dot h(t) = \lambda_1(u)$ in this case without changing the value of our bound in~\eqref{eq:diss-function}. 
We then have to verify that the function 
\begin{equation}\label{eq:Dcont-def}
    D_{\cont}(u) := (q(u;u_R) - \lambda_1(u) \eta(u|u_R)) -  \frac{a_1}{a_2}(q(u;u_L) - \lambda_1(u) \eta(u|u_L)) 
\end{equation}
satisfies the dissipation bound~\eqref{diss:shock} for all $u\in \partial \Pi$.
We note that for generic values of $a_1/a_2$ the dissipation $D_\cont$ cannot be shown to be non-positive.
For instance, in the case that $a_1 = a_2$ we find $\partial \Pi$ is a hyperplane of states equidistant from $u_L,u_R$. 
For $n \geq 2$ this is an unbounded set and, for generic systems, we find $D_{\cont}$ attains positive values on $\partial \Pi$ \cite{serre_vasseur}.
However, for sufficiently small values of $a_1/a_2$ we see $\Pi$ can be put into an arbitrarily small neighborhood of $u_L$, in which case negativity can be shown. 
Kang and Vasseur did exactly this and showed that, for a sufficiently small $a_1/a_2$ we obtain the dissipation~\eqref{eq:diss-function} is non-positive \cite{MR3519973} (and the third author of this paper later showed the upper bound \eqref{diss:shock} for large shocks \cite{move_entire_solution_system} by similar means), however without the control on our weights~\eqref{control_a_one} and~\eqref{control_a_two}.
Having such a control on the weights is vital to take the stability of individual shocks to the stability of systems via front tracking. 
Later Golding, Vasseur, and the third author of this manuscript were able to prove that the dissipation function~\eqref{eq:diss-function} is non-positive for weights satisfying~\eqref{control_a_one} and~\eqref{control_a_two} for purposes of front tracking \cite{MR4667839}; however, they did not establish the quantitative bound by the shift~\eqref{diss:shock}.
As seen in the statement of \Cref{prop:small_shock_diss}, our main result in this section is establishing~\eqref{diss:shock} in the regime of weights considered by Golding, Vasseur, and the third author.


To establish this Proposition, we first prove two sub-propositions: the first for the continuous case $u(h(t)+,t)=u(h(t)-,t)$ and the second in the discontinuous case $u(h(t)+,t)\ne u(h(t)-,t)$. 
In this continuous case, we bound the dissipation~\eqref{eq:Dcont-def} within the larger set 
\begin{equation}\label{eq:Pi-star-def}
    \Pi_{C,s_0}^* := \{u: \tilde \eta(u) < 0\} \cup B_r(u^*)
\end{equation}
for some $r > 0$ to be determined in Lemma \ref{lemma:D-cont-exp}.
We further define 
\begin{align}
    \tilde \eta(u) &:= (1+Cs_0) \eta(u|u_L) - \eta(u|u_R) \label{eq:tilde-eta} \\
    \tilde q(u) &:= (1+Cs_0) q(u;u_L) - q(u;u_R) \label{eq:tilde-q}
\end{align}
where $s_0 = |u_L -u_R|,$ $C$ is such that $a_1/a_2 = 1 + Cs_0$, and $u^*$ is the maximizer of $D_{\cont}$ within the original set $\Pi$; this state was shown to exist and be unique for $C$ sufficiently large and $s_0$ sufficiently small by \cite[Proposition 2.1]{MR4667839}.
With this notation our continuous dissipation function can now be concisely written as 
\begin{equation}
    D_{\cont}(u) = -\tilde q(u) + \lambda_1(u)\tilde \eta(u).
\end{equation}
With these definitions in hand we can now state our first proposition. 
\begin{proposition}\label{prop:cont}
    Consider a system~\eqref{system} and phase space $\Nu$ verifying Assumption \ref{assum}. 
    Fix a state $d \in \Nu$. 
    Then there exists a constant $K_{ball} > 0$ such that for any $C_1 > 0$ sufficiently large there exists constants $\tilde s_0,\ K > 0$ such that for any $C \in [C_1/2, 2C_1]$, any 1-shock $(u_L,u_R,\sigma_{LR})$ satisfying $|u_L - d| + |u_R - d| \leq \tilde s_0$, and $|u_L - u_R| = s_0 \leq \tilde s_0$
        and any $u \in \Pi_{C,s_0}^*$ we have the bound 
    \begin{equation} \label{eq:cont-bound}
        D_{\cont}(u) \leq -Ks_0 ( | u - u_L |^2 + | u - u_R|^2 ),
    \end{equation}
    where $\Pi_{C,s_0}^*$ is defined with $r = K_{ball}C^{-1}$ by equation~\eqref{eq:Pi-star-def}.
\end{proposition}
This proves the dissipation bound of \Cref{prop:small_shock_diss} in the case of $u_- = u_+ \in \Pi_{C,s_0}^*$. The proof of \Cref{prop:cont} is saved for \Cref{sec:cont}.
We note that there is a slight deviation between Proposition \ref{prop:cont} and the corresponding proposition proven in \cite{MR4667839}.
They showed~\eqref{eq:Dcont-def} was non-positive within the smaller set $\Pi_{C,s_0}$, while we have included the new region $B_r(u^*)$. 
The reason for this is a technical detail arising from the construction of our shift: in this continuous case, we wish to move as a generalized characteristic $\dot h(t) = u(h(t),t)$ when $u$ is within some fixed set (say $\Pi_{C,s_0}^*$) where we are able to show the bound~\eqref{eq:cont-bound}.
While we can do this within $\Pi_{C,s_0}$, there are also larger subsets of $\Nu$ where we can achieve this bound. 
If our selected region contains $\Pi$ then $\tilde \eta(u) > 0$ outside of it, hence selecting a sufficiently large $\dot h$ gives a sign to our dissipation~\eqref{eq:diss-function} for all states $u$ outside our fixed region.
If we do this for the set $\Pi$ as in \cite{MR4667839} an issue arises near $\partial \Pi$: on this set our dissipation~\eqref{eq:diss-function} is independent of our choice of $\dot h$, while our desired upper bound~\eqref{diss:shock} is not independent of $\dot{h}$ for states $u \notin \Pi$ near the boundary. 
As a result, if we select a sequence of states $u \to u^*$ (where $u^* \in \partial \Pi$ is defined in Proposition \ref{prop:gkv-cont}) that are outside of $\Pi$ we find $-\tilde q(u) + \dot h(t) \eta(u) \to D_{\cont}(u^*) \sim -s_0^3$. 
However, if we select the same shift as \cite{MR4667839} we find $|\sigma_{LR} - \dot h| = \bigO(1)$ when we encounter these states, preventing us from satisfying the bound~\eqref{diss:shock} for small shocks with this choice of shift. 
The solution to this is to prescribe that our shift move as a generalized characteristic for any state $u \in B_r(u^*)$ where $r$ is independent of $s_0$. 
By doing this, we can guarantee that for $u$ near $u^*$ we instead have $|\sigma_{LR} - \dot h| = |\sigma_{LR} - \lambda_1(u)|\sim s_0$, and for states sufficiently far from $u^*$ we have $D_\cont(u) \lesssim -s_0$, allowing us to prescribe our shift move at a high (but bounded) velocity.

Next, since we prescribe $\dot h$ to move at times as a generalized characteristic, it is possible that $h(t)$ exactly follows a shock $(u_-,u_+,\sigma_\pm)$ for a positive time. 
When this occurs, we encounter the entropy dissipation 
\begin{equation} \label{eq:Drh-def}
    D_{\RH}(u_-,u_+,\sigma_\pm) = [q(u_+;u_R) - \sigma_\pm \eta(u_+ | u_R)] - (1 + Cs_0)[q(u_-;u_L) - \sigma_\pm \eta(u_- | u_L)].
\end{equation}
We then need to prove our bound~\eqref{diss:shock} for the dissipation function~\eqref{eq:Drh-def}.
Thankfully, by choice of $h$ it suffices to consider only the 1-shock case, with $u_- \in \Pi_{C,s_0}^*$.

\begin{proposition}\label{prop:shock}
    Consider a system~\eqref{system} and phase space $\Nu$ verifying Assumption \ref{assum}. 
    Fix a state $d \in \Nu$. 
    Then there exists a constant $K_{ball} > 0$ such that for any $C_1 > 0$ sufficiently large there exists an $\tilde s_0 > 0$ and $K>0$
        such that for any $C \in [C_1/2, 2C_1]$, any 1-shock $(u_L,u_R,\sigma_{LR})$ satisfying $|u_L - d| + |u_R - d| \leq \tilde s_0$, and $|u_L - u_R| = s_0 \leq \tilde s_0$
        and any 1-shock $(u_-,u_+, \sigma_\pm)$ with $u_- \in \Pi_{C,s_0}^*$ we have the bound 
    \begin{equation}
        D_{\RH}(u_-,u_+,\sigma_\pm) \leq -Ks_0 ( | u_- - u_L |^2 + | u_+ - u_R|^2 ),
    \end{equation}
    where $\Pi_{C,s_0}^*$ is defined with $r = K_{ball}C^{-1}$ by equation~\eqref{eq:Pi-star-def}.
\end{proposition}
The proof of \Cref{prop:shock} is saved to \Cref{sec:shock-proof}.

\subsection{Proof of \Cref{prop:small_shock_diss}}
In this section we show that Proposition \ref{prop:shock} implies \Cref{prop:small_shock_diss}. 
We construct the shift function $h(t)$ using the framework developed in \cite{move_entire_solution_system} and used by \cite{MR4667839} which gives control of $\dot h(t)$. 

Note that for any system~\eqref{system} satisfying Assumption \ref{assum} the change of variables $y = -x$ creates another system $u_t - [f(u)]_y = 0$ where $1$-shocks for system~\eqref{system} correspond to $n$-shocks of this new system. 
As a result, it suffices to consider only $1$-shocks in this proof. 
In particular, this change of variables establishes equations~\eqref{control_a_two} and~\eqref{control_two_shock1020} for 2-shocks from the corresponding equations~\eqref{control_a_one} and~\eqref{control_one_shock1020} for 1-shocks. We also remark that this proof holds if our system does not satisfy Assumption \ref{assum} \ref{assum:convex_left_eign}, but does satisfy all other assumptions. 

Additionally, in the case of (Problem Isothermal) the proof follows almost identically to that of (Problem $\epsilon$-BV) with the following modifications:
\begin{enumerate}
    \item In this case, we measure our shock size in terms of the quantity $2\ln \rho$. 
Since we are restricted to considering shocks within the relatively compact set $B_{2\epsilon}(d) \subset \Nu$ we have uniform continuity of $\ln \rho$. Additionally, the shock curves at a fixed state $u$ are functions of the variable $\ln \rho$ (see \Cref{lem_3.1} for details on the structure of the shock curves for the isothermal Euler system).
This establishes that distance with respect to $2\ln \rho$ is equivalent to Euclidean distance for our purposes.
As a result, it suffices to only consider the case of (Problem $\epsilon$-BV) in the forthcoming proof, where we use the Euclidean distance. 

    \item We wish to show that given a compact set $\mathcal{K} \subset \Nu$ the constants $\epsilon,\ K,\ C_1 > 0$ can be selected so that \Cref{prop:small_shock_diss} holds with any state $d \in \mathcal{K}$. 
    To establish this, we note that the function $C_0 = C_0(d)$ (the constant for which \Cref{prop:shock} holds for all $C_1 > C_0$) depends continuously on $d$, and hence is bounded above over $\mathcal{K}$.
    The proof of this is not difficult, the interested reader can consult \cite[p.~25]{2024arXiv241103578C} where this was explicitly carried out for the qualitative stability estimate. 
    From there, the extension to all of $\mathcal{K}$ follows by Lebesgue's covering lemma. 
    
    
\end{enumerate}
Finally, we remark that in the (Problem $\epsilon$-BV) case \Cref{prop:small_shock_diss} measures the distance $s_0$ as the arclength between states $u_L,\ u_R$ on the shock curve $S_{u_L}^1$. 
We note that since we are restricting ourselves to small shocks this measurement is equivalent to the Euclidean distance $|u_L - u_R|$ when $\epsilon > 0$ is suitably small, allowing us to use $s_0 = |u_L - u_R|$ in the forthcoming proof.

The proof of \Cref{prop:small_shock_diss} is divided into four steps: selecting constants, constructing our shift, proving the dissipation bound~\eqref{diss:shock}, and showing the control on the shift~\eqref{control_one_shock1020}.

\vskip0.3cm
\noindent
\textbf{Step 1: Fixing constants.} 
    For any $d\in \mathcal{V}$ there exists $\epsilon_d,\alpha_1 > 0$ such that 
    \begin{equation} \begin{aligned}\sup_{u\in B_{2\epsilon_d}(d)} \lambda_1(u) < \inf_{u\in B_{2\epsilon_d}(d)} \lambda_2(u) \end{aligned}\label{eq:sep-char-vel} \end{equation}
    due to Assumption \ref{assum} \ref{assum:hyperbolic}, we have the containment $B_{2\epsilon_d}(d) \subset \Nu$, and we satisfy Lemma \ref{lem_hugoniot} within the ball $B_{2\epsilon_d}(d)$.
    We note by the control given by Lemma \ref{lem:gkv_eta} for $C_1$ sufficiently large $\Pi_{C,s_0}$ is compactly contained within $B_{4\epsilon/3}(d)$ when both $u_L,\ u_R \in B_{\epsilon_d}(d)$. 
    We can then further institute $C_1$ is large enough to verify Proposition \ref{prop:shock} and to ensure $K_{ball}C_1^{-1} < \epsilon_d/3$, resulting in $\Pi_{C,s_0}^*$ being compactly contained within $B_{5\epsilon_d/3}(d)$ by the definition of $\Pi^*_{C,s_0}$ given by Proposition \ref{prop:shock}.
    With $C_1$ fixed, Proposition \ref{prop:shock} grants us a constant $\tilde s_0$ giving the range of shock sizes our estimate holds for. 
    We now select 
    \begin{equation} \epsilon = \min(\tilde s_0, \epsilon_d/6). \label{eq:set-epsilon-3.1}\end{equation}
    This fixes all constants of \Cref{prop:small_shock_diss} except for the dissipation coefficient $K$ and shift bound $\hat \lambda$ which we delay until Steps 3 and 4, respectively.

    For any $1$-shock $(u_L,u_R,\sigma_{LR})$ satisfying $|u_L - d| + |u_R - d| \leq \epsilon$, we define $s_0 := |u_R - u_L|$. 
    From here, we define $C>0$ to satisfy $$ \frac{a_1}{a_2} = 1 + Cs_0.$$
    We note that by~\eqref{control_a_one} $C$ satisfies $C_1 / 2 < C < 2C_1$, so this weight verifies \Cref{prop:shock}.

    \vskip0.3cm
\noindent
\textbf{Step 2: Construction of the shift.} 
Before we begin, we recall the following lemma:
\begin{lemma}[Lemma 4.1 of \cite{MR4667839}] \label{lemma:gkv_eta_control_q}
    There exists $\delta > 0$ and $C^* > 0$ such that for any $u_L,u_R \in \mathcal{V}$ satisfying $|u_L - d| + |u_R - d| \leq \tilde s_0$ we have the bounds 
    \begin{align*}
        \tilde\eta(u) &\geq \delta s_0 d(u,\partial \Pi_{C,s_0}) && \text{for any $u \in \Nu \setminus \Pi_{C,s_0}$, and}\\
        |\tilde q(u)| &\leq C^* | \tilde \eta(u)| && \text{when $u \in \Nu \setminus \Pi_{C,s_0}$ and $\tilde q(u) \leq 0$.}
    \end{align*}
\end{lemma}
This gives us control on the dissipation by making the shift sufficiently large for states $u \notin \Pi_{C,s_0}$. 

Fix a shock $(u_L,u_R,\sigma_{LR})$ and weight $a_1,a_2 > 0$ satisfying the hypotheses of \Cref{prop:small_shock_diss}. 
This then also fixes the parameters $C,s_0 > 0$ by Step 1. 
We define the set 
$$ A = \{u\in \Nu: a_1 \eta(u|u_L) > a_2 \eta(u|u_R)\text{ and }|u - u^*| > r\} = \Nu \setminus \overline{\Pi_{C,s_0}^*}.$$
This set is open within $\Nu$ and hence $-\chi_A$ is upper semicontinuous. 
We now define the function 
$$ V(u) := \begin{cases}
    \lambda_1(u) - (C^* - 2L)\chi_A(u) & u \in \Nu \\
    -C^* - L & u\notin \Nu
\end{cases} $$
where $L$ is the constant defined in Assumption \ref{assum} \ref{assum:bounded-char-speed}, we select $C^*$ to be sufficiently large to both satisfy Lemma \ref{lemma:gkv_eta_control_q}, and
\begin{equation}\label{eq:c-star-bound}
     \norm{f}_{C^2(B_{2\epsilon_d}(d))}/\Gamma   \leq C^*
\end{equation}
where
\begin{equation}\label{eq:Gamma-def}
    \begin{aligned}\Gamma = \inf_{\substack{u \in B_{2\epsilon_d}(d),\\|v| = 1 }} \nabla^2 \eta(u)v\cdot v.\end{aligned}
\end{equation} 
We now take our shift to solve the following ODE
\begin{equation} \label{eq:shift-ode} \begin{cases}\begin{aligned}
    \dot h(t) &= V(u(h(t), t)), \\
    h(T_{start}) &= x_0\end{aligned}
\end{cases} \end{equation}
where we note that the right-hand side is discontinuous. 
The existence of such a solution in the Filippov sense was shown by the following lemma, proven in \cite[Proposition 1]{Leger2011}.
\begin{lemma}[Existence of Filippov flows]\label{lemma:filippov}
    Let $V: \Nu \to \R$ be bounded and upper semi-continuous on $\Nu$ and continuous on $U$, an open full measure subset of $\Nu$. 
    Let $u \in \S_{weak}$. 
    Then, there exists a Lipshitz $h$ a solution to the ODE~\eqref{eq:shift-ode} in the Filippov sense. 
    That is, for almost all $t > T_{start}$ we have 
    \begin{equation}\label{eq:Filippov-control}\begin{aligned}
        V_{\min}(t) &\leq \dot h(t) \leq V_{\max}(t) \\
        h(T_{start}) &= x_0 \\
        \text{Lip}[h] &\leq \norm{V}_{L^\infty}
    \end{aligned}\end{equation}
    where $u_\pm(t) = u(h(t) \pm , t)$ are the traces of $u$ along our shift $h$ given by Definition \ref{strong_trace_prop}, 
    $$V_{\max}(t) = \max(V(u_+(t)), V(u_-(t))),$$
    and 
    $$V_{\min}(t) = \begin{cases}\min(V(u_+(t)), V(u_-(t))) & \text{if }u_+(t),u_-(t) \in U,\\
    -\norm{V}_{L^\infty}&\text{otherwise.} \end{cases}$$
    Furthermore, for almost every $t > T_{start}$ the shift satisfies the Rankine-Hugoniot equations
    \begin{align}
        f(u_+) - f(u_-) &= \dot h(t)(u_+ - u_-) \\
        q(u_+) - q(u_-) &\leq \dot h(t)(\eta(u_+) - \eta(u_-)) 
    \end{align}
    establishing the dichotomy that either $u_-(t) = u_+(t)$ or $(u_-(t),u_+(t),\dot h(t))$ is an entropic shock for almost all $t > 0$. 
\end{lemma}

\vskip0.3cm
\noindent
\textbf{Step 3: Proof of~\eqref{diss:shock}.}
Maintaining the notation $u_\pm(t) = u(h(t) \pm , t)$ we consider the following four cases:
\begin{figure}[H]
    \center
    \begin{tabular}{lccc}
        \textbf{Case 1.} & $u_- \in \Pi_{C,s_0}^*$ & and &  $u_+ \in \Pi_{C,s_0}^*$, \\
        \textbf{Case 2.} & $u_- \in \Pi_{C,s_0}^*$ & and & $u_+ \notin \Pi_{C,s_0}^*$,\\
        \textbf{Case 3.} & $u_- \notin \Pi_{C,s_0}^*$ & and & $u_+ \in \Pi_{C,s_0}^*$,\\
        \textbf{Case 4.} & $u_- \notin \Pi_{C,s_0}^*$ & and & $u_+ \notin \Pi_{C,s_0}^*$.
    \end{tabular}\end{figure}
Since it suffices to show~\eqref{diss:shock} for almost all times we can restrict ourselves by Lemma \ref{lemma:filippov} to times where the following dichotomy holds:
\begin{equation}\label{eq:dia} u_-(t) = u_+(t) \quad\text{or}\quad (u_-(t),u_+(t),\dot h(t))\text{ is an entropic shock}.\end{equation}
Furthermore, Lemma \ref{lemma:filippov} gives us the following bounds on our shift 
\begin{equation} \label{eq:h-max}
    \dot h(t) \leq V_{\max}(t) = \max\left(\lambda_1(u_+) - (C^* + 2L)\chi_{\Nu \setminus \overline{\Pi_{C,s_0}^*}}(u_+),\ \lambda_1(u_-) - (C^* + 2L)\chi_{\Nu \setminus \overline{\Pi_{C,s_0}^*}}(u_-)\right),
\end{equation}
and for $u_+,u_- \notin \partial \Pi_{C,s_0}^*$, 
\begin{equation} \label{eq:h-min}
    \dot h(t) \geq V_{\min}(t) = \min\left(\lambda_1(u_+) - (C^* + 2L)\chi_{\Nu \setminus \overline{\Pi_{C,s_0}^*}}(u_+),\ \lambda_1(u_-) - (C^* + 2L)\chi_{\Nu \setminus \overline{\Pi_{C,s_0}^*}}(u_-)\right).
\end{equation}
\vskip0.3cm
\noindent
\textbf{Case 1. } 
    In this case we find by equation~\eqref{eq:h-max} 
    $$ \dot h(t) \leq \max(\lambda_1(u_+),\lambda_1(u_-)).$$
    In the case that $u_+ \ne u_-$ we know by the dichotomy~\eqref{eq:dia} that $(u_-,u_+,\dot h(t))$ is an entropic shock. 
    Further, by our construction in Step 1 know that within $\Pi_{C,s_0}^*$ we have separation of characteristic velocities by~\eqref{eq:sep-char-vel}.
    As a result, if $u_- \ne u_+$ we find $\dot h \geq \lambda_1(u_+)$ by Assumption \ref{assum} \ref{assum:1-shock-lax-cond-greater}. 
    Applying this to~\eqref{eq:h-max} we then see $\dot h \leq \lambda_1(u_-)$ and by Assumption \ref{assum} \ref{assum:1-shock-lax-cond-le} we see $(u_-,u_+,\dot h(t))$ must be a 1-shock. 
    In either of theses cases, we can apply Proposition \ref{prop:shock} to retrieve the bound 
    \begin{equation} \label{eq:case-1-intermediate}
        \begin{aligned}
            a_2[q(u_+;u_R) - \dot h(t) \eta(u_+|u_R)] - a_1[q(u_-;u_L) - \dot h(t) \eta(u_-|u_L)] \leq -a_2K s_0(|u_- - u_L|^2 + |u_+ - u_R|^2).
        \end{aligned}
    \end{equation}
    From here, we recall that Lemma \ref{lem_hugoniot} gives us that the map $(u,s)\mapsto \sigma_u^1(s)$ is Lipschitz on the set $S = \{(u,s):u \in \Pi_{C,s_0}^* \text{ and } s < s_u\}$,
        where we take $s_u$ to be the smallest argument such that 
        \begin{equation}\S_u^1(s_u) \in \partial B_{2\epsilon_d}(d). \label{eq:su-def} \end{equation}
    In particular, we recall $\Pi_{C,s_0}^* \subset B_{2\epsilon_d}(d)$ by our choice of parameters $C_1,\ s_0$. 
    Using this and our arclength parameterization of of shock curve given by Assumption \ref{assum} \ref{assum:shock-curve-param} we see 
    \begin{align*}
        |\dot h(t) - \sigma_{LR}| &= |\sigma_{u_-}^1(s) - \sigma_{u_L}^1(s_0)| \\
        &\lesssim |u_- - u_L| + |s - s_0| \\
        &= |u_- - u_L| + \big| |u_+-u_-| - |u_R - u_L|\big| \\
        &\leq |u_- - u_L| + |u_+ - u_R|.
    \end{align*}
    Squaring both sides of the above inequality and substituting into~\eqref{eq:case-1-intermediate} we establish~\eqref{diss:shock}.

\vskip0.3cm
\noindent
\textbf{Case 2. }
    In this case we see that $u_- \ne u_+$ and hence $(u_-,u_+, \dot h(t))$ must be a shock by~\eqref{eq:dia}. 
    By equation~\eqref{eq:h-max} we know that $$\dot h(t) \leq \max\left(\lambda(u_-), \lambda(u_+) - (C^* + 2L) \mathbbm{1}_{\Nu\setminus \overline{\Pi_{C,s_0}^*}} (u_+)\right) \leq \max( \lambda(u_-), - L) \leq \lambda(u_-).$$
    It then follows from Assumption \ref{assum} \ref{assum:1-shock-lax-cond-le} that this shock must be a 1-shock, and hence we can apply Proposition \ref{prop:shock} to again find the bound~\eqref{eq:case-1-intermediate}.
    If $u_+ = \S_{u_-}^1(s)$ for an $s < s_{u_-}$ (as defined in~\eqref{eq:su-def}) then the argument in Case 1 holds identically, establishing ~\eqref{diss:shock}.
    If this is not the case, then by the parameterization given by Assumption \ref{assum} \ref{assum:shock-curve-param}, selection of $\tilde s_0$ in Step 1,  the fact that $|u_- - u_+| \geq d(\Pi_{C,s_0}^*, \Nu\setminus B_{2\epsilon_d}(d)) \geq \epsilon_d/3$, and $|u_R - u_L| < \epsilon < \epsilon_d/6$ by equation~\eqref{eq:set-epsilon-3.1} we can establish the lower bound
    $$ |u_- - u_L| + |u_+ - u_R| \geq ||u_+-u_-|- s_0| \geq d(u_-, \Nu\setminus B_{2\epsilon_d}(d)) - | u_R - u_L | \geq \frac{1}{6}\epsilon_d.$$
    Additionally, Assumption \ref{assum} \ref{assum:bounded-char-speed}, \ref{assum:1-shock-lax-cond-greater}, and \ref{assum:1-shock-lax-cond-le} give the upper bound 
    $$ |\dot h(t) - \sigma_{LR} | \leq |\sigma_{\pm}| + |\sigma_{LR} | \leq 2L.$$
    Combining these considerations with~\eqref{eq:case-1-intermediate}, we find 
    \begin{equation}
        \begin{aligned}
            a_2[q(u_+;u_R) - \dot h(t) \eta(u_+|u_R)] - a_1[q(u_-;u_L) - \dot h(t) \eta(u_-|u_L)] &\leq -a_2K s_0(|u_- - u_L|^2 + |u_+ - u_R|^2)\\
            &\leq -\frac{a_2}{8}K\epsilon_d^2 s_0 \\
            &\leq -\frac{a_2}{16L^2}K\epsilon_d^2 s_0|\sigma_{LR} - \dot h(t)|^2
        \end{aligned}
    \end{equation}
    establishing~\eqref{diss:shock}. 

\vskip0.3cm
\noindent
\textbf{Case 3. }
    In this case $u_- \ne u_+$, so $(u_-,u_+, \dot h(t))$ must be a shock by~\eqref{eq:dia}. 
    From~\eqref{eq:h-max} we find 
    $$\dot h(t) \leq \max(\lambda_1(u_+), \lambda_1(u_-) - (C^* - 2L)) \leq \max(\lambda_1(u_+), -L) = \lambda_1(u^+).$$
    This, however, violates Assumption \ref{assum} \ref{assum:1-shock-lax-cond-greater}. Hence this case is never encountered. 

\vskip0.3cm
\noindent
\textbf{Case 4. }
    In this case we have by~\eqref{eq:h-max} that 
    \begin{equation} \dot h(t) \leq \max(\lambda_1(u_+), \lambda_1(u_-)) - (C^* + 2L) \leq - C^* - L < \inf_{u \in \Nu} \lambda_1(u) \label{eq:case-4-h-bound} \end{equation}
    where the last inequality is simply by the definition of $L$. 
    We can conclude from this that $u_- = u_+$ by the dichotomy~\eqref{eq:dia} and Assumption \ref{assum} \ref{assum:1-shock-lax-cond-greater}. 
    Let $u = u_- = u_+$ and define the new dissipation function 
    \begin{equation}\label{eq:Dcont-prime}
        \tilde D_{\cont}(u) = -\tilde q(u) - (C^* + L) \tilde \eta(u). 
    \end{equation}
    For this dissipation we can show an upper bound of order $-s_0$. 
    \begin{lemma}\label{lemma:outside-pi-bound}
        There exists a constant $M > 0$ independent of $s_0$ (but depending on the system, $d$, and $C_1$) such that for all $u \in \Nu \setminus \Pi_{C,s_0}^*$
        $$ \tilde D_{\cont}(u) < -Ms_0.$$
    \end{lemma}
    We first show how Lemma \ref{lemma:outside-pi-bound} implies~\eqref{diss:shock}.
    Note that by Lemma \ref{lemma:filippov} and Assumption \ref{assum} \ref{assum:bounded-char-speed} we have the bound
    $$ |\dot h(t) - \sigma_{LR}| \leq C^* + 4L.$$ 
    Furthermore, the dissipation function $\tilde D_\cont$ bounds~\eqref{eq:diss-function} in this regime, due to the shift bound~\eqref{eq:case-4-h-bound} and $\tilde \eta(u) \geq 0$ for all $u \notin \Pi_{C,s_0}^*$.
    As a result, we find 
    $$ D_\cont(u) \leq \tilde D_\cont(u) < -Ms_0 \leq -\frac{M}{(C^* + 4L)^2}s_0|\dot h(t) - \sigma_{LR}|^2$$
    establishing~\eqref{diss:shock} in this final case. 

\begin{proof}[Proof of Lemma \ref{lemma:outside-pi-bound}]
    We note that for any state $v \in \partial \Pi_{C,s_0}$ we have $\tilde \eta(v) = 0$ and as a result $D_\cont(v) = \tilde D_\cont(v)$.
    Applying Proposition \ref{prop:shock} we have the bound
    \begin{equation} \label{eq:case-4-boundary-bound}
        \tilde D_\cont(v) \leq -Ks_0( |v - u_L|^2 + | v - u_R|^2 ) \quad \text{for all} \quad v \in \partial \Pi_{C,s_0}.
    \end{equation}
    We now define for $u \in \Nu \setminus \overline{\Pi_{C,s_0}^*}$ the state
    $$\overline u = \underset{v \in \partial \Pi_{C,s_0}}{\text{argmin}} |u-v|. $$
    Applying the fundamental theorem of calculus along the segment connecting $u,\overline u$ we have the expression 
    \begin{equation} \label{eq:case-4-dcontprime-bound}
        \tilde D_\cont(u) = \tilde D_\cont(\overline u) + \nabla \tilde D_\cont(\overline u)(u-\overline u) + \frac{1}{2} \int_{\overline u}^u \nabla^2 \tilde D_\cont(v)(v-\overline u)\,dv.
    \end{equation}
    We easily compute
    $$\nabla D_{\cont}(\overline u) = -\nabla \tilde \eta(\overline u)(f'(\overline u) + LI) - C^*\nabla \tilde \eta(\overline u).$$
    We note by definition of $\overline u$ that $u - \overline u \parallel \nu(\overline u)$ and, by Lemmas \ref{lem:gkv_eta}, \ref{lem:gkv_nu}, we find $\nabla \tilde \eta(\overline u)\cdot (u-\overline u) > 0$. 
    This and the fact that $f' + LI$ is positive definite allows us to conclude that $\nabla \tilde D_\cont(\overline u)(u-\overline u) \leq 0$. 
    Next, we restrict ourselves to considering only $u \in B_{2\epsilon_d}(d) \setminus\Pi_{C,s_0}^*$. 
    In this set $\eta$ is uniformly convex, satisfying $\nabla^2 \eta(u)v\cdot v \geq \Gamma|v|^2$ where $\Gamma > 0$ is defined by~\eqref{eq:Gamma-def}.
    Computing the hessian in the integrand term of~\eqref{eq:case-4-dcontprime-bound} we find
    \begin{align*} 
        \nabla^2  \tilde D_{\cont}(u) =& \nabla^2(-\tilde q(u) - (C^* + L)\tilde \eta(u)) \\
        =& -\nabla^2 \tilde \eta(u)( f'(u) + (C^* + L)I) - \nabla \tilde \eta(u)f''(u).
    \end{align*}
    We now recall that $f'(u) + LI$ is positive semidefinite, $|\nabla \tilde \eta(u)| \lesssim Cs_0(u-u_L) + s_0$ by Lemma \ref{lem:gkv_eta}, and by Lemma \ref{lem_hugoniot} $\norm{f}_{C^2(B_{2\epsilon_d}(d))}$ is bounded to find
    \begin{align*}
        \nabla^2 \tilde D_{\cont}(u)v\cdot v &\lesssim \left(-\Gamma C^* s_0 + (Cs_0|u-u_L| + s_0)\norm{f}_{C^2(B_{2\epsilon_d}(d))}\right)|v|^2\\
        &\lesssim s_0\left(-\Gamma C^* + \norm{f}_{C^2(B_{2\epsilon_d}(d))}\right)|v|^2 \leq 0
    \end{align*}
    where the last inequality follows due to $C^*$ to satisfying~\eqref{eq:c-star-bound}. 
    This shows that within $B_{2\epsilon_d}(d)$ we have the bound 
    $$ \tilde D_\cont(u) \leq \tilde D_\cont(\overline u)$$ for the projection $\overline u \in \partial \Pi_{C,s_0}$. 
    We note that if $\overline u \notin \partial \Pi_{C,s_0}^*$ then there exists $v \in \partial B_r(u^*)$ on the segment connecting $\overline u, u$. 
    At this $v$ we note that since $\tilde \eta (v) \geq 0$ we have the bound
    $$ \tilde D_\cont(v) = D_\cont(v) -(C^* + L + \lambda_1(v))\tilde \eta(v) \leq D_\cont(v).$$
    Repeating our preceding argument, 
    \begin{align*} 
        \tilde D_{\cont}(u) =& \tilde D_\cont(\overline u) + \nabla \tilde D_\cont(\overline u)(v-\overline u) + \nabla \tilde D_\cont(\overline u)(u-v) \\
        &+ \frac{1}{2} \int_{\overline u}^v \nabla^2 \tilde D_\cont(v)(v-\overline u)\,dv + \frac{1}{2} \int_{v}^{\overline u} \nabla^2 \tilde D_\cont(v)(v-\overline u)\,dv \\
        =& \tilde D_\cont(v) + \nabla \tilde D_\cont(\overline u)(u-v) + \frac{1}{2} \int_{v}^{\overline u} \nabla^2 \tilde D_\cont(v)(v-\overline u)\,dv \\
        \leq& \tilde D_\cont(v).
    \end{align*}
    From this and the bound $D_{\cont}(u) \leq -Ks_0|u-u^*|$ established in the proof of Proposition \ref{prop:cont} we have 
    $$ \sup_{u \in B_{2\epsilon_d}(d) \setminus \Pi_{C,s_0}^*} \tilde D_\cont(u) \leq -Ks_0r^2 \leq -KK_{ball}^2C_1^{-2}\frac{s_0}{4} $$
    by the selection $r = K_{ball}C \geq C_1/2$. 
    For states $u$ in $\Nu \setminus B_{2\epsilon_d}(d)$ we recall Lemma \ref{lemma:gkv_eta_control_q}, which establishes the bound
    $$ \tilde D_\cont(u) = -\tilde q(u) -(C^* + L)\tilde \eta(u) \leq -L\tilde \eta(u).$$
    Using convexity along the line connecting $u$ to it's projection $\overline u$  on $\partial \Pi_{C,s_0}$, Lemma \ref{lem:gkv_eta}, and our construction of $\Pi_{C,s_0}$ in Step 1 we have the bound
    $$ \tilde D_\cont(u) \leq -L(\tilde \eta(\overline u) + \nabla \eta(\overline u)(u-\overline u)) \lesssim -Ls_0|u-\overline u| \leq -Ls_0 d(\partial \Pi_{C,s_0}^*, \Nu\setminus B_{2\epsilon_d}(d)) \leq -L\epsilon_d s_0.$$
    Combining these two cases, we have shown there exists a universal constant $M > 0$ independent of $s_0$ such that 
    $$ \sup_{u \in \Nu \setminus \Pi_{C,s_0}^*} \tilde D_\cont(u) < -Ms_0. $$
\end{proof}
\vskip0.3cm
\noindent
\textbf{Step 4: Proof of~\eqref{control_one_shock1020}.}
In each of the four cases considered in Step 3 we found 
\begin{equation}
    \dot h(t) \leq \inf_{u \in \Pi_{C,s_0}^*} \lambda_1(u) 
\end{equation}
which establishes the upper bound of~\eqref{control_one_shock1020} for almost all $t > T_{start}$, due to the containment $\Pi_{C,s_0}^* \subset B_{2\epsilon_d}(d)$.
For the lower bound~\eqref{eq:Filippov-control} of Lemma \ref{lemma:filippov} gives us that $|\dot h(t)| \leq \norm{V}_{L^\infty} \leq C^* + 3L.$
Hence taking $\hat \lambda =2(C^* + 3L)$ we establish the lower bound of~\eqref{control_one_shock1020}

\subsection{Proof of Proposition \ref{prop:cont}} \label{sec:cont}
\subsubsection{Proof Sketch}
In this section we prove Proposition \ref{prop:cont}. 
We take $(u_L,u_R,\sigma_{LR})$ to be an entropic $1$-shock of~\eqref{system} with $u_R = S_{u_L}^1(s_0)$ and $u_L,\ u_R\in B_{\epsilon}(d) $ for $d,\ \epsilon$ satisfying equation~\eqref{eq:depsilon}.
We further assume that $u_L,u_R \in B_{\epsilon}(d)$ and that $C,s_0$ are sufficiently large to ensure $\Pi_{C,s_0}^* \subset B_{5\epsilon/3}(d)$ by Lemma \ref{lem:gkv_eta} and our choice of $r$.
We begin by recalling the bound on the continuous dissipation attained by the authors of \cite{MR4667839}. 
\begin{proposition}[Proposition 2.1 from \cite{MR4667839}] \label{prop:gkv-cont}
    Under the assumptions of Proposition \ref{prop:cont} there is a  universal constant $K$ depending only on the system and $d$, such that for any $C_1>0$ large enough, there exists  $\tilde{s}_0(C_1),\,\tilde s_0>0$,   such that for any $C>0$ with $C_1/2\leq C\leq 2C_1$, for any  1-shock $u_L,u_R$  with $|u_L-d|+|u_R-d|\leq \tilde s_0$, and $|u_L-u_R|=s_0$ with 
    $0 < s_0 \leq \tilde s_0$, there exists a unique local maximum $u^* \in \overline{\Pi_{C,s_0}}$ of $D_{\cont}$ and at this state we have the estimate
\begin{equation} \label{eq:gkv-cont-bound}
D_{\cont}(u^*) \le -Ks_0^3.
\end{equation}
\end{proposition}
We remark that the bound~\eqref{eq:gkv-cont-bound} satisfies the desired dissipation estimate of Proposition \ref{prop:cont}.
This can be seen by the distance estimates given in Lemma \ref{lemma:gkv-dist-u0} which show $|u^* - u_L| \lesssim s_0$.
This allows us to use the bound at $u^*$ given by Proposition \ref{prop:gkv-cont} as a base to build our proof.
We prove Proposition \ref{prop:cont} in two steps:
\begin{description}
    \item[Step 1] We show that there exists a universal constant $K_{ball} > 0$ that allows us to find an intermediate upper bound $D_{\cont}(u) \lesssim -s_0^3 - s_0|u-u^*|^2$ for all $u \in B_{K_{ball}C^{-1}}(u^*)$, where the state $u^*$ is the maximizer found in Proposition \ref{prop:gkv-cont}. In this step, we select $r := B_{K_{ball}C^{-1}}$, finalizing our definition of $\Pi^*_{C,s_0}$. 
    \item[Step 2] Next, we extend this bound to the region $\Pi_{C,s_0}$ via a scaling argument. 
\end{description}
We comment that as we have restricted ourselves so that $\Pi_{C,s_0}^* \subset B_{2\epsilon}(d)$, we enjoy all the regularity properties of Lemma \ref{lem_hugoniot}. 
In particular, $\norm{f}_{C^4\left(\Pi_{C,s_0}^*\right)} \leq K$ depends only on $f$ and $d$ while the $C^3$ norm of $r_1,\ l_1,\ \lambda_1,\ \eta, q$ on $\Pi_{C,s_0}^*$ are uniformly bounded.
This allows us to use classical techniques from differential and integral calculus in the proof. 
\subsubsection{Step 1: Bound near $u^*$}

\begin{lemma}[Lemmas 5.4, 5.12 of \cite{MR4667839}] \label{lemma:taylor-bounds-d-cont}
    For $C$ sufficiently large and $s_0$ sufficiently small we have the bounds
    \begin{equation}
        \begin{aligned}
            \nabla D_{\cont}(u^*) &= 0, \\
            \nabla^2 D_{\cont(u^*)}v\cdot v &\lesssim -s_0 v_1^2 - Cs_0 \sum_{i = 2}^n v_i^2, \text{ and}\\
            |\nabla^3 D_{\cont}(u)| &\lesssim C s_0,
        \end{aligned}
    \end{equation}
    where $u \in B_{2\epsilon}(d)$ and $v \in \R^n$ is of the form $v = \sum_{i=1}^n r_i(u^*)v_i$.
\end{lemma}
\begin{proof}
    To begin, we recall that at $u^*$ we have $\nabla \tilde \eta(u^*) \parallel l_i(u^*)$ and $\tilde \eta(u^*) = 0$ by Lemma \ref{lemma:gkv-max-char}. As a result, 
    $$ \nabla D_{\cont}(u^*) = \nabla \tilde \eta(u^*) (\lambda_1(u^*) I - f'(u^*)) + \tilde \eta(u^*)\nabla \lambda_1(u^*) = 0.$$ 

    Next, we compute
    \begin{align*}
        \nabla^2 D_{\cont}(u^*) =& \nabla^2 \tilde \eta(u^*) (\lambda_1(u^*) I - f'(u^*)) + \nabla \tilde\eta(u^*)(\nabla \lambda_1(u^*) \otimes I + I\otimes \nabla \lambda_1(u^*) - f''(u^*)) + \tilde \eta(u^*)\nabla^2 \lambda_1(u^*)
    \end{align*}
    where we first note $\nabla^2 \tilde \eta(u^*) (\lambda_1(u^*) I - f'(u^*))v\cdot v \lesssim -Cs_0 \sum_{i=2}^n v_i^2$ due to $\lambda_1(u^*) I - f'(u^*)$ being negative definite on $\la v_i\ra_{i=2}^n$ by Assumption \ref{assum} \ref{assum:hyperbolic}.
    Further, $\tilde \eta(u^*)\nabla^2 \lambda_1(u^*) = 0$ due to $u^* \in \partial \Pi_{C,s_0}$ and since $ \nabla \tilde\eta(u^*) \parallel l_1(u^*), \, |\nabla \tilde \eta(u^*)|\gtrsim s_0$ we find by the well known identity
    $$ l_1 f''r_1\cdot r_1 = (l_1 r_1) ( \nabla \lambda_1 r_1 )$$ 
    that 
    $$ \nabla \tilde\eta(u^*)\left(\nabla \lambda_1(u^*) \otimes I + I\otimes \nabla \lambda_1(u^*) - f''(u^*)\right)v \cdot v \lesssim -s_0v_1^2 + s_0\sum_{i < j}c_{ij}|v_iv_j|, $$
    where the entries $c_{ij}$ are independent of $C$, $s_0$. 
    Finally, applying Young's inequality to the product terms we see 
    $$ \nabla \tilde\eta(u^*)(\nabla \lambda_1(u^*) \otimes I + I\otimes \nabla \lambda_1(u^*) - f''(u^*))v \cdot v \lesssim -s_0v_1^2 + s_0\sum_{i =2}c_{ij}v_i^2. $$
    Letting $C$ be sufficiently large, $s_0$ sufficiently small then yields our bound on $\nabla^2 D_{\cont}$. 

    Finally, we compute 
    \begin{equation}\label{eq:D3Dcont}\begin{aligned}
        \nabla^3 D_{\cont}(u) =& Cs_0\nabla^3 \eta(u) (\lambda_1(u) I - f'(u)) + 2Cs_0\nabla^2 \eta(u) (\nabla \lambda_1(u)\otimes I + I\otimes \nabla \lambda_1(u) I - f''(u)) \\
        &+\nabla \tilde \eta(u)  (\nabla^2 \lambda_1(u)\otimes I + I\otimes \nabla^2 \lambda_1(u) I - f'''(u)) + \tilde \eta(u) \nabla^3 \lambda_1(u).
    \end{aligned}\end{equation}
    We recall that Lemma \ref{lem_hugoniot} gives a uniform bounds on $|\nabla^k \eta|,\,|\nabla^k \lambda_1|,$ and $|f^{(k)}|$ within $B_{2\epsilon}(d)$ while the remaining terms satisfy
        $| \tilde \eta(u)| \leq s_0d(u,\partial \Pi_{C,s_0}) + Cs_0 |u-u_L| d(u,\partial \Pi_{C,s_0})$, $| \nabla \tilde \eta(u)| \lesssim s_0 + Cs_0|u - u_L|$ by Lemma \ref{lem:gkv_eta} and Lemma \ref{lem:gkv_dist}.
    Recalling that we are localized within $B_{2\epsilon}(d)$ and substituting the proceeding bounds into equation~\eqref{eq:D3Dcont} we find 
    $$ \nabla^2 D_\cont(u) \lesssim Cs_0 $$
    for $C$ sufficiently large and $s_0$ sufficiently small as desired. 
\end{proof}

\begin{lemma}\label{lemma:D-cont-exp}
    There exists a universal constant $K_{ball} > 0$ such that for all $C$ sufficiently large and $s_0$ sufficiently small we have at all $u \in B_{K_{ball}C^{-1}}(u^*)$ the dissipation bound 
    \begin{equation}\label{eq:intermediate-bound} D_{\cont}(u) \lesssim -s_0^3 - s_0 | u - u^*|^2.
    \end{equation}
\end{lemma}
\begin{proof}
    Expanding $D_\cont$ around $u^*$ we find for $C$ sufficiently large and $s_0$ sufficiently small Proposition \ref{prop:gkv-cont} and Lemma \ref{lemma:taylor-bounds-d-cont} give that 
    \begin{align*}
        D_{\cont}(u) \lesssim& D_{\cont}(u^*) + \nabla D_{\cont}(u^*)(u-u^*) + \frac{1}{2} \nabla^2 D_{\cont}(u^*):(u-u^*)^{\otimes 2} \\
        &+ \frac{1}{6}\sup_{\substack{|u-u^*| \leq K_{ball}C^{-1},\\ u \in B_{2\epsilon}(d) } }  \left|\nabla^3 D_{\cont}(u)\right| |u-u^*|^3 \\
        \lesssim& -s_0^3 - s_0|u - u^*|^2 + s_0 K_{ball} | u - u^*|^2 \\
        =& -s_0^3 - s_0 |u - u^*|^2
    \end{align*}
    where we see our bound holds for all $u \in B_{K_{ball} C^{-1}}(u^*)\cap B_{2\epsilon}(d)$ given $K_{ball}$ is sufficiently small, independent of the parameters $C$ and $s_0$.
    Further, for $C$ sufficiently large we have $B_{K_{ball} C^{-1}}(u^*)\subset B_{2\epsilon}(d)$ completing the proof.
\end{proof}

\subsubsection{Step 2: Extension to all of $\Pi_{C,s_0}^*$}

\begin{figure}[t]
    \centering
    \begin{subfigure}[b]{0.40\textwidth}
        \includegraphics[width=\textwidth]{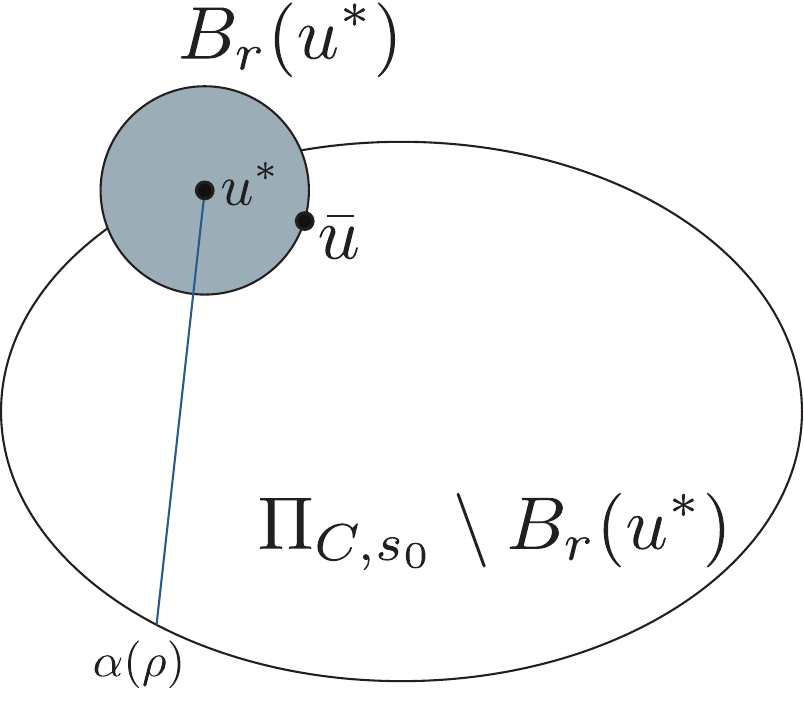}
        \caption{}\label{fig:first_panel}
    \end{subfigure}
    \begin{subfigure}[b]{0.59\textwidth}
        \includegraphics[width=\textwidth]{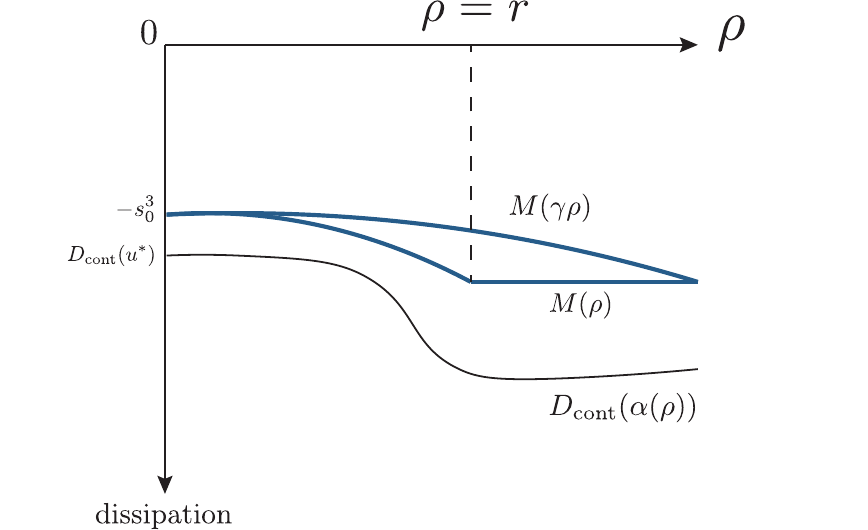}\vspace{.02in}
        \caption{}\label{fig:second_panel}
    \end{subfigure}
    \caption{This figure is a pictorial description of the scaling argument used in the Proof of \Cref{prop:cont}. In \Cref{fig:first_panel} $\alpha$ is a ray connecting $u^*$ to some state on $\partial \Pi_{C,s_0}$. In \Cref{fig:second_panel} we plot the dissipation of states along $\alpha$ with the associated upper bounds.}
    \label{fig:Dcont-extend}
\end{figure}



Henceforth, we set the radius of the ball in $\Pi_{C,s_0}^*$ to $r = K_{ball}C^{-1}$.
In Lemma \ref{lemma:D-cont-exp} we established an improved estimate within the ball of radius $r = K_{ball}C^{-1}$ around $u^*$.
Since the size of both sets is independent of $s_0$, we can now perform a scaling argument to extend the bound of Lemma \ref{lemma:D-cont-exp} to all of $\Pi_{C,s_0}^*$, then show that this bound is controlled by the desired bound of Proposition \ref{prop:cont}.
See \Cref{fig:Dcont-extend} for a pictorial representation of this scaling argument. 
\begin{proof}[Proof of Proposition \ref{prop:cont}]
To extend our bound to states $u \in \Pi_{C,s_0}^* \setminus \overline{B_r(u^*)}$ we recall the following characterization of critical points of $D_{\cont}$ from Lemma \ref{lemma:gkv-max-char}: 
    if $ O \subset \Pi_{C,s_0}$ is open then any local maximizers $u$ of $D_{\cont}$ in $\overline{O}$ necessarily satisfy $u \in \partial O$. 
This ensures that the $\overline u \in \Pi_{C,s_0}^* \setminus \overline{B_r(u^*)}$ maximizing $D_{\cont}$ satisfies 
$$\overline u \in [\partial \Pi_{C,s_0}\setminus \overline{B_r(u^*)}] \cup [\Pi_{C,s_0}\cap \partial B_r(u^*)]. $$
We can further deduce that $\overline u \in \Pi_{C,s_0}\cap \partial B_r(u^*)$, because by Proposition \ref{prop:gkv-cont} there exists a unique local maximizer of $D_\cont$ on $\partial \Pi_{C,s_0}$ and we know we cannot have $\overline u = u^*$ by construction of $\Pi_{C,s_0}^*$.

This shows that for states outside of $B_r(u^*)$ the dissipation is bounded by that of $\overline u$, which is bounded by Lemma \ref{lemma:D-cont-exp} as $|\overline u - u^*| =r$.
So, there exist universal constants $C_1,\ C_2 > 0$ such that 
$$D_{\cont}(u) \leq M(|u-u^*|)$$ where
$$ M(R) = \begin{cases}
    -C_1s_0^3 - C_2s_0R^2 & R \leq r \\
    -C_1s_0^3 - C_2s_0r^2 & R > r
\end{cases} $$
due to $\overline u$ being the maximum over $\Pi_{C,s_0} \setminus B_{r}(u^*)$ and $|\overline u - u^*| = r$. 
We note that $M$ is decreasing in $R$, so for any constant $0 \leq \gamma \leq 1$ we have 
$$ M(R) \leq M(\gamma R).$$
If we choose $\gamma = \min(1,r/\text{diam}(\Pi_{C,s_0})) \sim 1$ we find for all $u \in \Pi_{C,s_0}\setminus B_{r}(u^*)$
\begin{equation} D_{\cont}(u) \leq M(|u-u^*|) \leq M(\gamma |u-u^*|) \leq -C_1s_0^3 - C_2s_0 \gamma^2|u - u^*|^2 \label{eq:cont-int-simp-PROOF} \end{equation}
as $\gamma|u-u^*| < r|u-u^*|/\text{diam}(\Pi_{C,s_0}) \leq r$ for all $u \in \Pi_{C,s_0}$. 
This shows the bound~\eqref{eq:intermediate-bound} on all of $\Pi_{C,s_0}^*$. 

Finally, using estimates from Lemma \ref{lemma:gkv-dist-u0} and $|u_L - u_R| = s_0$ gives us
    $$ |u^* - u_L| + |u^* - u_R| \lesssim s_0.$$
This bound, equation~\eqref{eq:cont-int-simp-PROOF}, and the triangle inequality then yields 
$$ D_{\cont}(u) \lesssim -s_0^3 - s_0|u-u^*|^2 \lesssim -s_0( |u^* - u_L|^2 + |u^* - u_R|^2 + |u-u^*|^2 ) \lesssim -s_0( |u - u_L|^2 + |u - u_R|^2) $$
establishing Proposition \ref{prop:cont}.
\end{proof}

\subsection{Proof of Proposition \ref{prop:shock}}  \label{sec:shock-proof}
\subsubsection{Proof Sketch}
In this section we prove Proposition \ref{prop:shock}.
Our analysis begins by following Section 6 of \cite{MR4667839}, where the authors introduced the maximal dissipation function, 
\begin{equation}
    D_{\max}(u) := \max_{s \geq 0} D_{RH}(u, \S_u^1(s), \sigma_u^1(s)).
\end{equation}
The authors showed that this function is well defined when $C$ is sufficiently large and $s_0$ sufficiently small, with every state having exactly one maxima when the system~\eqref{system} satisfies Assumption \ref{assum}, where we take this maximum to occur at $s = s^*(u)$ and define $u^+(u) := \S_u^1(s^*(u))$. 
Further results concerning the maximal shock are reproduced in Section \ref{sec:app-u-plus}.
From here we can get global dissipation bounds on the entire shock curve $\S_u^1$ by just bounding the dissipation at the maximal shocks $(u,u^+(u), \sigma_\pm)$. 
The structure of this proof is divided into parts:
\begin{description}
    \item[Part 1] We establish the bound~\eqref{diss:shock} at every maximal shock $(u,u^+(u), \sigma_\pm)$. 
    This is exactly the content of the following proposition. 
\end{description}
\begin{proposition} \label{prop:dmax} 
    There exists a $K > 0$ such that for $C$ sufficiently large, $s_0$ sufficiently small, for all $u_- \in \Pi_{C,s_0}^*$,
    \begin{equation}
        D_{\max}(u) \leq - Ks_0(|u - u_L|^2 + |u^+(u) - u_R|^2). \label{eq:Dmax-bound}
    \end{equation}
\end{proposition}
\begin{description}
    \item[Part 2] Equipped with the bound~\eqref{diss:shock} in the continuous case $s = 0$ and maximal case $s = s^*(u)$ we now estimate the dissipation along other points in $\S_u^1$.
    To do this we make vital use of the convexity of $s\mapsto \eta(u|\S_u^1(s))$ for $s$ sufficiently small and the monotonicity of this map from Assumption \ref{assum} \ref{assum:rel-ent-strengthens} for all $s$.
\end{description}
Part 1 follows by performing the same partitioning as \cite{MR4667839} on the set $\Pi_{C,s_0}$, where they divided $\Pi_{C,s_0}$ into two different length scales: $Q_{C,s_0}$ of order $C^{-1}$ of states sufficiently far from the shock $(u_L, u_R,\sigma_{LR})$ and states within a set $R_{C,s_0}$ of length scale $s_0$ near our fixed shock. 
See Figures \ref{fig:Q-image} and \ref{fig:R-image} for a pictorial description of the sets $Q_{C,s_0}$ and $R_{C,s_0}$, respectively. 
These sets are precisely defined in equations~\eqref{eq:Qdef} and~\eqref{eq:Rdef}, respectively.
In this paper, we redo the analysis at the large length scale, where we are now able to use our bound from Proposition \ref{prop:cont} to obtain stronger estimates. 
The steps for this are as follows:
\begin{figure}[t] 
    \centering 
    \includegraphics[width=.7\linewidth]{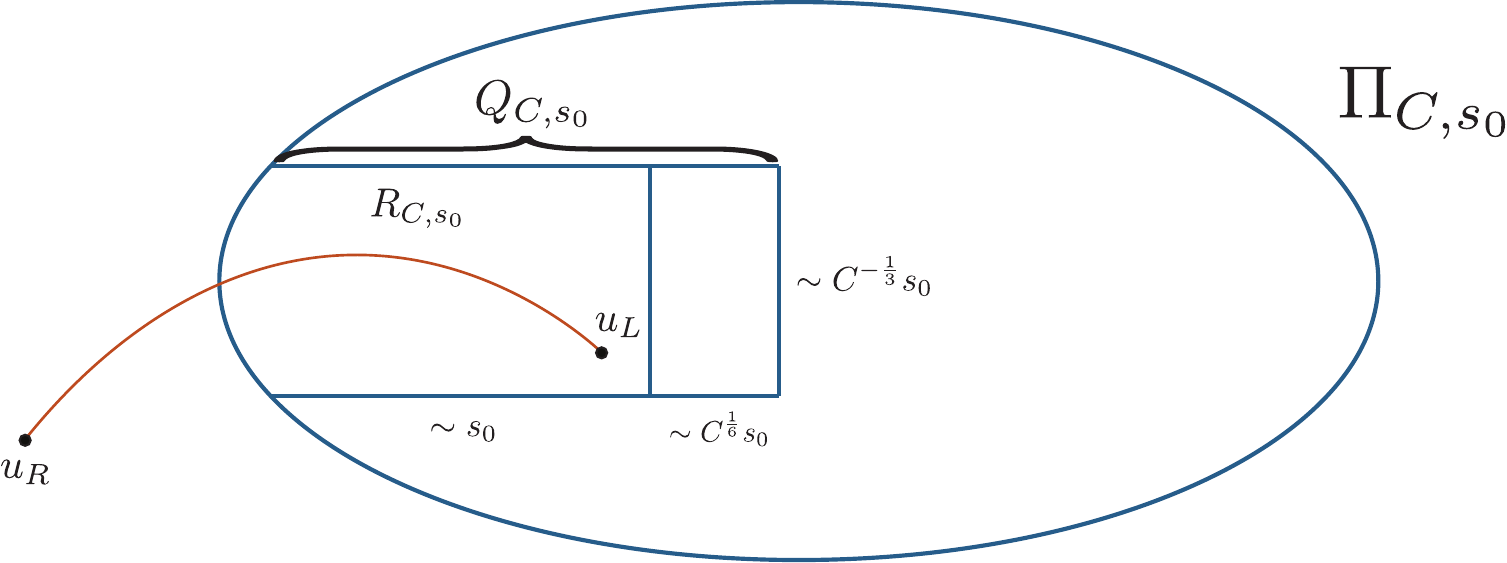}
    \caption{This figure depicts an idealized representation of the set $Q_{C,s_0}$. }
    \label{fig:Q-image}
\end{figure}
\begin{description}
    \item[Step 1.1] We prove the bound~\eqref{diss:shock} for maximal shocks $(u,u^+(u), \sigma)$ when $u^+(u) \not\in \Pi_{C,s_0}$ and $u \notin Q_{C,s_0}$. For a precise definition of $Q_{C,s_0}$ see equation~\eqref{eq:Qdef}.
        The first of these two assumptions allows us to conclude that $u$ is close to $\partial \Pi_{C,s_0}$ while the second places us far from $u_L,\, u_R$, giving an upper bound on the continuous dissipation from Proposition \ref{prop:cont}.

    \item[Step 1.2] We then show the bound for states satisfying $u^+(u) \in \Pi_{C,s_0}$. This proves the bound~\eqref{diss:shock} all for maximal shocks $(u,u^+(u), \sigma)$ with $u\notin R_{C,s_0}$ (see equation~\eqref{eq:Rdef} for a precise definition).
\end{description}
\begin{figure}[b] 
    \centering
    \includegraphics[width=.7\linewidth]{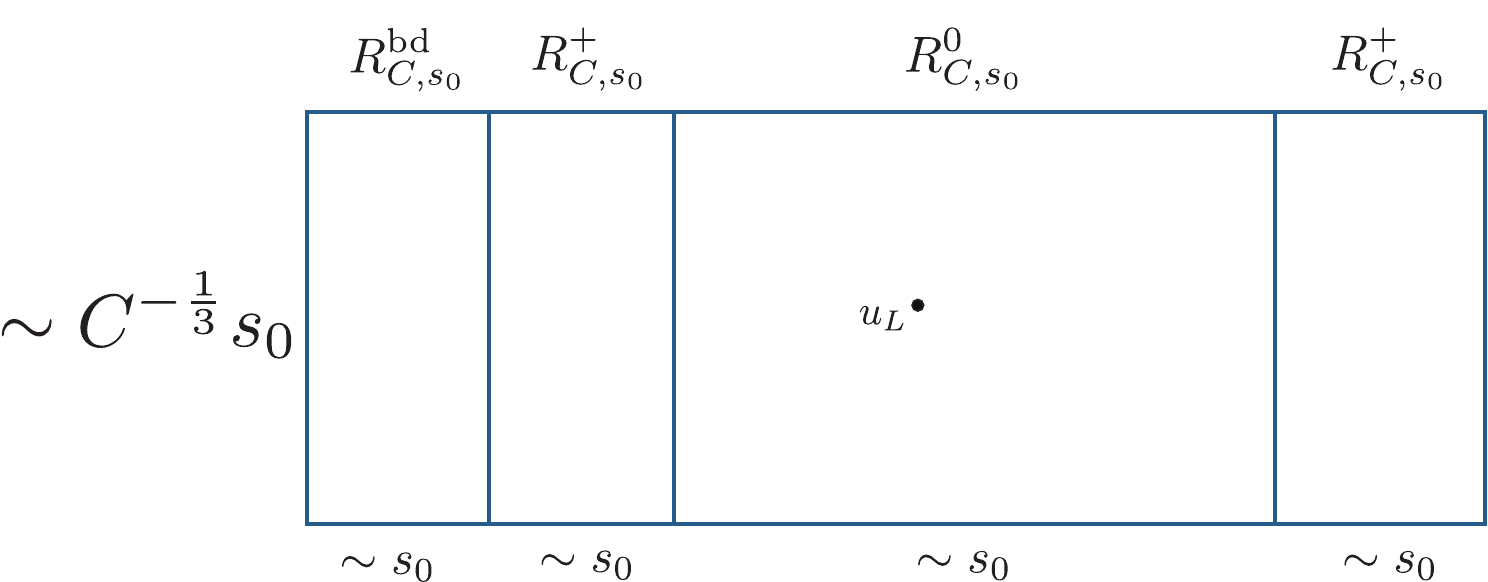}
    \caption{This figure depicts an idealized representation of the set $R_{C,s_0}$. We note that $R^+_{C,s_0}$ has two disconnected components.}
    \label{fig:R-image}
\end{figure}
\begin{description}
    \item[Step 1.3] In the small length scale $R_{C,s_0}$, the results of \cite{MR4667839} suffice to prove the bound~\eqref{diss:shock} for maximal shocks when $u$ is sufficiently close (of order $s_0$) to $u_L$ and close (of order $s_0$) to the boundary $(\partial \Pi_{C,s_0})\cap R_{C,s_0}$. 
        From here there remains an intermediate region $R^+_{C,s_0}$ within $R_{C,s_0}$ where we still have yet to prove~\eqref{diss:shock}.
        To do this, we use the fact that $D_{\max}$ has no critical values in this region and perform a scaling argument, similar to the final step of the proof of Proposition \ref{prop:cont}. 
\end{description}

Equipped with Proposition \ref{prop:dmax}, in Part 2 we use it to prove \Cref{prop:shock} in the following steps:
\begin{description}
    \item[Step 2.1]  First, extend the bound at the maximal shock $(u,u^+(u), \sigma_\pm)$ to small shocks along the shock curve $\S_u^1$ for states $u$ satisfying $s^*(u) \leq  K_1 s_0$ for some $K_1 \in (0,1)$. 
        This assumption should be understood as considering states where the maximal shock is small. 
        In doing this, we are considering a family of states sufficiently far from $u_L$ to where we will have negativity from the bound in Proposition \ref{prop:dmax}.
    \item[Step 2.2] Next, for states with $s^*(u) \geq K_1 s_0$ the bound follows after showing $$D_{RH}(u,\S_u^1(s)) \lesssim D_{\max}(u) - s^*(u) |s-s^*(u)|^2$$
        when $s$ is sufficiently small.
    \item[Step 2.3] Finally, we remove the smallness constraint on the argument $s$ imposed in Steps 2.1, 2.2. 
        This is done by exploiting the monotonicity of $s\mapsto \eta(u|\S_u^1(s))$ thanks to Assumption \ref{assum} \ref{assum:rel-ent-strengthens} along with the bound established in Steps 2.1 and 2.2. 
\end{description}

Before we begin with Step 1, we recall the following identity. 
\begin{lemma}[Lemma 6.1 and Proposition 6.2 of \cite{MR4667839}]
    For any $u \in \Pi_{C,s_0}^*$ and $s > 0$ we have
    \begin{equation} D_{\RH}(u, \S_u^1(s)) = D_{\cont}(u) + \int_0^s \dot \sigma(t)(\tilde \eta(u) + \eta(u|\S_u^1(t))) \,dt. \label{eq:RH-CONT}\end{equation}
\end{lemma}
This identity can be deduced via the quantified entropy dissipation equality~\eqref{eq:entropy-quant} with $v = u_R$ and the fundamental theorem of calculus.


\subsubsection{Part 1: The maximal shock} \label{sec:max}

As in \cite{MR4667839} this case will be handled by partitioning $\Pi_{C,s_0}^*$ into the smaller sets $Q_{C,s_0}^c, Q_{C,s_0}$, and $R_{C,s_0}$ and getting our bound in each piece separately. 
The bounds used in each regime are similar to what was done in \cite{MR4667839}, and the additional negativity in our bound comes almost entirely due the improvement from Proposition \ref{prop:cont}.
We note that the states $u \in \Pi^* \setminus \Pi$ are handled trivially, as $u^+(u) = u$ in this region giving $D_{\max}(u) = D_{\cont}(u)$ which satisfies our estimate from Proposition \ref{prop:cont}

\bigskip
\paragraph{\bf Step 1.1: States far from $u_L$ with $u^+$ outside $\Pi$}
To begin, we define the set $Q_{C,s_0}$. We will first show states in $\Pi \setminus Q_{C,s_0}$ satisfy~\eqref{eq:Dmax-bound}.
\begin{equation}\label{eq:Qdef}
    Q_{C,s_0} := \left\{p = u^* - \sum_{i = 1}^n b_i r_i(u^*) \ \biggr| \ p\in \Pi_{C,s_0},\ |b_1| \le C^{1/6}s_0, \text{ and }\left(\sum_{i = 2}^n |b_i|^2\right)^{1/2} \le C^{-1/3}s_0\right\}. 
\end{equation}

To begin, we recall the following bound on $D_\cont$ in $\Pi_{C,s_0} \setminus Q_{C,s_0}$.
\begin{lemma}[\cite{MR4667839} Proposition 6.2] \label{lemma:gkv-improved-cont}
    For $C$ sufficiently large and $s_0$ sufficiently small at any state $u \in \Pi_{C,s_0} \setminus Q_{C,s_0}$
    we have the continuous dissipation bound 
    \begin{equation}
        D_\cont(u) \lesssim -C^{1/3}s_0^3
    \end{equation}
\end{lemma}
The proof of this statement follows from the intermediate bound in the expansion of Lemma \ref{lemma:D-cont-exp}.
We note that this intermediate bound is vital, as it gives additional dissipation in the directions transverse to our shock direction $r_1$ allowing $Q_{C,s_0}$ to be narrower than tall. 

\begin{proposition}
    \label{prop:Q-uplus-outside}
    For any $C$ sufficiently large, $s_0$ sufficiently small (depending on $C$), and for any $u$ such that
    $$u \in \Pi_{C,s_0} \backslash Q_{C,s_0},$$
    also satisfies $u^+ = \S^1_u(s^*(u)) \notin \Pi_{C,s_0}$, the maximal entropy dissipation satisfies,
    \begin{equation}
    D_{\max}(u) \lesssim -s_0( |u - u_L|^2 + | u^+(u) - u_R|^2).
    \end{equation}
\end{proposition}
\begin{proof}
    From Proposition \ref{prop:cont}, Lemmas \ref{lemma:gkv-improved-cont}, \ref{lem:gkv_dist}, \ref{gkv:lem_s-star}, and $\dot \sigma \lesssim -1$ due to both Lemma \ref{lem_hugoniot} and Assumption \ref{assum} \ref{assum:rel-ent-strengthens} we get the following estimate on equation~\eqref{eq:RH-CONT}
    \begin{align*}
        D_{\max}(u) =& D_{\cont}(u) + \int_0^{s^*} \dot \sigma(t)(\tilde \eta(u) + \eta(u|\S_u^1(t))) \,dt \\
        \lesssim& \frac{1}{2}D_{\cont}(u) + \frac{1}{2}D_{\cont}(u) - \int_0^{s^*} \tilde \eta(u) + \eta(u|\S_u^1(t)) \,dt \\
        \lesssim& -s_0(|u-u_L|^2 + |u-u_R|^2) - C^{1/3}s_0^3 -(s^*)^3 + s_0s^*d(u, \partial \Pi) 
    \end{align*} 
    By our assumption that $u^+(u) \notin \Pi$ we have $d(u,\partial \Pi) \leq s^*.$ Applying Lemma \ref{lemma:leger-square}, Lemma \ref{lem:gkv_dist}, and Lemma \ref{gkv:lem_s-star} we find $ s^* \lesssim s_0$ hence taking any $u$ satisfying $ C^{1/3}s_0^2 \leq |u-u_L|^2 + |u-u_R|^2 $, we have 
    \begin{align*}
        D_{\max}(u) \lesssim& -s_0(|u-u_L|^2 + |u-u_R|^2) - C^{1/3}s_0^3 -(s^*)^3 + s_0s^*d(u, \partial \Pi) \\
        \leq& -s_0(|u-u_L|^2 + |u-u_R|^2) - C^{1/3}s_0^3 + s_0^3 \\
        \leq& -s_0(|u-u_L|^2 + |u-u_R|^2 + |u-u^+(u)|^2) + s_0|u-u^+(u)|^2 - C^{1/3}s_0^3 \\
        \lesssim& -s_0(|u-u_L|^2 + |u^+(u)-u_R|^2) + s_0(s^*)^2 - C^{1/3}s_0^3 \\
        \lesssim& -s_0(|u-u_L|^2 + |u^+(u)-u_R|^2)
    \end{align*}
    for $C$ sufficiently large and $s_0$ sufficiently small. 
\end{proof}

\bigskip
\paragraph{\bf Step 1.2: States with $u^+ \in \Pi$}

This leaves the states $u \in \Pi_{C,s_0}\setminus Q_{C,s_0}$ such that $u^+(u) \in \Pi_{C,s_0}$. 
We are able to show the bound for such $u$, as long as they do not get too close to $u_L$ (up to scale $\sim s_0$). 
In the following proposition we prove the bound for the all states $u$ where $u^+(u) \in \Pi_{C,s_0}$.
Furthermore, as was shown in \cite{MR4667839} this property is satisfied for states all states sufficiently far in $\Pi_{C,s_0}$, 
allowing us to prove Proposition \ref{prop:dmax} for states outside of a set $R_{C,s_0}(K_h)$, where
\begin{equation} \label{eq:Rdef}
    R_{C,s_0}(K_h) := \left\{ u = u^* - \sum_{i=1}^n b_i r_i(u^*) \ \biggr| \ u\in \Pi_{C,s_0}, \ |b_1| \leq K_hs_0, \ \left(\sum_{i=2}^n |b_i|^2\right)^{1/2} \leq C^{-1/3}s_0 \right\}.
\end{equation}
\begin{proposition} \label{prop:Dmax-kh}
    There exists a universal constant $K_h > 0$, depending only on the system and $d$, such that for any $C$ sufficiently large and $s_0$ sufficiently small any state $u \in \Pi \setminus R_{C,s_0}(K_h)$ satisfies
    \begin{equation}
        D_{\max}(u) \lesssim -s_0(|u - u_L|^2 + |u^+(u) - u_R|^2).
    \end{equation}
\end{proposition}
This proposition follows immediately from the following lemmas
\begin{lemma}
    \label{lemma:Dmax-u+-inside}
    Suppose the maximal shock $(u,u^+(u),\sigma_\pm)$ satisfies $u, u^+(u) \in \Pi_{C,s_0}$. 
    Then, 
    \begin{equation}
        D_{\max}(u) \leq D_{\cont}(u^+(u)) + \int_0^{s^*} \dot \sigma(t) \eta(u|\S_u^1(t))\,dt
    \end{equation}
    Moreover, for $C$ sufficiently large and $s_0$ sufficiently small we have 
    \begin{equation}
        D_{\max}(u) \lesssim -s_0(|u-u_L|^2 + |u^+(u) - u_R|^2).
    \end{equation}
\end{lemma}
Furthermore, we find that states $u$ far enough inside of $\Pi_{C,s_0}$ always have $u^+(u) \in \Pi_{C,s_0}$.
\begin{lemma}[Lemma 6.10 from \cite{MR4667839}]
    There is a universal constant $K^*$ such that for all $C$ sufficiently large and $s_0$ sufficiently small, for $u \in \Pi_{C,s_0}$ satisfying
    $$ - \tilde \eta(u) \geq K^* s_0^2$$ 
    we have $u^+(u) \in \Pi_{C,s_0}$ and, consequently, 
    $$ D_{\max}(u) \lesssim -s_0(|u-u_L|^2 + |u^+(u) - u_R|^2) $$
    for such $u$.  
\end{lemma}

\begin{proof}[Proof of Lemma \ref{lemma:Dmax-u+-inside}]
    We begin by improving the bound of Lemma 6.9 of \cite{MR4667839} by using the quantified dissipation formula~\eqref{eq:entropy-quant} with $v = u_L$ and $s = s^*(u)$,
    $$ q(u;u_L) - \sigma_{\pm} \eta(u|u_L) = q(u^+(u);u_L) - \sigma_{\pm} \eta(u^+(u)|u_L) - \int_0^{s^*} \dot \sigma(t) \eta(u|\S_u^1(t))\,dt.$$
    Substituting this into $D_{\max}$ we arrive at the identity
    \begin{align*}
        D_{\max}(u) =& [q(u^+(u); u_R) - \sigma_\pm \eta(u^+(u) | u_R)] - (1 + Cs_0) [q(u; u_L) - \sigma_\pm \eta(u | u_L)] \\
        =&[q(u^+(u); u_R) - \sigma_\pm \eta(u^+(u) | u_R)] - (1 + Cs_0) [q(u^+(u); u_L) - \sigma_\pm \eta(u^+(u) | u_L)] \\
        &+ (1+Cs_0)\int_0^{s^*} \dot \sigma(t) \eta(u|\S_u^1(t))\,dt,
    \end{align*}
    where now the relative quantities are all evaluated at $u^+$. 
    Adding and subtracting $\lambda_1(u^+(u))\eta(u^+(u)|u_R)$ and $(1+Cs_0)\lambda_1(u^+(u))\eta(u^+(u)|u_L)$ we find
    \begin{equation} \label{eq:Dmax-Dcont-uplus}
    \begin{aligned}
        D_{\max}(u) &= D_{\cont}(u^+(u)) + (\sigma_\pm - \lambda_1(u^+(u))) \tilde \eta(u) + (1+Cs_0)\int_0^{s^*} \dot \sigma(t) \eta(u|\S_u^1(t))\,dt \\
        &\leq  D_{\cont}(u^+(u)) + \int_0^{s^*} \dot \sigma(t) \eta(u|\S_u^1(t))\,dt,
    \end{aligned}\end{equation}
    where we used that $\tilde \eta(u^+(u)) \geq 0$ by assumption, $\sigma >  \lambda_1(u^+(u))$ due to Assumption \ref{assum} \ref{assum:1-shock-lax-cond-greater}, and the integral term being non-positive.

    To use this bound, we first note by Lemmas \ref{lemma:leger-square} and \ref{lem:gkv_dist} that $s_0^2\sim \tilde \eta(u_R) \lesssim s_0d(u,\Pi_{c,s_0})$, hence there exists a universal constant $K_1 > 0$ such that
    $$2 K_1^2 s_0^2 < d(u_R, \Pi_{C,s_0})^2.$$
    In particular we note that this constant is independent of $C$, $s_0$.  
    For states $u$ with $s^*(u) < K_1s_0$ we elect to use solely our bound on $D_{\cont}(u^+(u))$. 
    We have by~\eqref{eq:Dmax-Dcont-uplus}, Proposition \ref{prop:cont}, and Lemma \ref{lem:gkv_dist} the bound
    \begin{align*}
        D_{\max}(u) \leq& D_{\cont}(u^+(u)) \\
        \lesssim& -s_0( |u^+(u) - u_L|^2 + |u^+(u) - u_R|^2) \\
        =& -s_0\left( \frac{1}{2}|u - u_L|^2 + |u^+(u) - u_R|^2\right) + \frac{s_0}{2}( |u-u_L|^2 - 2|u^+(u) - u_L|^2) \\
        \leq& -s_0\left( \frac{1}{2}|u - u_L|^2 + |u^+(u) - u_R|^2\right) + s_0(s^*)^2 \\
        \leq& -\frac{s_0}{2}( |u - u_L|^2 + |u^+(u) - u_R|^2) + \frac{1}{2}s_0( 2(s^*)^2 - |u^+(u) - u_R|^2) \\
        \leq& -\frac{s_0}{2}( |u - u_L|^2 + |u^+(u) - u_R|^2) +  \frac{1}{2}s_0( 2(s^*)^2 - d(u_R, \Pi_{C,s_0})^2 ) \\
        \leq& -\frac{s_0}{2}( |u - u_L|^2 + |u^+(u) - u_R|^2) +  \frac{1}{2}s_0( 2K_1^2s_0^2 - d(u_R, \Pi_{C,s_0})^2 )\\
        \leq& -\frac{s_0}{2}( |u - u_L|^2 + |u^+(u) - u_R|^2)
    \end{align*}
    where the last line follows by definition of $K_1$ and carefully tracking the constants. 

    For states $s^*(u) \geq K_1s_0$ our upper bound on $D_{\cont}(u^+(u))$ is insufficient. 
    However, for the integral term of equation~\eqref{eq:Dmax-Dcont-uplus} there exists a universal constant $K_2$ giving the upper bound 
    $$\int_0^{s^*} \dot \sigma(t) \eta(u|\S_u^1(t))\,dt \leq -2K_2\int_0^{s^*} (s^*)^2\,dt \leq -K_2(s^*)^3,$$
    due to $\dot \sigma \lesssim -1$ and Lemma \ref{lemma:leger-square}. 
    Further, we note that $K_1 \leq 1$, as $d(u_R,\Pi_{C,s_0}) \leq d(u_R,u_L) = s_0$. For $s^* \geq K_1 s_0$ we find
    \begin{align*}
        D_{\max}(u) \leq& D_{\cont}(u^+(u)) - K_2(s^*)^3 \\
        \leq& \min\left(1, \frac{K_1K_2}{K}\right)D_{\cont}(u^+(u)) - K_1K_2s_0(s^*)^2 \\
        \leq& -\min(K,K_1K_2)s_0( |u - u_L|^2 + |u^+(u) - u_R|^2) + \frac{1}{2}K_1K_2s_0(s^*)^2 - K_1K_2s_0(s^*)^2 \\
        \lesssim& -s_0( |u - u_L|^2 + |u^+(u) - u_R|^2),
    \end{align*}
    where $K$ is the universal constant granted by Proposition \ref{prop:cont}.
\end{proof}
The remainder of Proposition \ref{prop:Dmax-kh} follows identically to the corresponding proof in \cite[p.~584]{MR4667839}.
This establishes Proposition \ref{prop:dmax} outside of the small length scale $R_{C,s_0}(K_h)$, which is defined by equation~\eqref{eq:Rdef}.

\bigskip
\paragraph{\bf Step 1.3: Extending into the small scale $R_{C,s_0}$}

To begin extending our bound into the region $R_{C,s_0}(K_h)$ we decompose it into $R_{C,s_0}(K_h) = R_{C,s_0}^{bd}  \cup R_{C,s_0}^0 \cup R_{C,s_0}^+$, where these sets will be defined in forthcoming lemmas.
We first handle $u\in R_{C,s_0}^{bd}$, which are the states near $\partial \Pi_{C,s_0} \cap R_{C,s_0}$,
\begin{proposition}[Proposition 6.11 from \cite{MR4667839}]
    There is a universal constant $K_{bd}$ such that for all $C$ sufficiently large and $s_0$ sufficiently small, defining
    \begin{equation}
        R_{C,s_0}^{bd} := \{ u\in R_{C,s_0} | |u-u^*| \leq K_{bd}s_0 \}
    \end{equation}
    we have for all $u\in R_{C,s_0}^{bd}$ that $$D_{\max}(u) \lesssim -s_0(|u-u_L|^2 + |u^+(u) - u_R|^2).$$
\end{proposition}
\begin{proof}
    By equation~\eqref{eq:RH-CONT} we have 
    \begin{align*}
        D_{\max}(u) =& D_{\cont}(u) +  \int_0^{s^*} \dot \sigma(t)(\tilde \eta(u) + \eta(u|\S_u^1(t))) \,dt \\
        \lesssim& \frac{1}{2}D_{\cont}(u)-\frac{s_0}{2}(|u-u_L|^2 + |u-u_R|^2 ) + s^*s_0 d(u,\partial \Pi_{C,s_0} ) \\
        \lesssim& D_{\cont}(u)-s_0(|u-u_L|^2 + |u^+(u)-u_R|^2) + s_0(s^*)^2+s^*s_0d(u,\partial \Pi), \\
        \shortintertext{supposing $u$ satisfies $d(u,\partial \Pi_{C,s_0}) < K_{bd} s_0$ we find }
        \lesssim& D_{\cont}(u)-s_0(|u-u_L|^2 + |u^+(u)-u_R|^2) + K_{bd}s_0^3 + K_{bd}^{3/2} s_0^3
    \end{align*}
    where we have applied Lemma \ref{lem:gkv_dist}, Lemma \ref{lemma:leger-square}, and Lemma \ref{gkv:lem_s-star}. 
    Recalling $D_{\cont}(u) \lesssim -s_0^3$ by Proposition \ref{prop:gkv-cont} we find for $K_{bd}$ sufficiently small we have the bound~\eqref{eq:Dmax-bound}.
\end{proof}

In the case of states sufficiently close to $u_L$ we have
\begin{proposition}[Proposition 6.12 of \cite{MR4667839}]
    There is a universal constant $K_0$ such that for any $C$ sufficiently large and $s_0$ sufficiently small (depending on $C$), defining
    \begin{equation} \label{eq:R0-def}
        R_{C,s_0}^0 := \{u \in R_{C,s_0} : | u - u_L| \leq K_0 s_0\}
    \end{equation}
    for any $u \in R_{C,s_0}^0$ we have the bound 
    \begin{equation} \label{eq:R0-bound}
        D_{\max}(u) \lesssim -s_0|u-u_L|^2 \lesssim -s_0( | u - u_L |^2 + |u^+(u) - u_R|^2).
    \end{equation}
\end{proposition}
The proof of the above statement follows almost identically to that of Proposition 6.12 from \cite{MR4667839} and the Lipshitz bound $|\nabla u^+(u)| \lesssim 1$ from Lemma \ref{gkv_uplus-ident-2}.

We have shown the bound~\eqref{diss:shock} on both $R_{C,s_0}^{bd}$ and $R_{C,s_0}^0$, which leaves the region $$R^+_{C,s_0} := R_{C,s_0} \setminus [R_{C,s_0}^{bd} \cup R^0_{C,s_0}].$$
By continuity we have established the bound~\eqref{eq:Dmax-bound} on $\partial R_{C,s_0}^+$. We can further conclude that 
$$\underset{\mathclap{u \in R_{C,s_0}^+}}{\text{argmax}}\, D_{\max}(u) \in \partial R_{C,s_0}^+$$ 
by the following lemma.
\begin{lemma}[Lemmas 6.19, 6.20 from \cite{MR4667839}] \label{lemma:dmax-no-extrema}
    For any $C$ sufficiently large and $s_0$ sufficiently small, the function $D_{\max}$ does not have a critical point within the set $ R_{C,s_0}^+$. 
\end{lemma}

Using this and that the bound~\eqref{eq:Dmax-bound} holds on $\partial R^+_{C,s_0}$, we can now extend it into $R_{C,s_0}^+$ by a scaling argument. 
\begin{lemma}\label{lemma:dmax-R+}
    For any $C$ sufficiently large and $s_0$ sufficiently small, we have the bound 
    \begin{equation}
        D_{\max}(u) \lesssim -s_0(|u-u_L|^2 + |u^+(u) - u_R|^2)
    \end{equation}
    for all $u \in R_{C,s_0}^+$. 
\end{lemma}
\begin{proof}
    Let 
    $$ \bar u = \underset{u \in R_{C,s_0}^+}{\text{argmax}}\, D_{\max}(u).$$
    By Lemma \ref{lemma:dmax-no-extrema} we know $\overline u \in \partial R_{C,s_0}^+$ and, due to continuity and the preceding propositions, $\overline u$ satisfies the bound~\eqref{eq:Dmax-bound}.
    Let 
    $$ L = \underset{\substack{u\in R_{C,s_0} \\ \setminus [R_{C,s_0}^{bd} \cup R^0_{C,s_0}]}}{\sup} |u-u_L|^2 + |u^+(u) - u_R|^2$$
    be the longest squared distance of a maximal shock $(u,u^+(u))$ to the shock states $(u_L, u_R)$ within $R^+_{C,s_0}$. 
    On this region we find the bound 
    $$D_{\max}(u) \leq D_{\max}(\overline u) \leq \frac{1}{L} D_{\max}(\overline u) (|u-u_L|^2 + |u^+(u) - u_R|^2).$$
    What remains is to show $\frac{1}{L} D_{\max}(\overline u) \lesssim -s_0$. 
    Since $\overline u$ lies on the boundary we have the quantitative estimate 
    $$ D_{\max}(\overline u) \lesssim -s_0 (|\overline u-u_L|^2 + |u^+(\overline u) - u_R|^2).$$
    We note that since $\overline u \notin R_{C,s_0}^0$ we have $|\overline u - u_L|^2 \gtrsim s_0^2$ by~\eqref{eq:R0-def} giving $D_{\max}(\overline{u}) \lesssim -s_0^3$.
    From the definition~\eqref{eq:Rdef} we know $\text{diam}(R_{C,s_0}) \lesssim C^{1/6}s_0$
    so for all $u$ in $R_{C,s_0}$ we have 
    $$ |u-u_L|^2 + |u^+(u) - u_R|^2 \lesssim C^{1/3}s_0^2 $$
    by the size of $R_{C,s_0}$ and the Lipschitz estimate on $u^+$ by Lemma \ref{gkv_uplus-ident-2}
    which shows $L \lesssim s_0^2$. 
    Equipped with these bounds on $L$ and $D_{\max}(\overline u)$ we now have
    \begin{equation}\frac{1}{L} D_{\max}(\overline u) \lesssim -\frac{s_0^3}{s_0^2} = -s_0, \end{equation}
    as desired. 
    The suffices to establish the quantitative estimate on every shock $(u,u^+(u), \sigma_{\pm})$ where $u \in \Pi^*$. 
\end{proof}


\subsubsection{Part 2: Extending the rest of the shock curve} \label{sec:shock-curve-ext}
In this section we use the fact that the map $t \mapsto \eta(u| \S_u^1(t))$ is convex for $t$ sufficiently small to show the bound~\eqref{diss:shock} for states $(u,\S_u^1(s))$ with $u \in \Pi_{C,s_0}^*$ and $s$ sufficiently small. 
Before we begin, we some establish bounds independent of our parameters $C$ and $s_0$.
\begin{lemma} \label{lemma:Drh-Dmax-bound}
    There exists $\overline t, L > 0$ such that
    we have for all $t < \overline t$ and $u\in \Pi_{C,s_0}^*$ that 
    \begin{equation} \frac{d^2}{dt^2} \eta(u|\S_u^1(t)) \geq \lambda := \frac{1}{2}\min_{u\in \Pi_{C,s_0}^*} \nabla^2 \eta(u)r_1(u)\cdot r_1(u) \label{eq:uni-conv-rel-ent}\end{equation}
    and giving a Lipschitz estimate on the shock curve for $s,t \in [0,\overline t], $
    \begin{equation}\label{eq:shock-lip}
        |\S_u^1(s) - \S_u^1(t)| \leq L|s-t|
    \end{equation}
\end{lemma}

\begin{proof}
    Using the expansion of the shock curve from Lemma \ref{lem_hugoniot} we have
    \begin{align}
        \frac{d}{dt}\eta(u|\S_u^1(t)) =& -\nabla^2 \eta(\S_u^1(t))  (u - \S_u^1(t)) \cdot (\partial_t \S_u^1(t)) \nonumber \\
        =& \nabla^2 \eta(\S_u^1(t))(r_1(u)t)\cdot r_1(u) + \bigO(t^2), \label{eq:rel-shock-first-der} \\
        \frac{d^2}{dt^2} \eta(u|\S_u^1(t)) =& \nabla^2 \eta(\S_u^1(t))(\partial_t \S_u^1(t))\cdot (\partial_t \S_u^1(t)) \nonumber\\
        &- \nabla^2\eta(\S_u^1(t)) (u- \S_u^1(t)) \cdot \partial_t^2 (\S_u^1(t)) \nonumber\\
        &- \nabla^3 \eta(\S_u^1(t)):(u-\S_u^1(t))\otimes (\partial_t \S_u^1(t))^{\otimes 2}, \nonumber
    \end{align}
    and we note that by continuity there exists $\overline t > 0$ such that for all $u\in \Pi_{C,s_0}^*$ the second derivative satisfies equation~\eqref{eq:uni-conv-rel-ent}
    for all $t < \overline t$. 

    We recall the map $(u,s) \mapsto \S_u^1(s)$ is $C^1$ from Lemma \ref{lem_hugoniot} and $\overline \Pi_{C,s_0}^* \times [0,\overline t]$ is compact, which suffices to establish equation~\eqref{eq:shock-lip}.

\end{proof}
\noindent
\textbf{Note:} Henceforth we will denote generic states on the shock curve $\S_u^1$ as $u_+$ to avoid confusion with the function $u^+(u)$. 

\bigskip
\paragraph{\bf Step 2.1: Small maximal shocks}

\begin{lemma} \label{lemma:Drh-up-to-tbar}
    There exists a universal constant $K_1$ such that, for $C$ sufficiently large and $s_0$ sufficiently small,
        when $s^*(u) \leq K_1s_0$, we find for all $u \in\Pi_{C,s_0}^*$ and $0 < s \leq \overline t $ that 
        \begin{equation}\label{eq:Drh-up-to-tbar}
            D_{\RH}(u,u_+) \lesssim -s_0(|u-u_L|^2 + |u_+-u_R|^2) 
        \end{equation}
        where $u_+ = \S_u^1(s)$.
\end{lemma}
\begin{proof}
    Let $K_1$ be a constant $0 < K_1 < 1$, to be fixed later. 
    We will prove Lemma \ref{lemma:Drh-up-to-tbar} in the cases $K_1s_0 < s - s^*$ and $K_1s_0 \geq s - s^*$. 
    The primary idea in this proof is that $D_{RH}(u,u_+)$ is close to $D_{\max}(u)$ in the first case, and since $s^* \leq K_1 s_0$ we can get a strictly negative upper bound on $D_{\max}(u)$. 
    
    First, we note we can take equation~\eqref{eq:RH-CONT} and split the interval of integration into $[0,s^*]$ and $[s^*, s]$ to find 
    \begin{equation} \label{eq:Drh-Dmax}
        D_{\RH}(u, \S_u^1(s)) = D_{\max}(u) + \int_{s^*}^s \dot \sigma(t)(\tilde \eta(u) + \eta(u|\S_u^1(t)))\,dt.
    \end{equation}
    We next bound the integral term by noting that $\dot \sigma(s) \lesssim -1$, the expansion of the shock curve, equation~\eqref{eq:rel-shock-first-der}, and the second derivative bound~\eqref{eq:uni-conv-rel-ent} to establish 
    \begin{align*}
        -\int_{s^*}^s \tilde \eta(u) + \eta(u|\S_u^1(t))\,dt \leq& -\int_{s^*}^s\partial_t [\eta(u|\S_u^1(t))]_{t = s^*} (t-s^*) + \frac{\lambda}{2} (t-s^*)^2\,dt\\
        \lesssim& -s^*|s-s^*|^2 - (s-s^*)^3,
    \end{align*}
    when $C$ sufficiently large and $s_0$ sufficiently small.
    When $K_1s_0 \leq s - s^*$ and $s < \overline t$ we immediately retrieve the desired upper bound thanks to the above bound, triangle inequality, equation~\eqref{eq:Drh-Dmax}, and Lipschitz bound~\eqref{eq:shock-lip},
    \begin{align*}
        D_{\RH}(u,\S_u^1(s)) \lesssim& D_{\max}(u) + \int_{s^*}^s \dot \sigma(t) ( \tilde \eta(u) + \eta(u| \S_u^1(t)))\,dt \\
        \lesssim& D_{\max}(u) - (s-s^*)(s-s^*)^2 \\
        \leq& D_{\max}(u) - K_1s_0(s-s^*)^2 \\
        \lesssim& -s_0(|u-u_L|^2 + |u^+(u) - u_R|^2) - s_0|u_+ - u^+(u)|^2 \\
        \lesssim& -s_0(|u-u_L|^2 + |u_+ - u_R|^2).
    \end{align*}
    We now assume $u_+ = \S_u^1(s)$ is close to $u$ of order $s_0$; specifically, we suppose $s$ satisfies $K_1s_0 \geq s - s^*$.
    For such a state $|s-s^*| \leq K_1s_0$, as we have already supposed $s^* \leq K_1s_0$. 
    Using this, the Lipschitz bound on the shock curve~\eqref{eq:shock-lip}, the relation~\eqref{eq:Drh-Dmax}, and Proposition \ref{prop:dmax}
    \begin{align}
        D_{\RH}(u,\S_u^1(s)) \leq& D_{\max}(u) \nonumber\\
        \lesssim& -s_0(|u-u_L|^2 + |u^+(u) - u_R|^2) \nonumber\\
        =& -s_0\left(|u-u_L|^2 + \frac{1}{2}|u_+ - u_R|^2\right) + \frac{s_0}{2}(|u_+ - u_R|^2 - 2|u^+(u) - u_R|^2) \nonumber\\
        \leq& -s_0\left(|u-u_L|^2 + \frac{1}{2}|u_+ - u_R|^2\right) + s_0|u_+ - u^+(u)|^2 \nonumber\\
        \leq& -s_0\left(|u-u_L|^2 + \frac{1}{2}|u_+ - u_R|^2\right) + s_0L^2|s-s^*|^2 \nonumber\\
        \leq& -\frac{s_0}{2}(|u-u_L|^2 + |u_+ - u_R|^2) + s_0\left(K_1^2L^2 s_0^2 -\frac{1}{4}|u-u_L|^2 \right)  \nonumber \\
        \leq & -\frac{s_0}{2} (|u-u_L|^2 + |u^+ - u_R|^2) \label{eq:close-close}
    \end{align}
    where the inequality~\eqref{eq:close-close} holds for states $u$ satisfying $|u-u_L| > 2 K_1L s_0$. 
    In this last step we are simply borrowing negativity from our bound in Proposition \ref{prop:dmax}, which is enough to establish our bound when $u$ is sufficiently far from the state $u_L$. 
    We note that for the states satisfying $|u-u_L| \leq 2K_1 Ls_0$ we have the bound 
    \begin{align*}
        | s^* - s_0| =& \left| |u - u^+(u)| - |u_L-u_R| \right|\\ 
        \leq & \left| u - u_L +  u^+(u)-u_R \right|\\
        \leq &|u - u_L| + |u^+(u) - u_R|  \\
        \leq& |u - u_L| \leq 2K_1Ls_0
    \end{align*}
    which immediately implies
    $$s^*(u) \geq s_0(1-2LK_1). $$ 
    If we now allow $K_1$ to be sufficiently small, we find these states satisfy 
    $$ s^*(u) > K_1 s_0.$$
    Hence, we have established~\eqref{eq:Drh-up-to-tbar} for all states with $s^*(u) \leq K_1 s_0$.
\end{proof}

\bigskip
\paragraph{\bf Step 2.2: Large maximal shocks}

The remaining case for $u$ with large maximal shocks ($\overline t > s^*(u) \geq K_1 s_0$) can then be handled easily. 
\begin{lemma} \label{lemma:shock-small}
    For $C$ sufficiently large and $s_0$ sufficiently small,
        when $s^*(u) \geq K_1 s_0$, we find for all $u \in\Pi_{C,s_0}^*$ and $0 < s \leq \overline t$ that 
        $$ D_{\RH}(u,u_+) \lesssim -s_0(|u-u_L|^2 + |u_+-u_R|^2) $$
        where $u_+ = \S_u^1(s)$.
\end{lemma}
\begin{proof}
    To begin, we show the intermediate bound 
    \begin{equation}\label{eq:Drh-Dmax-around-s-star} 
        D_{\RH}(u,\S_u^1(s)) \lesssim D_{\max}(u) -s^*(u)(s-s^*(u))^2 
    \end{equation}
    which holds for $C$ sufficiently large, $s_0$ sufficiently small, and $s < \overline t$. 
    We first recall that by Lemmas \ref{lemma:leger-square}, \ref{lem:gkv_dist}, and \ref{gkv:lem_s-star} we have the estimate 
    \begin{equation} (s^*)^2 \lesssim \eta(u|u^+(u)) \lesssim -\tilde \eta(u) \lesssim s_0C^{-1}. \label{eq:s-star-control}\end{equation}
    We now take that $s_0$ is sufficiently small to ensure $s^* < \overline t$. 
    We now enjoy the following bound due to the convexity of $t \mapsto \eta(u|\S_u^1(t))$, 
    \begin{equation} \int_{s^*}^s \dot \sigma(t) (\tilde \eta(u) + \eta(u|\S_u^1(t)) )\,dt \lesssim -\int_{s^*}^s \tilde \eta(u) + \eta(u|\S_u^1(t))\,dt \leq \frac{(s-s^*)^2}{2} \begin{cases} 
            -\partial_t [\eta(u|\S_u^1(t))]_{t = s^*} & s \geq s^*\\
            \frac{\tilde \eta(u)}{s^*}   & s < s^*
        \end{cases} \label{eq:Drh-large-max}
    \end{equation}
    where $s \geq 0$ corresponds to the state $u_+ = \S_u^1(s)$ while $s^*$ is the argument corresponding to the maximal shock's right state, $u^+(u) = \S_u^1(s^*)$. 
    We note that in the case $s \geq s^*$ our coefficient satisfies
    $$-\partial_t [\eta(u|\S_u^1(t))]_{t = s^*} \lesssim -s^*$$ 
    for $s_0$ sufficiently small by equation~\eqref{eq:rel-shock-first-der} while in the second case $s < s^*$ we also find 
    $$\tilde \eta(u)/s^* \lesssim -s^*$$
    due to the bound~\eqref{eq:s-star-control} showing $(s^*)^2 \lesssim -\tilde\eta(u).$
    From the bounds~\eqref{eq:shock-lip} and~\eqref{eq:Drh-large-max} we arrive at the following,
    \begin{align*}
        D_{\RH}(u,\S_u^1(s)) =& D_{\RH}(u,u^+(u)) + \int_{s^*}^s \dot \sigma(t)(\tilde \eta(u) + \eta(u|\S_u^1(t))) \,dt \\
        \lesssim& D_{\RH}(u,u^+(u)) - \int_{s^*}^s \tilde \eta(u) + \eta(u|\S_u^1(t)) \,dt \\
        \lesssim& -s_0(|u-u_L|^2 + |u^+(u) - u_R|^2) - s^*|s^* - s|^2 \\
        \lesssim& -s_0(|u-u_L|^2 + |u^+(u) - u_R|^2) - s^*| u^+(u) - u_+ |^2
    \end{align*}

    Now equipped with~\eqref{eq:Drh-Dmax-around-s-star} the proof follows immediately,
    \begin{align*}
        D_{\RH}(u,\S_u^1(s)) \lesssim& -s_0(|u-u_L|^2 + |u^+(u) - u_R|^2) - s^*| u^+(u) - u^+ |^2 \\
        \leq& -K_1s_0(|u-u_L|^2 + |u^+(u) - u_R|^2) - K_1s_0| u^+(u) - u^+ |^2 \\
        \lesssim& -K_1 s_0(|u-u_L|^2 + |u^+ - u_R|^2) 
    \end{align*}
    for all $s < \overline t$. 
\end{proof}
We have now established Proposition \ref{prop:shock} for all shocks $(u, \S_u^1(s))$ with $u \in \Pi_{C,s_0}^*$ and $s < \overline t$.

\bigskip
\paragraph{\bf Step 2.3: Extending to the rest of the shock curve}

For states $u_+ = \S_u^1(s)$ with $s > \overline{t}$ we can no longer exploit the uniform convexity of the relative entropy $\eta(u|\S_u^1(s))$ from Lemma \ref{lemma:Drh-Dmax-bound}, but Assumption \ref{assum} \ref{assum:rel-ent-strengthens} ensures this quantity is increasing in $s$ and that $\dot \sigma < 0$. 
Hence, for all $u\in \Pi_{C,s_0}^*$ we retain an upper bound on the dissipation within this regime. 

\begin{lemma} \label{lemma:Drh-large-uni-bound}
    For $C$ sufficiently large and $s_0$ sufficiently small, for all states $u\in \Pi_{C,s_0}^*$ and 1-shocks $(u,u_+,\sigma)$ with $u_+ = \S_u^1(s) \in \Nu, s > \overline t$ we have
    $$ D_{\RH}(u,u_+) \leq D_\RH(u,\S_u^1(\overline t)) \lesssim -s_0 \overline t^2. $$
\end{lemma}

\begin{proof}
    By the formula \ref{eq:RH-CONT} and linearity of integration we have the expression,
    $$ D_{RH}(u, u_+) = D_\RH(u,\S_u^1(\overline t)) + \int_{\overline t}^s \dot\sigma(s)(\tilde \eta(u) + \eta(u|\S_u^1(t)))\,dt.$$
    We first note that, by Lemma \ref{lemma:Drh-Dmax-bound}, we have $\dot \sigma(s) < 0$. 
    Furthermore, by Lemmas \ref{gkv:lem_s-star}, \ref{lem:gkv_dist}, \ref{lem:gkv_eta}, and \ref{lemma:leger-square} we have, for $C$ sufficiently large and $s_0$ sufficiently small,
    $$ \tilde \eta(u) + \eta(u|\S_u^1(\overline t)) \gtrsim -s_0 C^{-1} + \overline t^2 \geq \frac{1}{2}\overline t^2.$$
    Combining these, we see 
    $$ \int_{\overline t}^s \dot\sigma(s)(\tilde \eta(u) + \eta(u|\S_u^1(t)))\,dt < 0, $$
    which establishes 
    $$ D_{RH}(u, u_+) \leq D_\RH(u,\S_u^1(\overline t)).$$
    Now, by Lemma \ref{lemma:Drh-up-to-tbar} we can further bound the right hand side by
    \begin{align*}
        D_\RH(u,\S_u^1(\overline t)) \lesssim& -s_0(|u- u_L|^2 + |\S_u^1(\overline t) - u_R|^2) \\
        \lesssim& -s_0|\S_u^1(\overline t) - u|^2 + s_0|u-u_R|^2 \\
        \lesssim& -s_0\overline t^2 + s_0(C^{-2} + s_0^2) \\
        \lesssim& -s_0\overline t^2
    \end{align*}
    where we have used Lemma \ref{lem:gkv_eta} and taken $C$ sufficiently large and $s_0$ sufficiently small ensure $C^{-2} + s_0^2 \lesssim \overline t^2$. 
\end{proof}

At this point we are prepared to prove Proposition \ref{prop:shock}.
\begin{proof}[Proof of Proposition \ref{prop:shock}]
    The proposition has been proven for all 1-shocks $(u,u_+,\sigma_\pm)$ with $u \in \Pi_{C,s_0}^*$, $u_+ = \S_u^1(s)$, $s \leq \overline t$ by Lemma \ref{lemma:shock-small}.
    For states with $s > \overline t$ simply note that
    $$ |u - u_L|^2 + |u_+ - u_R|^2 \lesssim \text{diam}(\Nu)^2,$$
    so by Lemma \ref{lemma:Drh-large-uni-bound} we have 
    $$ D_{\RH}(u,u_+) \lesssim -s_0\overline t^2 \lesssim -s_0\frac{\overline t^2}{\text{diam}(\Nu)^2}(|u - u_L|^2 + |u_+ - u_R|^2 )$$
    for states $u_+ = \S_u^1(s)$ with $s > \overline t$.
\end{proof}

\appendix 

\section{}
\subsection{Lemmas from \cite{MR4667839}}
In \Cref{sec:proof_prop:small_shock_diss} several results from \cite{MR4667839} are used to establish the dissipation inequality for small shocks.
Here we reproduce these results and make minor modifications, where needed. 
\subsubsection{Geometric results on the set $\Pi$}
Section \ref{sec:cont} vitally relies on controlling geometric quantities of $\Pi_{C,s_0}$ by the parameters $C$ and $s_0$. 
The results concerning this control are reproduced here. 
\begin{lemma}[Lemma 3.5 of \cite{MR4667839}]\label{lem:gkv_eta}
    For any $C > 0$ and $s_0 > 0$, $(\tilde \eta,\ \tilde q)$ are an entropy, entropy-flux pair for the system~\eqref{system}
        satisfying
    \begin{equation}\label{eqn_lemeta_desired1}
        \nabla \tilde \eta(u) = Cs_0(\nabla \eta(u) - \nabla \eta(u_L)) - (\nabla \eta(u_L) - \nabla \eta(u_R))\quad\text{and}\quad \nabla^2\tilde \eta(u) = Cs_0\nabla^2 \eta(u).
    \end{equation}
    For $C$ sufficiently large and $s_0$ sufficiently small the set $\Pi_{C,s_0}$ is strictly convex with nonempty interior, compactly contained within $B_{3\epsilon/2}(d)$, and $\text{diam}(\Pi_{C,s_0}) \sim C^{-1}$. 
    As a result of this bound and~\eqref{eqn_lemeta_desired1} we have for all $u\in \Pi_{C,s_0}^*$ 
    \begin{equation}\label{eqn_lemeta_desired2}
        |\nabla \tilde \eta(u)| \lesssim s_0\quad\text{and}\quad |\nabla^2\tilde \eta(u)| \lesssim  Cs_0.
    \end{equation}
\end{lemma}

\begin{lemma}[Lemma 3.6 of \cite{MR4667839}] \label{lem:gkv_nu}
    For $C$ sufficiently large, $s_0$ sufficiently small, and $\overline u \in \partial \Pi_{C,s_0}$ we have the lower bound 
    \begin{equation} \label{eqn_lemnu_desired}
        |\nabla \tilde \eta(\overline u)| \gtrsim s_0.
    \end{equation}
    In particular, this shows that the normal vector $\nu(\overline u) = \nabla \tilde \eta(\overline u)/|\nabla \tilde \eta(\overline u)|$ of $\Pi_{C,s_0}$ is well defined.
\end{lemma}

\begin{lemma}[Modification of Lemma 3.7 of \cite{MR4667839}]\label{lem:gkv_dist}
    For $C$ sufficiently large, $s_0$ sufficiently small, and $u\in \Pi^*_{C,s_0}$, we have
    \begin{equation}
    |\tilde\eta(u)| \lesssim s_0 d(u,\partial\Pi_{C,s_0}).
    \end{equation}
\end{lemma}
\begin{proof}
    The case of $u \in \Pi_{C,s_0}$ is proven in \cite{MR4667839}.
    For $u\in \Pi^*_{C,s_0} \setminus \Pi_{C,s_0}$ we define $\overline u = \overline u(u) \in \partial \Pi$ as the closest point in $\partial \Pi$ to $u$. 
    Considering the map $t \mapsto \tilde \eta(\overline u + t\nu(\overline u))$, where $\nu$ is the outward facing unit normal of $\Pi_{C,s_0}$, we have by the mean value theorem a $t \in (0,|u-\overline u|)$ such that
    \begin{align*}
        \tilde \eta(u) &=  \tilde \eta(u) - \tilde \eta(\overline u) \\
        &= \left[\left.\partial_s \tilde \eta(\overline u + s \nu (\overline u))\right|_{s = t}\right] |u-\overline u| \\
        &= \nabla \tilde \eta(\overline u + t\nu(\overline u)) \nu(\overline u)d(u,\partial \Pi_{C,s_0})  \\
        &\lesssim \sup_{u\in \Pi^*} |\nabla \tilde \eta(u)|d(u,\Pi) \lesssim s_0 d(u, \partial \Pi_{C,s_0})
    \end{align*}
    where the final line follows by equation~\eqref{eqn_lemeta_desired2}.
\end{proof}

\subsubsection{Characterization of the maximizer $u^*$ in Proposition \ref{prop:gkv-cont}}
\begin{lemma}[Lemmas 5.2 and 5.3 of \cite{MR4667839}] \label{lemma:gkv-max-char}
    Let $O \subset \Pi_{C,s_0}$ be open. 
    Then there are no critical points of $D_{cont}$ within $O$ and the extreme values of $D_{cont}$ are attained on $\partial O$. 
    Furthermore, if $\partial O$ is $C^1$ a local maximizer $\overline u \in \partial O$ is characterized by
    $$ \nabla \tilde \eta(\overline u) \parallel l^1(\overline u) \quad\text{ and }\quad r_1(\overline u)\text{ points outwards from $O$.}$$
\end{lemma}
 
For the following lemma, we define $u_l, u_0$ to be the first intersection points of $R_{u_L}^1(s)$ and $S_{u_L}^1(s)$ with $\partial \Pi_{C,s_0}$, respectively.
The state $u^*$ is as defined in Proposition \ref{prop:gkv-cont}. 
\begin{lemma}[Lemma 5.7 and Proposition 5.11 of \cite{MR4667839}] \label{lemma:gkv-dist-u0}
    For all $C$ sufficiently large and $s_0$ sufficiently small, 
    \begin{align*}
        |u_l - u_L| + |u_0 - u_L| &\sim s_0, \\
        |u_l - u_0| &\lesssim Cs_0^2, \text{ and} \\
        |u^* - u_0| &\lesssim s_0/C.
    \end{align*}
\end{lemma}

\subsubsection{Maximal shock results} \label{sec:app-u-plus}
These results can be found in Section 6 of \cite{MR4667839}. They are reproduced here for completeness. 
\begin{lemma}[Lemma 6.1 of \cite{MR4667839}] \label{gkv:lem_s-star}
    For any $u \in \Pi_{C,s_0}^*$ the map $s \mapsto D_{RH}(u,\S_u^1,\sigma_u^1(s))$ has a unique maximum $s^* = s^*(u) \in [0,\infty)$. 
    For $u \notin \Pi_{C,s_0}$ we have $s^* = 0$, while for $u \in \Pi_{C,s_0}$ we find $s^* > 0$ and satisfies 
    \begin{equation}
        \eta(u | \S^1_u(s^*)) = - \tilde\eta(u).
        \end{equation}
    Moreover, for $u^+(u) = \S_u^1(s^*(u))$ this gives us the bound $|u-u^+|^2 \lesssim s_0d(u,\partial \Pi_{C,s_0}) \lesssim s_0/C$. 
    So for $C$ sufficiently large and $s_0$ sufficiently small, we have $u^+(u) \in B_{2\epsilon}(d)$ for all $u \in \Pi^*_{C,s_0}$.  
\end{lemma}

\begin{lemma}[Lemma 6.14 of \cite{MR4667839}] \label{gkv_uplus-ident-2}
    For $C$ sufficiently large and $s_0$ sufficiently small we have the Lipschitz estimate 
    \begin{equation}
        | \nabla u^+(u) | \lesssim 1
    \end{equation}
    for all $u \in \Pi_{C,s_0}$. 
\end{lemma}

\subsection{Useful estimate for shock-shock overtaking interaction for isothermal gas}\label{app_A3}
We add some useful interaction estimate for shock-shock overtaking interaction for isothermal gas dynamics \eqref{isothermal} on $(\rho,w)$ with $\rho=1/\tau$. We adapt the method in \cite{MR3040678} for isentropic gas dynamics.


Given a base point $(\bar \rho,\bar w)$, we consider
the curves of points $(\rho,w)$ such that the Riemann problem with left state $(\bar \rho,\bar w)$
and right state $(\rho,w)$ yields a single 1 or 2 wave (rarefaction or entropic 
shock). We use the Rankine-Hugoniot condition and the Riemann invariant is constant on the left and right states of rarefaction wave in the same family.

Letting $b$ and $f$ denote the $\rho$-ratios $\rho_\text{r}/\rho_\text{l}$ across 
backward and forward waves, respectively, the parametrizations are given by:
\begin{align}
	&\text{1 waves:}\qquad
	{\ba{V}}(b;\bar\rho,\bar w) = \left(\!\!\begin{array}{c}
		b\bar \rho \\
		\bar u -{\ba{\Omega}}(b)\bar \rho\\
	\end{array} \!\!\right)
	\qquad\left\{\begin{array}{ll}
	\ba R:\quad 0<b<1 \\
	\ba S:\quad b>1
	\end{array}\right.\label{bkwd_wave}\\
	&\text{2 waves:}\,\,\,\qquad
	{\fa{V}}(f;\bar \rho,\bar w) = \left(\!\!\begin{array}{c}
		f\bar \rho \\
		\bar u +{\fa{\Omega}}(f)\bar \rho\\
	\end{array} \!\!\right)
	\qquad\left\{\begin{array}{ll}
	\fa R:\quad f>1 \\
	\fa S:\quad 0<f<1,
	\end{array}\right.\label{frwd_wave}
\end{align}
where the auxiliary functions $\ba\Omega$ and $\fa\Omega$ are given by
\[{\ba{\Omega}}(x) = \left\{\begin{array}{ll}
	 (x-1) \quad & 0<x<1\\\\
	\sqrt{x}\quad & x>1
\end{array}\right\}\] 
\[{\fa{\Omega}}(x) = \left\{\begin{array}{ll}
	-\sqrt{x} \quad & 0<x<1\\\\
	 (x-1) \quad  & x>1.
\end{array}\right\}\] 
Note that the auxiliary functions $\ba\Omega$ and $\fa\Omega$ are both strictly increasing and satisfy
\beq\label{phi_reln}
	\fa \Omega(x)=-x{\ba{\Omega}}\big(\textstyle\frac{1}{x}\big)\,.
\eeq
Next, we consider the overtaking interaction between two 1-shocks.
Let the state on the far left be $(\bar\rho,\bar w)$ and let the $\rho$-ratios of the 
incoming  backward waves be $b$ (leftmost) and $\hat b$ (rightmost).
As above let $B$ and $F$ denote the $\rho$-ratios of the outgoing $1$ and $2$
waves, respectively. Traversing the waves before and after interaction yields
\begin{align}
	BF &= b \bar b\label{overtaking1}\\
	\ba\phi(B)-B\fa\phi(F) &=\ba\phi(b)+b\ba\phi(\bar b).\label{overtaking2}
\end{align}
Using \eq{phi_reln} we get that the $\rho$-ratio $B$ solves the equation
\beq\label{overtaking_B}
	\mathcal{G}(B;b,\bar b):=\ba\phi(B)+b{\bar b}\ba\phi\left(\frac{B}{b{\bar b}}\right)-\ba\phi(b)-b\ba\phi({\bar b})=0.
\eeq
It is easy to check that for any fixed $b$, $\bar b$ larger than $1$, $\mathcal{G}(B;b,{\bar b})$ is increasing on $B$, with $\mathcal{G}(1;b,{\bar b})<0$ and $\mathcal{G}(\infty;b,{\bar b})>0$. Hence there is one unique solution $B>1$ of \eqref{overtaking_B}, which will determine a unique $F$ by \eqref{overtaking1}. It is easy to show that $F>1,$ i.e. the outgoing 2-wave is a rarefaction wave, see \cite{MR1301779,MR3450063}.

Now we prove that $B>b$ and $B>{\bar b}$.

In fact, by \eqref{phi_reln},
\[
\mathcal{G}(b;b,{\bar b}):=\ba\phi(b)+b{\bar b}\ba\phi\left(\frac{1}{{\bar b}}\right)-\ba\phi(b)-b\ba\phi({\bar b})=-b\ba\phi(y{\bar b})-\ba\phi(b)<0
\]
and
\[
\mathcal{G}({\bar b};b,{\bar b}):=\ba\phi({\bar b})+b{\bar b}\ba\phi\left(\frac{1}{b}\right)-\ba\phi(b)-b\ba\phi({\bar b})=(1-b)\ba\phi({\bar b})-{\bar b}\ba\phi(b)-\ba\phi(b)<0,
\]
Hence by the montonicity of $\mathcal{G}(B;b,{\bar b})$ on $B$ and $\mathcal{G}(\infty;b,{\bar b})>0$, we know that 
\[
B>\max(b,{\bar b}).
\]

By \eqref{si_iso}, we know that, the strength of outgoing 1-shock (mensured by $|\sigma|$) is larger than the strength of each
incoming shock.  

Using a similar proof, we can show this result for the interaction between two 2-shocks.

\bibliographystyle{apalike}
\bibliography{references-2}

\begin{center} 
\includegraphics[width=.4\linewidth]{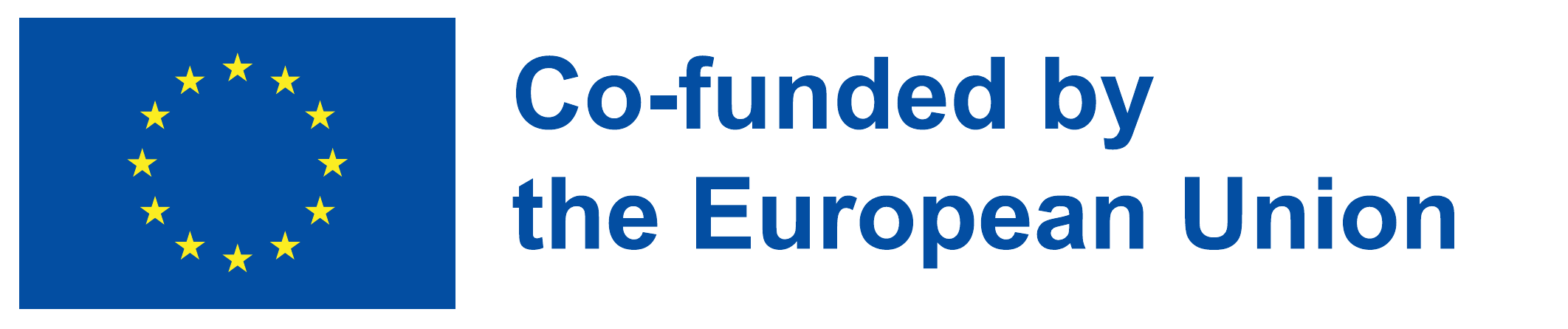}
\end{center}

\end{document}